\documentclass[a4paper,10pt,reqno]{amsart}
\usepackage{a4wide}
\usepackage{amsmath,amsfonts,amssymb,amsthm,mathabx}
\usepackage{dsfont} % for the characteristic function "1" in mathmode
\usepackage{tgpagella}
\usepackage{tikz}
\usetikzlibrary{calc,shapes}
\usepackage[colorlinks=true,citecolor=forestgreen]{hyperref}
\usepackage{float} % float placement
\usepackage{blkarray}  % pour faire des matrices fancy
\usepackage{booktabs}  % pour mettre des lignes dans les matrices fancy
\usepackage{scalerel}  % for special symbols
\usepackage{tipa}      % for typographic symbols
\usepackage{multirow}  % for multirows in table

\theoremstyle{plain}
\newtheorem{theorem}{Theorem}[section]
\newtheorem*{theorem*}{Theorem}
\newtheorem{lemma}[theorem]{Lemma}
\newtheorem{corollary}[theorem]{Corollary}
\newtheorem{proposition}[theorem]{Proposition}

\newtheorem{mainthm}{Theorem}
\newtheorem{maincor}{Corollary}

\theoremstyle{definition}
\newtheorem{definition}[theorem]{Definition}
\newtheorem{remark}[theorem]{Remark}
\newtheorem{example}[theorem]{Example}
\newtheorem{problem}[theorem]{Problem}

\definecolor{darkblue}{rgb}{0,0,0.7} % darkblue color
\definecolor{forestgreen}{rgb}{0.07,0.35,0.10} % forestgreen
\newcommand{\darkblue}{\color{darkblue}} % darkblue command
\newcommand{\defn}[1]{\emph{\darkblue #1}} % emphasis of a definition

\newcommand{\A}{\mathbf{A}}
\newcommand{\B}{\mathbf{B}}
\newcommand{\N}{\mathbb{N}}
\newcommand{\R}{\mathbb{R}}
\newcommand{\Red}{\mathcal{R}}
\newcommand{\abel}{\alpha}
\newcommand{\C}{\mathcal{C}}
\newcommand{\F}{\mathcal{F}}
\newcommand{\G}{\mathcal{G}}
\newcommand{\M}{\mathcal{M}}
\newcommand{\PP}{\mathcal{P}}
\newcommand{\Sym}{\mathfrak{S}}
\newcommand{\T}{\mathcal{T}}
\newcommand{\V}{\mathcal{V}}
\newcommand{\ZZ}{\mathcal{Z}}
\newcommand{\ZZbij}{\mathfrak{z}}

\newcommand\wo{w_\circ} % longest element

\DeclareMathOperator{\braid}{braid}
\DeclareMathOperator{\cone}{cone}
\DeclareMathOperator{\comm}{comm}
\DeclareMathOperator{\even}{even}
\DeclareMathOperator{\Gale}{Gale}
\DeclareMathOperator{\Id}{Id}
\DeclareMathOperator{\inv}{inv}
\DeclareMathOperator{\odd}{odd}
\DeclareMathOperator{\rev}{rev}
\DeclareMathOperator{\sign}{sign}
\DeclareMathOperator{\std}{std}
\DeclareMathOperator{\Vand}{Vander}

\DeclareMathAlphabet{\mathdutchcal}{U}{dutchcal}{m}{n}
\newcommand{\Sh}{\mathdutchcal{S}}
\newcommand{\sh}{\mathdutchcal{s}}

\makeatletter
\newcommand{\punc}{\scalerel*{\@puncpic}{\ensuremath{\Sigma}}}% Or \Sigma, or any symbol you want to scale to.
\newcommand{\@puncpic}{%
\setlength{\unitlength}{0.28cm}% reduce this to increase thickness of line
\begin{picture}(1,1.5)%
\thicklines%
\put(0,0){\line(2,3){1}}%
\put(1,1.5){\line(-1,0){1}}%
\put(0,1.5){\line(2,-3){1}}%
\put(1,0){\line(-1,0){1}}%
\end{picture}%
}
\makeatother

\title[Combinatorial foundations for geometric realizations of subword complexes]{Combinatorial foundations for geometric realizations\\ of subword complexes of Coxeter groups}
\date{\today}

\author[J.-P.~Labb\'e]{Jean-Philippe Labb\'e}
\address[J.-P.~Labb\'e]{Institut f\"ur Mathematik, Freie Universit\"at Berlin, Arnimallee 2, 14195 Berlin, Germany}
\email{labbe@math.fu-berlin.de}
\urladdr{http://page.mi.fu-berlin.de/labbe}
\thanks{This research was supported by the DFG Collaborative Research Center TRR~109 ``Discretization in Geometry and Dynamics''.}


\begin{document}

\begin{abstract}
Multitriangulations, and more generally subword complexes, yield a large family of simplicial complexes that are homeomorphic to spheres.
Until now, all attempts to prove or disprove that they can be realized as convex polytopes faced major obstacles.
In this article, we lay out the foundations of a framework---built upon notions from algebraic combinatorics and discrete geometry---that allows a deeper understanding of geometric realizations of subword complexes of Coxeter groups.
Namely, we describe explicitly a family of chirotopes that encapsulate the necessary information to obtain geometric realizations of subword complexes.
Further, we show that the space of geometric realizations of this family covers that of subword complexes, making this combinatorially defined family into a natural object to study.

The family of chirotopes is described through certain \emph{parameter matrices}. 
That is, given a finite Coxeter group, we present matrices where certain minors have prescribed signs.
Parameter matrices are \emph{universal}: The existence of these matrices combined with conditions in terms of Schur functions is equivalent to the realizability of all subword complexes of this Coxeter group as chirotopes.
Finally, parameter matrices provide extensions of combinatorial identities; for instance, the Vandermonde determinant and the dual Cauchy identity are recovered through suitable choices of parameters.
\end{abstract}

\maketitle

\vspace{-0.75cm}
{
\hypersetup{linkcolor=black}
\tableofcontents
}

\section{Introduction}

Let $d\geq 1$ and $\Delta$ be a simplicial complex homeomorphic to a $(d-1)$-dimensional sphere.
\begin{center}
\emph{Is there a $d$-dimensional simplicial convex polytope $P$ whose face lattice is isomorphic to that of $\Delta$?}
\end{center}
In the affirmative case, the simplicial sphere $\Delta$ is often called \defn{polytopal}.
Steinitz showed that every $2$-dimensional simplicial sphere is polytopal~\cite{steinitz_vorlesungen_1976},\cite[Section~13]{grunbaum_convex_2003}.
In higher dimensions, this question is part of ``Steinitz' Problem'' asking to determine polytopal spheres among all simplicial spheres, see~\cite[Introduction]{kalai_many_1988}\cite[Chapter~3.4]{ewald_combinatorial_1996}.
Though a lot of work has been done towards brute force enumeration of simplicial and polytopal spheres in low dimensions, and spheres with fews vertices, the determination of the polytopality of spheres is famously known to be fraught with pitfalls, see~\cite{firsching_realizability_2017,firsching_complete_2018} and~\cite{brinkmann_fvector_2016,brinkmann_small_2016} to get an overview of the most recent results and discussions on the delicacy of the enumeration.
The determination of the polytopality of a simplicial sphere is known to be an NP-hard problem \cite{mnev_universality_1988,richtergebert_realization_1995,richtergebert_realization_1996}, making progress in this direction continually limited.
Further, techniques aimed at finding combinatorial local conditions are bound to fail in general~\cite{sturmfels_boundary_1987}.
Kalai's ``squeezed spheres'' \cite{kalai_many_1988} and further constructions by Pfeifle and Ziegler~\cite{pfeifle_many_2004} and by Nevo, Santos, and Wilson \cite{nevo_many_2016} show that, for $d\geq 4$, as the number of vertices increases, \emph{most} simplicial $(d-1)$-spheres are \emph{not} polytopal.
Further, polytopal simplicial spheres are also \emph{rare} among geodesic simplicial sphere, i.e. simplicial spheres with a realization on the unit sphere where edges are geodesic arcs~\cite{nevo_many_2016}.
Geodesic simplicial spheres correspond to complete simplicial fans in $\R^{d}$; one obtains the geodesic arcs by intersecting the fan with the unit sphere.

Facing this situation, a potential Ansatz to study Steinitz' problem consists in finding novel flexible polytopal constructions or combinatorial obstructions to polytopality.
One famous example of polytope with many constructions is that of the associahedron~\cite{stasheff_homotopy_1963}.
Notable constructions include the fiber polytope realization~\cite[Chapter~7]{gelfand_discriminants_1994}, the cluster algebra approach~\cite{chapoton_polytopal_2002}, and the simple combinatorial construction using planar binary trees~\cite{loday_realization_2004}.
A myriad of descriptions and variations are possible~\cite{tamari_2012}, making a exhaustive survey a challenging task. 
The fact that the associahedron is related to so many areas of mathematics opens the door to approaches to Steinitz' problem from areas outside of discrete geometry.

Among the several possible extensions of the associahedron, multi-triangulations of a convex $n$-gon offer a particularly broad generalization of the underlying boundary complex of its polar~\cite{pilaud_multitriangulations_2012}.
This extension offers the opportunity to discover new approaches to Steinitz' Problem.
Indeed, given an integer $k\geq 1$, the simplicial complex whose facets are maximal ``$k$-crossing-free'' sets of diagonals of a convex polygon (i.e. multi-triangulations) is conjectured to be the boundary of a convex polytope whose polar would generalize the associahedron.
This conjecture first appeared in the Oberwolfach Book of Abstract, handwritten by Jonsson in 2003~\cite{jonsson_abstract_2003}, which did not subsequently appear in the printed MFO Report~\cite{jonsson_generalized_2003}.
Currently, the only known non-classical polytopal construction of multi-associahedron is for the $2$-triangulations of the $8$-gon~\cite{bokowski_symmetric_2009,ceballos_associahedra_2012,bergeron_fan_2015}.
Further, certain cases are known to be realizable as geodesic spheres~\cite{bergeron_fan_2015,manneville_fan_2018}.

Furthermore, the simplicial complex of multi-triangulations turns out to be an example of subword complexes, a broader family of simplicial complexes related to the Bruhat order of Coxeter groups \cite{knutson_subword_2004,knutson_groebner_2005}.
Introduced in the context of Gr\"obner geometry of Schubert varieties, these simplicial complexes are at the crossroad of a variety of objects via their intrinsic relation with reduced words: cluster algebras~\cite{ceballos_subword_2014}, toric geometry~\cite{escobar_brick_2016}, root polytopes~\cite{escobar_subword_2018}, Hopf algebras~\cite{bergeron_hopf_2017}, totally non-negative matrices \cite{davis_fibers_2019}, among others.

In~\cite{bergeron_fan_2015}, taking advantage of the combinatorics of reduced words, the notions of \emph{sign function} and \emph{signature matrices} are used to lay down salient necessary conditions for the polytopality of subword complexes.
The first step consists in showing the existence of a certain sign function on reduced words, which is then used to formulate sign conditions on minors of matrices to obtain signature matrices.
The existence of such matrices constitutes a step towards realizations of chirotopes, which are abstractions of geometric realizations of subword complexes.
Then, a combinatorial construction of signature matrices is derived and it is possible to prove that they lead to complete simplicial fans for subword complexes of type~$A_3$ and for certain cases in type~$A_4$.
In spite of this progress, the reason \emph{why} the construction works is still rather mysterious.
Optimistically, a notion from the aforementioned areas may be key to determine the polytopality of subword complexes and therefore determine if multi-associahedra exist as convex geometric entities.

However, such a notion leading to a general construction remains to be found, and appears to be difficult to find.
Despite this fact, in this article we do succeed in deriving universal families of chirotopes whose realizability implies the realizability of subword complexes.
Here, we aim to make the point that the study of these new families of chirotopes provides the natural framework for the study of geometric realizations of subword complexes.
To do so, we expose some of the crucial details of the construction in \cite{bergeron_fan_2015} and give them a structural combinatorial description.
In particular, this reveals precisely which combinatorial properties of reduced words are relevant, and that Schur functions lay at the center of geometric realizations of subword complexes.
More specifically, we introduce \emph{parameter tensors} and show their universality for geometric realizations of subword complexes: 
Every realization of any spherical subword complex as a polytope or a simplicial fan delivers a realization of a partial chirotope, called parameter tensor, that depends essentially on data derived from the Coxeter group, and little on the words used to define the subword complexes (see Theorem~\ref{thm:C_univ}).
\begin{theorem*}
For each finite Coxeter group~$W$, there exists an explicit family of chirotopes $\mathfrak{X}$ such that every realization of a subword complex as a chirotope $\chi$ yields a realization of a chirotope in~$\mathfrak{X}$.
\end{theorem*}
This theorem has many ramifications.
First, the combinatorial description of $\mathfrak{X}$ is based upon an extension of a classical notion related to permutations and their parity (being \emph{even} or \emph{odd}).
To this end, we introduce a set of sign functions on words (Definition~\ref{def:sign_fct}).
These sign functions are used to prescribe the signs of minors of the involved parameter tensors (Theorem~\ref{thm:A_Sbraid} and Corollaries~\ref{cor:B_commutations} and \ref{cor:D_sign_of_model}).
Remarkably, the proof of the theorem above uncovers how a mild extension of Schur functions play a central role in the description of the realization space of parameter tensors (Theorem~\ref{thm:B_det}), and at the same time leads to the dual Cauchy identity (Example~\ref{ex:dual_cauchy}).
Finally, the chirotopes in the family $\mathfrak{X}$ are structurally very similar and empirical evidence further suggests that the family $\mathfrak{X}$ may consist of only 1 chirotope, and that it dictates much of the chirotope~$\chi$.

The study of parameter tensors opens the door to a in-depth investigation of the space of geometric realizations of subword complexes that relies on Schur functions and novel combinatorial properties of reduced words.
Completed with the universality result, this framework provides a structured strategy to determine the polytopality of subword complexes and finally establish the difficulty of this problem through its intrication with the ``halving-line problem'' and the ``shortest common supersequence problem'', see Section~\ref{sec:parameter_sign_mat}.

{\bf Outline.} 
We present the necessary notions from linear algebra, algebraic combinatorics and discrete geometry in Section~\ref{sec:prelim}.
In Section~\ref{sec:s_sign}, we present a \emph{theory of sign functions on words} that unifies the usual sign of permutations and the sign function presented in~\cite{bergeron_fan_2015} and prove Theorem~\ref{thm:A_Sbraid}.
In Section~\ref{sec:model_mat}, we define \emph{parameter tensors} and \emph{model matrices} as a tool to factorize the determinant of ``partial alternant matrices'' and thereby prove Theorem~\ref{thm:B_det}.
Finally, in Section~\ref{sec:parameter_sign_mat}, we combine both tools and present the \emph{universality of parameter tensors} in Theorem~\ref{thm:C_univ}.

{\bf Acknowledgements.}
The author would like to express his gratitude to
Federico Castillo,
Cesar Ceballos,
Joseph Doolittle,
Gil Kalai,
Eran Nevo,
Arnau Padrol,
Vincent Pilaud, 
Vic Reiner,
Christophe Reutenauer,
Francisco Santos, 
Rainer Sinn,
and G\"unter M. Ziegler
for several important discussions leading to the results in this article.
The author is grateful to Adriano Garcia, Erza Miller and Richard P. Stanley for enlightening discussions that took place during the ``Conference Algebra and Combinatorics at LaCIM'' for the 50$^{\text{th}}$ anniversary of the {\it Centre de Recherches Math\'ematiques} in September 2018.
Further, the author is thankful to the {\tt Sagemath} community, whose work allowed to create the experimental tools leading to the present results.

\section{Preliminaries}
\label{sec:prelim}
We adopt the following conventions: $\N=\{0,1,2,\dots\}$, $n\in\mathbb{N}\setminus\!\{0\}$, $[n]:=\{1,2,\dots,n\}$, and $\binom{[n]}{k}$ denotes the collection of $k$-elements subsets of $[n]$.
The cardinality of a set $S$ is denoted by~$\# S$.
The symmetric group on~$n$ objects is denoted by $\Sym_n$ or $\Sym_{S}$, where $S$ is a set of cardinality~$n$, and the multiplicative group $(\{\pm 1\},\times)$ is denoted by $\mathbb{Z}_2$.
Vectors and linear functions are denoted using bold letters such as~$\mathbf{e,g,v,x,y}\dots$, and scalars are denoted using normal script such as~$x$.
Tensors are denoted using capitalized calligraphic letters such as $\mathcal{M,T,X}\dots$, and we use normal font such as $M,T,X$ if they should be thought of in some flattened form.
Exponents on tensors designate vector spaces while indices designate dual vector spaces.
For ease of reading, some known combinatorial objects are also denoted using uppercase calligraphic letters although they are not tensors; hopefully the context should help avoid any confusion.
Labeled sets of vectors are denoted by uppercase boldface letters such as $\mathbf{A}$.

\subsection{Multilinear algebra}
Let $d\geq 1$ and $V_d$ be a $d$-dimensional real vector space and denote its dual space by $V_d^*$.
As usual, vectors in $V_d$ are represented as column vectors, while linear functions~$V^*_d$ are represented as row vectors.
We denote the transpose of vectors and of linear functions by $\cdot^{\top}$.
We use Einstein summation convention for tensors, described as follows.
Given a basis $\{\mathbf{e}^1,\dots,\mathbf{e}^d\}$ of $V_d$ and a basis $\{\mathbf{f}_1,\dots,\mathbf{f}_d\}$ of $V^*_d$,  a $(d\times d)$-matrix $M^i{}_{j}=(m_{i,j})=(m_{j}^{i})$ represents the tensor
\[
\M^i{}_{j}:=\sum_{i=1}^{d}\sum_{j=1}^{d}m_j^i \mathbf{e}^i\otimes \mathbf{f}_j\in V_d\otimes V^*_d.
\]
Given a tensor $\T\in (V_d)^a\otimes (V_d^*)^b$, a \emph{row} of $\T$ is given by the restriction of $\T$ to a basis element of $(V_d)^a$, i.e. a row
is indexed by a tuple $(i_1,\dots,i_a)\in [d]^a$.
Similarly, \emph{columns} of $\T$ are obtained by restricting $\T$ to basis elements of $(V_d^*)^b$, and are labeled by tuples in $[d]^b$.
We view the product of $(d_1\times d_2)$-matrices with $(d_2 \times d_3)$-matrices using tensors via the following linear map:
\begin{align}
\left(V_{d_1} \otimes V_{d_2}^*\right) \times \left(V_{d_2} \otimes V_{d_3}^*\right) & \rightarrow V_{d_1} \otimes V_{d_3}^* \nonumber\\
((\mathbf{x}\otimes \mathbf{f}),(\mathbf{y}\otimes \mathbf{g})) & \mapsto \mathbf{f}(\mathbf{y})\cdot (\mathbf{x}\otimes \mathbf{g}) \label{eq:prodtensor}.
\end{align}
More generally, given a tensor $\T^{i}{}_{j}\in V_{d_1}\otimes V_{d_2}^*$ and a tensor $\mathcal{U}^j{}_{k}\in V_{d_2}\otimes V_{d_3}^*$, we write the tensor contraction as
\[
\mathcal{V}^{i}{}_{k}:=\T^{i}{}_{j}\cdot\mathcal{U}^{j}{}_{k},
\]
using the rule given in Equation~\eqref{eq:prodtensor}.
Contraction of higher rank tensors is defined similarly, by matching the appropriate pairs of indices.

\subsection{Vandermonde matrix and Binet--Cauchy Formula}
\label{ssec:vandermonde}

The Vandermonde matrix of size~$d$ is
\[
\Vand(d):=\left(\begin{array}{llll}
1 & 1 & \cdots & 1 \\
x_1 & x_2 & \cdots & x_d \\
x_1^2 & x_2^2 & \cdots & x_d^2 \\
\vdots & \vdots & \ddots & \vdots \\
x_1^{d-1} & x_2^{d-1} & \cdots & x_d^{d-1}
\end{array}\right)=\sum_{i=1}^{d}\sum_{j=1}^d x_j^{i-1}\mathbf{e}^i\otimes \mathbf{f}_j=(x_j^{i-1})_{(i,j)\in[d]\times[d]},
\]
and its determinant is
\begin{align}
\label{eq:vandermonde}
\det \Vand(d)=\prod_{1\leq i<j\leq d} (x_j-x_i)=\sum_{\pi\in\Sym_d}\sign(\pi)x_1^{\pi(1)-1}\cdots x_d^{\pi(d)-1}.
\end{align}
The Vandermonde matrix is also obtained as the product of two rectangular
matrices as follows.
Let $W_d:=\mathcal{W}^i{}_{j,k}:=\bigoplus_{i=1}^{d}\Id_d$ be the augmentation of $d$ identity matrices by concatenating them columnwise.
This matrix can be rewritten as a tensor in $V_d\otimes V^*_d \otimes V^*_d$ as $\sum_{i=1}^d \left(\mathbf{e}^i\otimes \left(\sum_{j=1}^{d}\mathbf{f}_j\right) \otimes \mathbf{f}_i\right)$.
Further, we define the tensor $\mathcal{X}^{k,j}{}_{l}\in V_d\otimes V_d \otimes V^*_d$ as
\[
\mathcal{X}^{k,j}{}_{l}:=\sum_{j=1}^d \left(\sum_{k=1}^{d}x_j^{k-1}\right) \mathbf{e}^k\otimes \mathbf{e}^j \otimes \mathbf{f}_j.
\]
By flattening it, the tensor $\mathcal{X}^{k,j}{}_{l}$ can be written as a matrix as 

\[
\resizebox{!}{3cm}{$
X_d=
\left(\begin{array}{cccc}
\begin{array}{c}
1 \\
x_1 \\
x_1^2 \\
\vdots \\
x_1^{d-1}
\end{array} & 0 & \cdots & 0 \\
0 & \begin{array}{c}
1 \\
x_2 \\
x_2^2 \\
\vdots \\
x_2^{d-1}
\end{array} & \cdots & 0 \\
\vdots &  & \ddots & \vdots \\
0 & \cdots & 0 & \begin{array}{c}
1 \\
x_N \\
x_N^2 \\
\vdots \\
x_N^{d-1}
\end{array}
\end{array}\right).
$}
\]
From the definitions of $W_d$ and $X_d$, and the properties of product of tensors, we get 
\[
\Vand(d)=W_dX_d=\mathcal{W}^i{}_{j,k}\cdot\mathcal{X}^{k,j}{}_{l}.
\]
Indeed, the data is split in order for the product to give back the Vandermonde matrix.

In Section~\ref{ssec:bin_cau_model}, we look at the determinants of variations of the Vandermonde matrices given in this product form.
In order to get hold of the signs of determinants of these variations, we use the Binet--Cauchy formula.
If $M$ is a matrix, we denote by $[M]_Z$ the submatrix of $M$ formed by the rows (or columns) indexed by the set $Z$ ordered as in $M$.
When the order on $Z$ is not determined by $M$, we write the indices in a tuple.

\begin{lemma}[{Binet--Cauchy formula, see \cite[Section~10.5, p.377]{shafarevich_linear_2013} or \cite[Chapter~31]{aigner_proofs_2014}}]
\label{lem:binetcauchy}
If~$P$ is an $(r\times s)$-matrix, $Q$ an $(s\times r)$-matrix, and $r\leq s$, then
\begin{equation}
\det(PQ)=\sum_{Z\in\binom{[s]}{r}}(\det [P]_Z)(\det [Q]_Z),\label{eq:cauchy_binet}
\end{equation}
where $[P]_Z$ is the $(r\times r)$-submatrix of $P$ with column-set $Z$, and $[Q]_Z$ the $(r\times r)$-submatrix of $Q$ with the corresponding row-set $Z$. 
\end{lemma}

In the Vandermonde case, since $W_d$ and $X_d$ enjoy a simple structure, it is easy to reobtain the formula for its determinant:
\[
\det \Vand(d) = \det (W_dX_d) = \sum_{Z\in\binom{[s]}{r}}(\det [W_d]_Z)(\det [X_d]_Z).
\]

\noindent
Let $Z\in\binom{[s]}{r}$.

\medskip
\noindent
$\sqbullet$ The matrix $[W_d]_Z$ either has two equal columns or is a permutation matrix.
Therefore, the determinant $\det[W_d]_Z$ is $0$, $1$, or $-1$. 
It is zero if at least two columns are equal, and equal to the sign of the permutation associated to the permutation matrix otherwise. 

\medskip
\noindent
$\sqbullet$ The determinant of $[X_d]_Z$ is zero if two rows use the same variable $x_i$.
Otherwise, the determinant is a monomial with variables in $\{x_1,\dots,x_d\}$.

\medskip
\noindent
Combining the two conditions for the determinants to be non-zero, we get back the
fact that the determinant of the Vandermonde matrix is the sum of the signed
monomials in exactly $d-1$ variables which have all distinct powers in $\{1,\dots,d-1\}$ as in Equation~\eqref{eq:vandermonde}.

\subsection{Partial Schur functions}
\label{ssec:partial_schur}

It is worth noting that altering $W_d$ and $X_d$ defined above in particular ways
lead to generalizations of the Vandermonde matrix.
For example, Schur functions can be defined this way.

\begin{definition}[{Schur functions via Vandermonde matrices \cite[Section~1.3]{macdonald_symmetric_2015}, \cite[Chapter~7.15]{stanley_enumerative_1999}, and \cite[Section~4.6]{sagan_symmetric_2001}}]
\label{def:schur}
Let $\lambda=(\lambda_1,\lambda_2,\dots,\lambda_d)$ be a partition with $\lambda_1\geq \lambda_2\geq \cdots \geq \lambda_d\geq 0$, and $J$ be an ordered set of indices of cardinality $d$.
The \defn{Schur function}~$\sh_{\lambda,J}$ in the variables $\{x_j : j\in J\}$ is the quotient
\[
\sh_{\lambda,J}:=\frac{\det (x_{j}^{i-1+\lambda_{d-i+1}})_{(i,j)\in[d]\times J}}{\det \Vand_J(d)},
\]
where $\Vand_J(d)$ is the Vandermonde matrix $\Vand(d)$ with variables indexed by $J$.
When $J=[n]$ for some $n\geq1$ we omit the subscript $J$ and simply write $\sh_\lambda$.
\end{definition}

\begin{example}
\label{ex:partition}
\hfill
\begin{enumerate}
\item
Let $d=3$ and consider the partition $(4,1,0)$. 
The Schur function~$\sh_{(4,1,0)}$ is the determinant of the matrix below divided by the Vandermonde determinant $\det \Vand(3)$:

\begin{align*}
\sh_{(4,1,0)} & =\frac{1}{\det \Vand(3)}
\left|\begin{array}{rrr}
x_1^{0+0} & x_2^{0+0} & x_3^{0+0} \\
x_1^{1+1} & x_2^{1+1} & x_3^{1+1} \\
x_1^{2+4} & x_2^{2+4} & x_3^{2+4}
\end{array}\right| \\
& = {\left(x_{1}^{2} + x_{2}^{2} + x_{3}^{2}\right)} {\left(x_{1} + x_{2}\right)} {\left(x_{1} + x_{3}\right)} {\left(x_{2} + x_{3}\right)}.
\end{align*}
The product of the matrices \[
\resizebox{0.9\hsize}{!}{$
W_{(4,1,0)} = \bigoplus_{i=1}^3
\left(\begin{array}{rrrrrrr}
1 & 0 & 0 & 0 & 0 & 0 & 0 \\
0 & 0 & 1 & 0 & 0 & 0 & 0 \\
0 & 0 & 0 & 0 & 0 & 0 & 1
\end{array}\right)
\text{ and }
X_{(4,1,0)}=\sum_{j=1}^3 \left(\sum_{k=1}^{3}x_j^{k-1}\right) \mathbf{e}^k\otimes \mathbf{e}^j \otimes \mathbf{f}_j$}
\]
gives the above matrix. 
\item 
If we set each part in the partition $\lambda$ to $0$, we get back the Vandermonde matrix, which corresponds to the degree $0$ symmetric polynomial $1$.
\end{enumerate}
\end{example}

The following definition is then quite natural.
It is directly used to state Theorem~\ref{thm:B_det} on page~\pageref{thm:B_det}.

\begin{definition}[Partial Schur functions]
Let $m\geq 1$, $P=(p_i)_{i=1}^n$ be an ordered set partition of~$[m]$ where parts may be empty, and $\Lambda=(\lambda^i)_{i=1}^n$ be a sequence of partitions such that the number of parts of $\lambda^i$ is the cardinality of $p_i$.
The \defn{partial Schur function} with respect to $P$ and $\Lambda$ is
\[
\Sh_{\Lambda,P}:=\prod_{i=1}^n\sh_{\lambda^i,p_i}.
\]
This function is symmetric with respect to the action of $\prod_{i=1}^n\Sym_{p_i}$ as a subgroup of $\Sym_m$.
\end{definition}

\begin{example}
\label{ex:partial_schur}
Let $m=4$, $P=(\{1,3\},\{2,4\})$, $\Lambda_1=((0,0),(1,0))$, $\Lambda_2=((1,0),(0,0))$, and $\Lambda_3~=~((3,1),(2,0))$.
The partial Schur functions with respect to $P$ and the sequences of partitions are
\[
\begin{split}
\Sh_{\Lambda_1,P} & = \sh_{(0,0),\{1,3\}}\cdot \sh_{(1,0),\{2,4\}} = 1 \cdot (x_2+x_4), \\
\Sh_{\Lambda_2,P} & = \sh_{(1,0),\{1,3\}}\cdot \sh_{(0,0),\{2,4\}} = (x_1+x_3) \cdot 1, \\
\Sh_{\Lambda_3,P} & = \sh_{(3,1),\{1,3\}}\cdot \sh_{(2,0),\{2,4\}} = x_1x_3(x_1^2+x_1x_3+x_3^2) \cdot (x_2^2+x_2x_4+x_4^2). \\
\end{split}
\]
\end{example}

\subsection{Combinatorics on words}

For basic notions on combinatorics on words and monoids, we refer to the books \cite[Chapter~1]{lothaire_combinatorics_1997} and
\cite[Chapter~1]{diekert_combinatorics_1990}.
Let $S=\{s_1,\dots,s_n\}$ be a finite alphabet of \emph{letters} or \emph{generators}
equipped with the lexicographic order $s_1<s_2<\dots<s_n$.
Let $S^*:=\bigoplus_{i\in\mathbb{N}}S^i$ be the free monoid generated by elements in $S$ by concatenation, and
call its elements \emph{words} or \emph{expressions} and denote by $e$ the
\emph{identity element} or \emph{empty word}.
Given a word $w\in S^m$, it is usually written $w=w_1\cdots w_m$, where $w_i$ denotes its $i$-th letter, and the \emph{length} of $w$ is $m$.
A word $w$ of length $m$ is equivalently defined as a function $w:[m]\rightarrow S$; then $w_i$ represents the image $w(i)\in S$.
When there exists two words $u,v\in S^*$ such that $w=ufv$, the word $f\in S^*$ is called a \emph{factor} of~$w$.
An \emph{occurrence} of a factor $f$ of length~$k$ in a word $w=w_1\cdots w_m$ is a set of positions $\{i,\dots,i+(k-1)\}$ with $i\in[m-k+1]$ such that $w=w_1\cdots w_{i-1}fw_{i+k}\cdots w_{m}$, i.e. $f=w_{i}\cdots w_{i+(k-1)}$.
In particular, a factor of length $1$ gives an occurrence of some letter $s\in S$.
We denote by $\{w\}_i$ the \emph{set of occurrences of the letter} $s_i\in S$ in $w$.
Given a word $w\in S^m$, we hence associate the ordered set partition $\Omega_w = (\{w\}_1,\dots,\{w\}_n)$ of $[m]$ to $w$.
Further we denote the cardinality of the set $\{w\}_i$ by $|w|_i:=\# \{w\}_i$.
Let $k\leq m$ and $w:[m]\rightarrow S$ be a word of length~$m$, a \emph{subword}~$v$ of length~$k$ of $w$ is a word obtained by the composition $v:=
w\circ u$, for some strictly increasing function $u:[k]\rightarrow[m]$.
Therefore, a subword is a concatenation of a sequence of factors.
By extension, we define an occurrence of a subword as the union of the occurrences of the factors whose concatenation give the subword.
Consequently, the same word $v\in S^*$ may give rise to several subword occurrences in a longer word $w$, which are obtained by distinct sequences of factors of~$w$.
Given a subword $v=w\circ u$ of length $k$ of a word~$w$, its \emph{complement word} $w\setminus v$ is the subword of~$w$ obtained by $w\circ u'$ where
$u':[m-k]\rightarrow[m]$ is the increasing function such that~$u'(i)$ is the $i$-th smallest element in $[m]\setminus u([k])$.
Given a word $w=w_1\cdots w_m$, its \emph{reverse word} $\rev(w)$ is the word $w_m\cdots w_1$.

Define the monoid morphism $\abel$ from $S^*$ to the free commutative monoid $\N^S$ by its image on the set
$S$ as follows 
\[
\begin{split}
\abel:S^* & \rightarrow\N^S\\
s\,\,\, & \mapsto \mathds{1}_s(t):=\begin{cases}1, \text{ if } t=s, \\ 0, \text{ else.}\end{cases}
\end{split}
\]
Given an ordering of the alphabet of generators $S$ and a word $w\in S^*$, the image
\[
\abel(w)=(|w|_1,\dots,|w|_n)\]
is a weak composition (i.e. an ordered partition where zeros can appear) called the
\defn{abelian vector} of $w$.
The morphism $\abel$ records the number of occurrences of letters in words, and can be thought of as an ``abelianization'' of the monoid $S^*$.
Hence, the sum of the entries in $\abel(w)$ is the \emph{length} of the word $w$.
To lighten the notation, we shall write $\abel_w$ for the abelian vector $\abel(w)$.

\subsection{Coxeter groups}
\label{ssec:cox_groups}

For basic notions on Coxeter groups, we refer to the books \cite{humphreys_reflection_1990} or \cite{bjoerner_combinatorics_2005}.
Let $(W,S=\{s_1,\dots,s_n\})$ be a \emph{finite irreducible Coxeter system}
with Coxeter matrix $M=(m_{i,j})_{i,j\in[n]}$.
We denote by
$R=\{(s_is_j)^{m_{i,j}} : s_i,s_j\in S\}$ the associated set of \emph{braid relations}. 
Some choices of lexicographic orders are more natural; hence, when influencial, we specify the ordering by indicating $M$ and $R$ explicitly.
The set of braid relations $R$ generates a free submonoid $R^*$ of~$S^*$
and the quotient monoid $S^*/R^*$ consisting of left-cosets of $R^*$ in $S^*$
has a left- and right-inverse and thus forms a group which is isomorphic to~$W$
\cite[p.3]{bjoerner_combinatorics_2005}. 
From this standpoint, the elements of a Coxeter group $W$ are equivalence classes of
expressions and the representative expressions with shortest length are called \emph{reduced}.
Bearing this in mind, we henceforth represent an element $w$ of the monoid
$S^*$ and of the group~$W$ \emph{both} using concatenation of letters.
Whenever a distinction is pertinent we emphasize if the word or its equivalence class is
meant by writing $w\in S^*$, or $w\in W$, respectively.

% \begin{example}
% Let $S=\{s_1,s_2,s_3\}$, $w_1=s_2s_2s_1s_2s_3$, and $w_2=s_1s_3s_1s_3$. 
% The abelian vector of $w_1$ is $\abel_{w_1}=(1,3,1)$ and the abelian vector of $w_2$ is $\abel_{w_2}=(2,0,2)$.
% \end{example}

Throughout this text, we adopt the following notations.
The function $\ell:W\rightarrow\N$ denotes the \emph{length function} sending an element $w\in W$ to the length of its reduced expressions, the symbol~$\wo$ denotes the \emph{longest element} of $W$, and $N:=\ell(\wo)$.
By extension, $\ell$ also denotes the length function sending a word $w$ to its length.
Given an element $w\in W$, we denote the \emph{set of reduced expressions} of $w$ by $\Red(w)$ which is a finite subset of $S^*$.
Solving the following problem would be very helpful for the upcoming sections.

\begin{problem}
\label{prob:abelian}
Let $w\in W$.
\begin{enumerate}
\item Characterize the set of abelian vectors $\{\abel_v :  v\in \Red(w)\}$.
\item Give precise bounds on the number
\[
	\nu(w):=\max\big\{ \max\{\abel_v(s) : s\in
	S\} : v\in\Red(w)\big\},
\]
which is \defn{the maximum number of occurrences of a letter in any reduced expression of $w$}. 
\item Describe the vector
\[
	\mu(w):=\big(\min\{\abel_v(s) : v\in\Red(w)\}\big)_{s\in S},
\]
which gives \defn{the minimum number of occurrences of each letter in any reduced expression of $w$}. 
\end{enumerate}
\end{problem}

The abelian vectors for the reduced words of the longest element of types~$A,B,D$, for small rank and $H$ are gathered in Tables~\ref{tab:abelian_vectors_A} to \ref{tab:abelian_vectors_H} of Appendix~\ref{app:abelian vectors}.
In the symmetric group case $\Sym_{n+1}=A_n$ and taking $\wo$ for $w$, Problem~\ref{prob:abelian}(2) is related to the ``$k$-set problem'' or ``halving line problem'' via duality between points and
pseudolines on the plane, see \cite[Chapter~11]{matousek_lectures_2002}, \cite[Chapter~5]{goodman_handbook_2018}, and \cite[Section~3.1]{pilaud_multitriangulations_2012} for a contextual explanation.
Currently, the best lower bound we know for~$\nu(\wo)$ in this case is $\Omega\left(ne^{\sqrt{\log(4)\log n}}/\sqrt{\log(n)}\right)$~\cite{nivasch_improved_2008}, and the best upper bound is $O(n^{4/3})$~\cite{dey_improved_1998}.
For a nice recent book on related topics, see \cite[Section~3.5]{eppstein_forbidden_2018}.
Further, in type $A$, Problem~\ref{prob:abelian}(3) is answered by
\[
\mu(\wo) = \begin{cases} (1,2,\dots, \lceil\frac{n}{2}\rceil, \dots, 2,1) & \text{if $n$ is odd,} \\ (1,2,\dots, \frac{n}{2},\frac{n}{2},\dots, 2,1) & \text{if $n$ is even.}\end{cases}
\]
Indeed, this can be proved directly using the minimal number of inversions used at each position necessary to obtain the reverse permutation $[n+1,n,\dots,2,1]$.

We let $\nu:=\nu(\wo)$ and refer to this value as the \defn{h\"ochstfrequenz}\label{def:hf} of the group $W$.
As Section~\ref{ssec:BCL_parameter} reveals, the h\"ochstfrequenz of the group $W$ is an important indicator for the genericity of vector configurations that could geometrically realize subword complexes as chirotopes, fans or polytopes.

\subsection{Graphs on reduced expressions}

Given two words $u,v\in S^*$ representing an element $w\in W$, they are
related by a sequence of insertion or deletion of factors contained in $R^*$.
Based on this fact, one defines \emph{braid moves} in a word by replacing
a factor $s_is_js_i\cdots$ of length~$m_{i,j}$ by a factor $s_js_is_j\cdots$ of
length $m_{i,j}$, where $i\neq j$.
It is a well-known property of Coxeter groups that reduced expressions in 
$\Red(w)$ are connected via finite sequences of braid moves, in particular that no reductions
$s_i^2=e$ are necessary \cite{matsumoto_generateurs_1964,tits_probleme_1969} (see \cite[Theorem~3.3]{bjoerner_combinatorics_2005} for a textbook version). 
The graph $\G(w)$ whose vertices are reduced expressions of $w$ and
edges represent braid moves between expressions is hence connected.
Certain minors of $\G(w)$ are of particular interest here.
They are represented in Figure~\ref{fig:minors} using a Hasse diagram of the graph minor containment ordering.
For example, $\G^{\comm}(w)$ is obtained from $\G(w)$ by contracting edges of $\G(w)$ representing braid moves of length 2. 
In these minors, the resulting multiple edges are fusionned into a unique edge.

\begin{figure}[htbp]
\begin{center}
\begin{tikzpicture}

\node (gw)     at (0,3.5)  {$\G(w)$};
\node (gcomm)  at (-1.5,2.5) {$\G^{\comm}(w)$}
edge[<<-,thick] node[midway,label=above left:\small${\{2\}}$,inner sep=0pt] {} (gw);
\node (gbraid) at (1.5,2.5)  {$\G^{\braid}(w)$}
edge[<<-,thick] node[midway,label=above right:\small${\{3\}}$,inner sep=0pt] {} (gw);
\node (godd)   at (-1.5,1) {$\G^{\odd}(w)$}
edge[<<-,thick] node[midway,label=left:\small \{even\},inner sep=0pt] {} (gcomm);
\node (geven)  at (1.5,1)  {$\G^{\even}(w)$}
edge[<<-,thick] node[midway,label=right:\small \{odd\},inner sep=0pt] {} (gbraid);
\node (g2)     at (1.5,-0.5)  {$\G^{2}(w)$}
edge[<<-,thick] node[midway,label=right:\small $\{>2\}$,inner sep=0pt] {} (geven);
\end{tikzpicture}
\end{center}
\caption{Minors of $\G(w)$ obtained by contracting the edges indicated by the arrow labels: for example $\{2\}$ means that all edges labeled by ``2'' (that is commutations) are contracted.}
\label{fig:minors}
\end{figure}
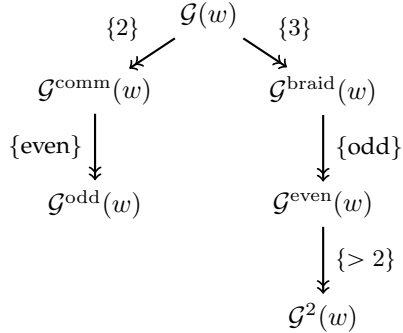

The symmetric group case has received more scrutiny: the vertices of $\G^{\braid}$ are called \defn{braid classes} and the vertices of $\G^{\comm}$ are called \defn{commutation classes}.
Bounds on the number of vertices of $\G(w)$ in terms of the number of vertices of $\G^{\braid}(w)$ and $\G^{\comm}(w)$ have been obtained in \cite{fishel_enumerations_2018}.

\begin{remark}\label{rem:simply_laced}
In the ``simply laced'' cases (types $A_n$, $D_n$, $E_6$, $E_7$, or $E_8$), there are only two types of braid moves, therefore $\G^{\braid}=\G^{\even}$ and $\G^{\comm}=\G^{\odd}$.
\end{remark}

\begin{remark}
In the symmetric group case $A_n$, the graph $\G^{\odd}(\wo)$ is the underlying graph of the Hasse diagram of the higher Bruhat order $B(n+1,2)$, see
\cite{manin_arrangements_1989,ziegler_higher_1993,felsner_theorem_2000}.
It is also studied using rhombic tilings \cite{elnitsky_rhombic_1997}, and is used to study intersections of Schubert cells \cite{shapiro_connected_1997}.
In other finite types, the relation between $\G^{\odd}(\wo)$ and potential higher Bruhat orders remains unclear.
\end{remark}

\noindent
{\bf Convention in Figures.} To lighten figures, we write reduced words such as $s_1s_2s_1s_3s_2s_1$ using the indices of the letters in $S$ to get $121321$. 
To take less space, we denote commutation classes of reduced words with more than one element as $12\{13\}21:=\{121321,123121\}$.
A class denotes by $\{232123\}$ therefore has only one word in it.
We encode the length of braid moves with continuous edges for even lengths, dashed edges for odd lengths, and doubled edges for length $4$ or $5$.

\begin{example}[Dihedral Group $I_2(m)$]
Let $W=I_2(m)$, with $m\geq 2$.
The graph $\G(\wo)$ has two vertices and one edge between them, see Figure~\ref{fig:I2_Gwo}.

\begin{figure}[!ht]
\begin{center}
\begin{tikzpicture}[vertex/.style={inner sep=1pt,circle,draw=black,fill=blue,thick},
	    deux/.style={thick,blue},
	    trois/.style={thick,dashed,red}]

\node at (-2.5,0) {$m$ even:};

\node[vertex,label=left:$(s_1s_2)^{\frac{m}{2}}$] (s1) at (0,0) {};
\node[vertex,label=right:$(s_2s_1)^{\frac{m}{2}}$] (s2) at (1.25,0) {};

\draw[deux] (s1) -- (s2);

\node at (4,0) {$m$ odd:};

\node[vertex,label=left:$(s_1s_2)^{\lfloor\frac{m}{2}\rfloor} s_1$] (t1) at (7,0) {};
\node[vertex,label=right:$(s_2s_1)^{\lfloor\frac{m}{2}\rfloor} s_2$] (t2) at (8.25,0) {};

\draw[trois] (t1) -- (t2);

\end{tikzpicture}
\end{center}
\caption{The graphs $\G(\wo)$ for the dihedral group $I_2(m)$}
\label{fig:I2_Gwo}
\end{figure}
\end{example}

\begin{example}[Symmetric group $\Sym_{4}=A_3$]
\label{ex:red_A3}
Let $W=A_3$ and $S=\{s_1,s_2,s_3\}$, such that
$(s_1s_2)^3=(s_1s_3)^2=(s_2s_3)^3=e$, see Figure~\ref{fig:A3_Gwo}.

\begin{figure}[!ht]
\resizebox{\hsize}{!}{
\begin{tabular}{c}
\begin{tikzpicture}[vertex/.style={inner sep=1pt,circle,draw=black,fill=green,thick},
	    deux/.style={thick,blue},
	    trois/.style={thick,dashed,red}]

\coordinate (center1) at (0,0);
\node at (center1) {$\G(\wo)$};

\node[vertex,label=below:{\small $123121$}] (123121) at (0.75,-0.75) {};
\node[vertex,label=below:{\small $123212$}] (123212) at (2.25,-0.75) {} edge[trois] (123121);
\node[vertex,label=below:{\small $132312$}] (132312) at (3.75,-0.75) {} edge[trois] (123212);
\node[vertex,label=above:{\small $312132$}] (312132) at (3.75,0.75) {};
\node[vertex,label=above:{\small $321232$}] (321232) at (2.25,0.75) {} edge[trois] (312132);
\node[vertex,label=above:{\small $321323$}] (321323) at (0.75,0.75) {} edge[trois] (321232);
\node[vertex,label=above:{\small $323123$}] (323123) at (-0.75,0.75) {} edge[deux] (321323);
\node[vertex,label=above:{\small $232123$}] (232123) at (-2.25,0.75) {} edge[trois] (323123);
\node[vertex,label=above:{\small $231213$}] (231213) at (-3.75,0.75) {} edge[trois] (232123);
\node[vertex,label=below:{\small $213231$}] (213231) at (-3.75,-0.75) {};
\node[vertex,label=below:{\small $212321$}] (212321) at (-2.25,-0.75) {} edge[trois] (213231);
\node[vertex,label=below:{\small $121321$}] (121321) at (-0.75,-0.75) {} edge[trois] (212321) edge[deux] (123121);

\node[vertex,label=left:{\small $312312$}] (312312) at (3,0) {} edge[deux] (132312) edge[deux] (312132);
\node[vertex,label=right:{\small $132132$}] (132132) at (4.5,0) {} edge[deux] (132312) edge[deux] (312132);
\node[vertex,label=right:{\small $213213$}] (213213) at (-3,0) {} edge[deux] (213231) edge[deux] (231213);
\node[vertex,label=left:{\small $231231$}] (231231) at (-4.5,0) {} edge[deux] (213231) edge[deux] (231213);

\node[vertex] (a) at (6.5,0.75) {};
\node[vertex] (b) at (7.5,0.75) {} edge[deux] node[midway,label=above:${m_{ij}=2}$] {} (a);

\node[vertex] (a) at (6.5,-0.75) {};
\node[vertex] (b) at (7.5,-0.75) {} edge[trois] node[midway,label=above:${m_{ij}=3}$] {} (a);

\end{tikzpicture}
\\[0.75em]
\begin{tabular}{c@{\hspace{2cm}}c}
\begin{tikzpicture}[vertex/.style={inner sep=1pt,circle,draw=black,fill=red,thick},
	    deux/.style={thick,blue},
	    trois/.style={thick,dashed,red}]

\coordinate (center1) at (0,-1.75);
\node at (center1) {$\G^2(\wo)=\G^{\braid}(\wo)=\G^{\even}(\wo)$};

\node[vertex,label=below:{\small $\{123121,123212,132312\}$}] (132312) at (2,-0.75) {};
\node[vertex,label=above:{\small $\{321323,321232,312132\}$}] (312132) at (2,0.75) {};
\node[vertex,label=above:{\small $\{231213,232123,323123\}$}] (231213) at (-2,0.75) {} edge[deux] (312132);
\node[vertex,label=below:{\small $\{213231,212321,121321\}$}] (213231) at (-2,-0.75) {} edge[deux] (132312);

\node[vertex,label=below left:{\small $\{312312\}$}] (312312) at (1.25,0) {} edge[deux] (132312) edge[deux] (312132);
\node[vertex,label=right:{\small $\{132132\}$}] (132132) at (2.75,0) {} edge[deux] (132312) edge[deux] (312132);
\node[vertex,label=above right:{\small $\{213213\}$}] (213213) at (-1.25,0) {} edge[deux] (213231) edge[deux] (231213);
\node[vertex,label=left:{\small $\{231231\}$}] (231231) at (-2.75,0) {} edge[deux] (213231) edge[deux] (231213);

\end{tikzpicture}
&
\begin{tikzpicture}[vertex/.style={inner sep=1pt,circle,draw=black,fill=blue,thick},
	    deux/.style={thick,blue},
	    trois/.style={thick,dashed,red}]

\coordinate (center1) at (0,-1.75);
\node at (center1) {$\G^{\comm}(\wo)=\G^{\odd}(\wo)$};

\node[vertex,label=below:{\small $12\{13\}21$}] (123121) at (0,-0.75) {};
\node[vertex,label=below:{\small $\{123212\}$}] (123212) at (1.5,-0.75) {} edge[trois] (123121);
\node[vertex,label=left:{\small $\{13\}2\{13\}2$}] (132312) at (2.5,0) {} edge[trois] (123212);
\node[vertex,label=above:{\small $\{321232\}$}] (321232) at (1.5,0.75) {} edge[trois] (132312);
\node[vertex,label=above:{\small $32\{13\}23$}] (321323) at (0,0.75) {} edge[trois] (321232);
\node[vertex,label=above:{\small $\{232123\}$}] (232123) at (-1.5,0.75) {} edge[trois] (321323);
\node[vertex,label=right:{\small $2\{13\}2\{13\}$}] (231213) at (-2.5,0) {} edge[trois] (232123);
\node[vertex,label=below:{\small $\{212321\}$}] (212321) at (-1.5,-0.75) {} edge[trois] (231213) edge[trois] (123121);

\end{tikzpicture}
\end{tabular}
\end{tabular}
}
\caption{The graphs $\G(\wo)$, $\G^2(\wo)=\G^{\braid}(\wo)=\G^{\even}(\wo)$, and $\G^{\comm}(\wo)=\G^{\odd}(\wo)$ for the symmetric group $\Sym_4=A_3$}
\label{fig:A3_Gwo}
\end{figure}
\end{example}

\begin{example}[The hyperoctahedral group $B_3$]
Let $W=B_3$ and $S=\{s_1,s_2,s_3\}$, such that
$(s_1s_2)^4=(s_1s_3)^2=(s_2s_3)^3=e$.
The commutation classes of $\wo$ are illustrated in Figure~\ref{fig:B3_Gwo}.
Since every odd-length braid moves has length $3$, we get $\G^{\braid}(\wo)=\G^{\even}(\wo)$.
There are $42$ reduced expressions for $\wo$. 
The impatient reader is invited to see Figure~\ref{fig:B3_Ssign} on page \pageref{fig:B3_Ssign} for an illustration.
\end{example}

\begin{figure}[H]
\begin{center}
\resizebox{0.85\hsize}{!}
{
\begin{tikzpicture}[vertex/.style={inner sep=1pt,circle,draw=black,fill=blue,thick},
	    quatre/.style={thick,blue,double,double distance=1mm},
	    trois/.style={thick,dashed,red}]

\coordinate (center1) at (0,-1.5);
\node at (center1) {$\G^{\comm}(\wo)$};

\node[vertex] (a) at (7,0.75) {};
\node[vertex] (b) at (8.25,0.75) {} edge[quatre] node[midway,label=below:${m_{ij}=4}$] {} (a);

\node[vertex] (c) at (7,-0.75) {};
\node[vertex] (d) at (8.25,-0.75) {} edge[trois] node[midway,label=below:${m_{ij}=3}$] {} (c);

\node[vertex,label={[anchor=225]above:{\small $3212\{13\}212$}}] (rb1) at (0.75,-0.75) {};
\node[vertex,label={below:{\small $\{13\}2123212$}}] (rb2) at (2.25,-0.75) {} edge[quatre] (rb1);
\node[vertex,label={below right:{\small $\{13\}2\{13\}2\{13\}2$}}] (rb3) at (3.75,-0.75) {} edge[trois] (rb2);
\node[vertex,label={right:{\small $123212\{13\}2$}}] (r) at (4.5,0) {} edge[trois] (rb3);

\node[vertex,label={[anchor=135]below:{\small $\{121232123\}$}}] (ra1) at (0.75,0.75) {};
\node[vertex,label={above:{\small $12\{13\}2\{13\}23$}}] (ra2) at (2.25,0.75) {} edge[trois] (ra1);
\node[vertex,label={above right:{\small $12\{13\}21232$}}] (ra3) at (3.75,0.75) {} edge[trois] (ra2) edge[quatre] (r);

\node[vertex,label={[anchor=-45]above:{\small $\{321232121\}$}}] (lb1) at (-0.75,-0.75) {} edge[quatre] (rb1);
\node[vertex,label={below:{\small $32\{13\}2\{13\}21$}}] (lb2) at (-2.25,-0.75) {} edge[trois] (lb1);
\node[vertex,label={below left:{\small $23212\{13\}21$}}] (lb3) at (-3.75,-0.75) {} edge[trois] (lb2);
\node[vertex,label={[rotate=0]left:{\small $2\{13\}212321$}}] (l) at (-4.5,0) {} edge[quatre] (lb3);

\node[vertex,label={[anchor=45]below:{\small $212\{13\}2123$}}] (la1) at (-0.75,0.75) {} edge[quatre] (ra1);
\node[vertex,label={above:{\small $2123212\{13\}$}}] (la2) at (-2.25,0.75) {} edge[quatre] (la1);
\node[vertex,label={above left:{\small $2\{13\}2\{13\}2\{13\}$}}] (la3) at (-3.75,0.75) {} edge[trois] (la2) edge[trois] (l);

\end{tikzpicture}
}
\end{center}
\caption{The commutation classes of the longest word of the group $B_3$}
\label{fig:B3_Gwo}

\end{figure}
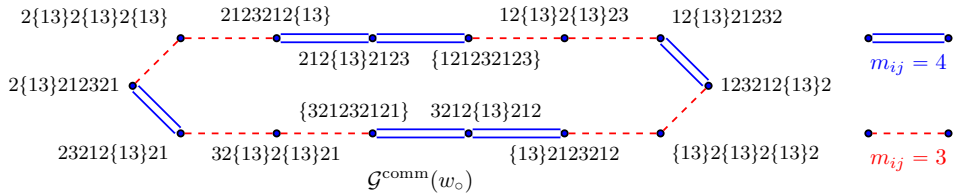

\begin{example}[Icosahedral group $H_3$]
Let $W=H_3$ and $S=\{s_1,s_2,s_3\}$, such that
$(s_1s_2)^5=(s_1s_3)^2=(s_2s_3)^3=e$.
There are $44$ commutation classes in $\mathcal{G}^{comm}(\wo)$; see Figure~\ref{fig:H3_Gwo} for an illustration.

\begin{figure}[!ht]
\resizebox{\hsize}{!}{
\begin{tikzpicture}[vertex/.style={inner sep=1pt,circle,draw=black,fill=blue,thick},
	    cinq/.style={thick,forestgreen,double,dashed,double distance=1mm},
	    trois/.style={thick,dashed,red}]

\def\lgx{1.5}
\def\lgy{0.70710678118}

\coordinate (center1) at (0,0);
\node at (center1) {$\G^{\comm}(\wo)$};

\node[vertex] (a) at (-1*\lgx,-8.5*\lgy) {};
\node[vertex] (b) at (-2*\lgx,-8.5*\lgy) {} edge[cinq] node[midway,label=below:${m_{ij}=5}$] {} (a);

\node[vertex] (c) at (1*\lgx,-8.5*\lgy) {};
\node[vertex] (d) at (2*\lgx,-8.5*\lgy) {} edge[trois] node[midway,label=below:${m_{ij}=3}$] {} (c);

% 7 Bottom vertices:
\node[vertex,label={[anchor=north west,rotate=-30]above:{\small $212\{13\}2\{13\}2\{13\}2123$}}] (b0) at (\lgx*0,\lgy*-5) {};
\node[vertex,label={[anchor=north west,rotate=-30]above:{\small $212123212\{13\}2123$}}] (b1) at (\lgx*1,\lgy*-5) {} edge[trois] (b0);
\node[vertex,label={[anchor=north west,rotate=-30]above:{\small $1212\{13\}212\{13\}2123$}}] (b2) at (\lgx*2,\lgy*-5) {} edge[cinq] (b1);
\node[vertex,label={[anchor=north west,rotate=-30]above:{\small $\{121232121232123\}$}}] (b3) at (\lgx*3,\lgy*-5) {} edge[cinq] (b2);

\node[vertex,label={[anchor=north west,rotate=-30]above:{\small $212\{13\}2123212123$}}] (b-1) at (\lgx*-1,\lgy*-5) {} edge[trois] (b0);
\node[vertex,label={[anchor=north west,rotate=-30]above:{\small $212\{13\}212\{13\}212\{13\}$}}] (b-2) at (\lgx*-2,\lgy*-5) {} edge[cinq] (b-1);
\node[vertex,label={[anchor=north west,rotate=-30]above:{\small $2123212123212\{13\}$}}] (b-3) at (\lgx*-3,\lgy*-5) {} edge[cinq] (b-2);

% Bottom right chessboard
\node[vertex,label={right:{\small $12123212\{13\}2\{13\}23$}}] (br_1r) at (\lgx*3.5,\lgy*-4) {} edge[trois] (b3);
\node[vertex,label={left :{\small $12\{13\}2\{13\}21232123$}}] (br_1l) at (\lgx*2.5,\lgy*-4) {} edge[trois] (b3);
\node[vertex,label={right:{\small $12123212\{13\}21232$}}] (br_2r) at (\lgx*4,\lgy*-3) {} edge[trois] (br_1r);
\node[vertex,label={left :{\small $12\{13\}2\{13\}2\{13\}2\{13\}23$}}] (br_2l) at (\lgx*3,\lgy*-3) {} edge[trois] (br_1l) edge[trois] (br_1r);
\node[vertex,label={right:{\small $12\{13\}2\{13\}2\{13\}21232$}}] (br_3r) at (\lgx*3.5,\lgy*-2) {} edge[trois] (br_2r) edge[trois] (br_2l);
\node[vertex,label={left :{\small $12\{13\}2123212\{13\}23$}}] (br_3l) at (\lgx*2.5,\lgy*-2) {} edge[trois] (br_2l);
\node[vertex,label={right:{\small $12\{13\}21232121232$}}] (br_4) at (\lgx*3,\lgy*-1) {} edge[trois] (br_3l) edge[trois] (br_3r);

% Bottom left chessboard
\node[vertex,label={right:{\small $2\{13\}2\{13\}2123212\{13\}$}}] (bl_1r) at (\lgx*-2.5,\lgy*-4) {} edge[trois] (b-3);
\node[vertex,label={left :{\small $2123212\{13\}2\{13\}2\{13\}$}}] (bl_1l) at (\lgx*-3.5,\lgy*-4) {} edge[trois] (b-3);
\node[vertex,label={right:{\small $2\{13\}2\{13\}2\{13\}2\{13\}2\{13\}$}}] (bl_2r) at (\lgx*-3,\lgy*-3) {} edge[trois] (bl_1r) edge[trois] (bl_1l);
\node[vertex,label={left :{\small $2123212\{13\}212321$}}] (bl_2l) at (\lgx*-4,\lgy*-3) {} edge[trois] (bl_1l);
\node[vertex,label={right:{\small $2\{13\}2123212\{13\}2\{13\}$}}] (bl_3r) at (\lgx*-2.5,\lgy*-2) {} edge[trois] (bl_2r);
\node[vertex,label={left :{\small $2\{13\}2\{13\}2\{13\}212321$}}] (bl_3l) at (\lgx*-3.5,\lgy*-2) {} edge[trois] (bl_2l) edge[trois] (bl_2r);
\node[vertex,label={left :{\small $2\{13\}212321212321$}}] (bl_4) at (\lgx*-3,\lgy*-1) {} edge[trois] (bl_3l) edge[trois] (bl_3r);

% Middle two vertices
\node[vertex,label={left :{\small $2\{13\}212\{13\}212\{13\}21$}}] (l) at (\lgx*-3,0) {} edge[cinq] (bl_4);
\node[vertex,label={right:{\small $12\{13\}212\{13\}212\{13\}2$}}] (r) at (\lgx*3,0) {} edge[cinq] (br_4);

% Top right chessboard
\node[vertex,label={right:{\small $123212123212\{13\}2$}}] (tr_4) at (\lgx*3,\lgy*1) {} edge[cinq] (r);
\node[vertex,label={right:{\small $123212\{13\}2\{13\}2\{13\}2$}}] (tr_3r) at (\lgx*3.5,\lgy*2) {} edge[trois] (tr_4);
\node[vertex,label={left :{\small $\{13\}2\{13\}2123212\{13\}2$}}] (tr_3l) at (\lgx*2.5,\lgy*2) {} edge[trois] (tr_4);
\node[vertex,label={right:{\small $123212\{13\}2123212$}}] (tr_2r) at (\lgx*4,\lgy*3) {} edge[trois] (tr_3r);
\node[vertex,label={left :{\small $\{13\}2\{13\}2\{13\}2\{13\}2\{13\}2$}}] (tr_2l) at (\lgx*3,\lgy*3) {} edge[trois] (tr_3l) edge[trois] (tr_3r);

\node[vertex,label={right:{\small $\{13\}2\{13\}2\{13\}2123212$}}] (tr_1r) at (\lgx*3.5,\lgy*4) {} edge[trois] (tr_2l) edge[trois] (tr_2r);
\node[vertex,label={left :{\small $\{13\}2123212\{13\}2\{13\}2$}}] (tr_1l) at (\lgx*2.5,\lgy*4) {} edge[trois] (tr_2l);

% Top left chessboard
\node[vertex,label={left :{\small $23212123212\{13\}21$}}] (tl_4)  at (\lgx*-3,\lgy*1)   {} edge[cinq]  (l);
\node[vertex,label={right:{\small $32\{13\}2123212\{13\}21$}}] (tl_3r) at (\lgx*-2.5,\lgy*2) {} edge[trois] (tl_4);
\node[vertex,label={left :{\small $23212\{13\}2\{13\}2\{13\}21$}}] (tl_3l) at (\lgx*-3.5,\lgy*2) {} edge[trois] (tl_4);
\node[vertex,label={right:{\small $32\{13\}2\{13\}2\{13\}2\{13\}21$}}] (tl_2r) at (\lgx*-3,\lgy*3)   {} edge[trois] (tl_3r) edge[trois] (tl_3l);
\node[vertex,label={left :{\small $23212\{13\}21232121$}}] (tl_2l) at (\lgx*-4,\lgy*3)   {} edge[trois] (tl_3l);

\node[vertex,label={right:{\small $32123212\{13\}2\{13\}21$}}] (tl_1r) at (\lgx*-2.5,\lgy*4) {} edge[trois] (tl_2r);
\node[vertex,label={left :{\small $32\{13\}2\{13\}21232121$}}] (tl_1l) at (\lgx*-3.5,\lgy*4) {} edge[trois] (tl_2l) edge[trois] (tl_2r);

% 7 Top vertices:
\node[vertex,label={[anchor=south west,rotate=30]above:{\small $3212\{13\}2\{13\}2\{13\}212$}}] (t0) at (\lgx*0,\lgy*5) {};
\node[vertex,label={[anchor=south west,rotate=30]above:{\small $3212123212\{13\}212$}}] (t1) at (\lgx*1,\lgy*5) {} edge[trois] (t0);
\node[vertex,label={[anchor=south west,rotate=30]above:{\small $\{13\}212\{13\}212\{13\}212$}}] (t2) at (\lgx*2,\lgy*5) {} edge[cinq] (t1);
\node[vertex,label={[anchor=south west,rotate=30]above:{\small $\{13\}2123212123212$}}] (t3) at (\lgx*3,\lgy*5) {} edge[cinq] (t2) edge[trois] (tr_1r) edge[trois] (tr_1l);

\node[vertex,label={[anchor=south west,rotate=30]above:{\small $3212\{13\}212321212$}}] (t-1) at (\lgx*-1,\lgy*5) {} edge[trois] (t0);
\node[vertex,label={[anchor=south west,rotate=30]above:{\small $3212\{13\}212\{13\}2121$}}] (t-2) at (\lgx*-2,\lgy*5) {} edge[cinq] (t-1);
\node[vertex,label={[anchor=south west,rotate=30]above:{\small $\{321232121232121\}$}}] (t-3) at (\lgx*-3,\lgy*5) {} edge[cinq] (t-2) edge[trois] (tl_1r) edge[trois] (tl_1l);

\end{tikzpicture}
}
\caption{The commutation classes of the longest word of the group $H_3$}
\label{fig:H3_Gwo}

\end{figure}
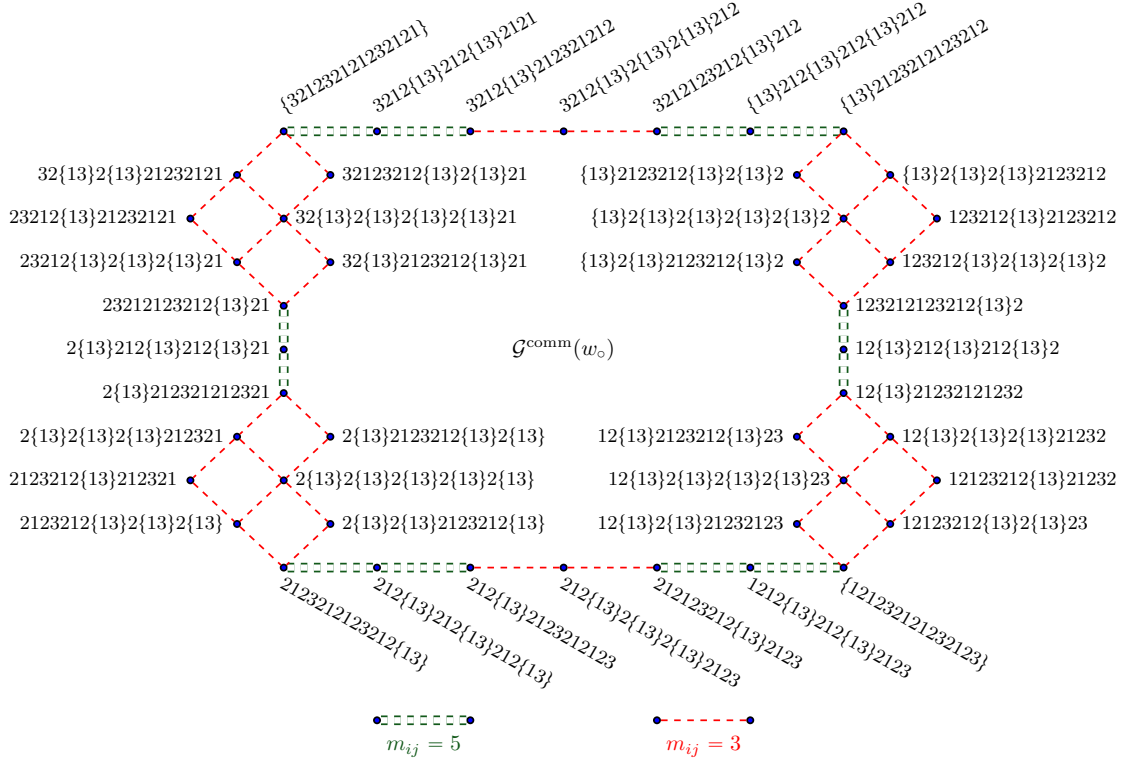
\end{example}

Of particular importance for us is the fact that the graphs $\G(\wo)$, $\G^{\even}(w)$, and $\G^{\odd}(w)$ are bipartite graphs.
The following theorem was proved for finite Coxeter group by Bergeron, Ceballos, and the author using a geometric
argument in \cite{bergeron_fan_2015} and generalized to infinite Coxeter groups
and extended to a finer description by Grinberg and Postnikov in \cite{grinberg_proof_2017} using only
conjugations instead of automorphisms.

\begin{lemma}[{\cite[Theorem~3.1]{bergeron_fan_2015},\cite[Theorem~2.0.3]{grinberg_proof_2017}}]
\label{lem:megabi}
Let $W$ be a Coxeter group and $w\in W$. For any set $\mathcal{E}$ of braid moves
closed under automorphism of $W$, the minor of $\G(w)$ obtained by contracting
edges not contained in $\mathcal{E}$ is a bipartite graph. In particular, $\G(w)$,
$\G^{\even}(w)$, and $\G^{\odd}(w)$ are bipartite graphs.
\end{lemma}

\subsection{Subword complexes of Coxeter groups}

For each finite Coxeter group $(W,S)$, Knutson and Miller introduced a family of simplicial complexes called subword complexes that are constructed by using the Bruhat order within words in~$S^*$ \cite{knutson_subword_2004,knutson_groebner_2005}.
Subword complexes form one of the tools they used to connect the algebra and combinatorics of Schubert polynomials to the geometry of Schubert varieties.
We present here an adaptation of the original definition, using results from \cite[Section~3]{ceballos_subword_2014} and \cite[Section~2]{bergeron_fan_2015}, that particularly suits our purposes.
Given a word $p\in S^*$, we can order the occurrences of all its subwords by set-inclusion to obtain a Boolean lattice $(2^{[\ell(p)]},\subseteq)$ and subword complexes are certain lower ideals of such Boolean lattices determined using reduced words. 
More precisely, let $(W,S)$ be a finite Coxeter group, $\wo$ be its longest element and $p\in S^m$.
The \emph{subword complex}~$\Delta_W(p)$ is the simplicial complex on the set $[m]$ whose facets are complements of occurrences of reduced words for $\wo$ in the word~$p$~\cite[Definition~2.1]{knutson_subword_2004}.

\begin{remark}
The original definition allows to use reduced words for any element $\pi\in W$ instead of the longest element $\wo$. 
By \cite[Theorem~3.7]{ceballos_subword_2014}, every spherical subword complex $\Delta_W(p,\pi)$ is isomorphic to~$\Delta_W(p')$ for some word~$p'$.
\end{remark}

Subword complexes possess a particularly nice combinatorial and topological structure: They are vertex-decomposable and homeomorphic to sphere or balls \cite[Theorem~2.5 and~3.7]{knutson_subword_2004}.
Knutson and Miller originally asked whether spherical subword complexes can be realized as the boundary of a convex polytope \cite[Question~6.4]{knutson_subword_2004}.
So far, the realized subword complexes include famous polytopes: simplices, even-dimensional cyclic polytopes, polar dual of generalized associahedra, see \cite[Section~6]{ceballos_subword_2014} for a survey on the related conjectures and the references therein.
Subword complexes of type~$A$ are intimately related to multi-triangulations, a generalization of usual triangulations of a convex polygon \cite{jonsson_abstract_2003,pilaud_multitriangulations_2012}.
The only ``non-classical'' instance of multi-associahedron which is realized is a $6$-dimensional polytope with $12$ vertices realizing the simplicial complex of $2$-triangulations of the $8$-gon, see \cite{bokowski_symmetric_2009,ceballos_associahedra_2012,bergeron_fan_2015}, which is a type $A_3$ subword complex.
Further, fan realizations for type $A_3$ and two cases in $A_4$ spherical subword complexes have been provided \cite{bergeron_fan_2015} and for $2$-triangulations (type~$A$) with rank $5,6,7$ and $8$ \cite{manneville_fan_2018}.

Due to their combinatorial provenance, we can attribute a combinatorial type to each facet.

\begin{definition}[Combinatorial type and abelian vector of facet]
Let $\Delta_W(p)$ be a subword complex.
The \defn{combinatorial type} of a facet $f$ of $\Delta_W(p)$ is the complement subword $p\setminus f$ of $f$ in~$p$.
Two facets are \defn{combinatorially equivalent} when their combinatorial types are the same.
The \defn{abelian vector} of a facet~$f$ is the abelian vector $\abel_{p\setminus f}$ of its combinatorial type.
\end{definition}

\begin{example}
Let $W=A_2$ and $p=p_1p_2p_3p_4p_5=s_1s_2s_1s_2s_2$. The subword complex $\Delta_{A_2}(p)$ has two combinatorial types of facets. 
The facets $\{1,4\}$ and $\{1,5\}$ have type $s_2s_1s_2$ and the facet $\{4,5\}$ has type $s_1s_2s_1$. 
The abelian vector of $\{1,4\}$ and $\{1,5\}$ is $(1,2)$ and the abelian vector of $\{4,5\}$ is $(2,1)$.
The letters $p_2$ and $p_3$ are contained in every occurrences of reduced word for $\wo$ in $p$, so $2$ and $3$ are non-vertices of the subword complex $\Delta_{A_2}(p)$, see Figure~\ref{fig:sw_complex}.

\begin{figure}[!ht]
\begin{center}
\resizebox{!}{2cm}{
\begin{tikzpicture}[vertex/.style={inner sep=1pt,circle,draw=black,fill=blue,thick},
	    nonvertex/.style={inner sep=1pt,cross out,draw=black,fill=blue,thick},
	    edge/.style={thick,blue,}]

\node[vertex,label=right:$1$] (1) at (0:1) {};
\node[vertex,label=left:$4$] (4) at (120:1) {} edge[edge] (1);
\node[vertex,label=left:$5$] (5) at (240:1) {} edge[edge] (4) edge[edge] (1);

\node[nonvertex] (2) at (3,-1) {2};
\node[nonvertex] (3) at (3,1) {3};
\end{tikzpicture}
}
\end{center}
\caption{A small subword complex with two non-vertices}
\label{fig:sw_complex}
\end{figure}
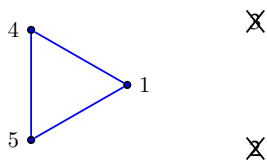
\end{example}

\subsection{Fans, Gale duality and chirotopes}

In this section, we describe a general technique to realize simplicial spheres as the boundary of convex polytopes.
We rely on the reference books \cite[Section~5.4]{grunbaum_convex_2003}, \cite[Chapter~6]{ziegler_lectures_1995} and \cite[Section 2.5, 4.1, and 5.4]{de_loera_triangulations_2010} for the notions used here.
We assume the reader's familiarity with the elementary objects from convex geometry apart from the following essential notions, that we recall below.

Given a label set $J$ of cardinality $m$, a \emph{vector configuration} in $\R^d$ is a set $\A:=\{{\bf p}_j~|~j\in~J\}$ of labeled vectors ${\bf p}_j\in\R^d$~\cite[Definition~2.5.1]{de_loera_triangulations_2010}.
The \emph{rank} of $\A$ is its rank as a set of vectors.
We assume vector configurations to have maximal rank, i.e. $r=d$, and we write them as matrices in $\R^{d\times m}$; observe that in order to do so we implicitly assume that $J$ is totally ordered.
We denote the non-negative span of a set of vectors labeled by a subset $C\subseteq J$ by $\cone_\A(C)$.
A vector configuration is \emph{acyclic} if there is a linear function that is positive in all the elements of the configuration. 
It is \emph{totally cyclic} if $\cone_\A(J)$ is equal to the vector space spanned by $\A$.

Let $\A$ be a vector configuration in $\R^d$ consisting of $m$ labeled vectors.
A \emph{fan} supported by $\A$ is a family $\F=\{K_1,K_2,\dots,K_k\}$ of non-empty polyhedral cones generated by vectors in $\A$ such that every non-empty face of a cone in $\F$ is also a cone in $\F$, and the intersection of any two cones in~$\F$ is a face of both, see e.g. \cite[Section~7.1]{ziegler_lectures_1995}.
A fan $\F$ is \emph{simplicial} if every $K\in\F$ is a simplicial cone, that is if every cone is generated by linearly independent vectors, and it is \emph{complete} if the union $K_1\cup \dots \cup K_k$ is~$\R^d$.
The $1$-dimensional cones of a fan are called \emph{rays}.
In a complete fan, inclusion-maximal cones are called \emph{facets}, and $(d-1)$-dimensional cones are called \emph{ridges}.

The simplicial complex $\Delta_\F$ on $[m]$ whose faces are index sets of cones in a complete simplicial fan $\F$ is homeomorphic to a sphere.
In this case, we say that $\F$ gives a \defn{complete simplicial fan realization} of the simplicial complex $\Delta_\F$.
Such simplicial complexes are also called \defn{geodesic spheres}, since the intersection of $\F$ with the unit sphere is a geodesic triangulation of it.
Other names include ``star-shaped'' or ``fan-like'', see e.g. \cite[Part~I, Chapter~3, Definition~5.4 and 5.6]{ewald_combinatorial_1996}.
To each vector configuration correspond a dual object called Gale transform.

\begin{definition}[{Gale transform \cite[Definition~4.1.35]{de_loera_triangulations_2010} \cite[Section~6.4]{ziegler_lectures_1995}}]
Let $\A\in\R^{d\times m}$ be a rank $d$ vector configuration with $m$ elements.
A \defn{Gale transform} $\B\in\R^{(m-d)\times m}$ of $\A$ is a vector configuration of rank $m-r$ whose rowspan equals the right-kernel of $\A$.
We denote by $\Gale(\A)$ the set of all Gale transforms of $\A$.
\end{definition}

Let $\A\in\R^{d\times m}$ be a vector configuration supporting a complete simplicial fan $\F$.
Since $\A$ is totally cyclic, Gale duality implies that its Gale transforms are acyclic.
Furthermore, Gale duality makes an independent subset $C\subseteq J$ in the primal $\A$ correspond to a spanning subset $J\setminus C$ in the dual
$\B$, and vice-versa.
Full-dimensional cones in $\F$ are spanning and independent in $\A$ and hence are complements of independent and spanning subconfigurations of any $\B\in\Gale(\A)$.
Given such a full-dimensional cone in $\F$ spanned by a set of vectors $C\subseteq J$, its \emph{dual simplex}~$C^*$ is $J\setminus C$.
The cone generated by $C^*$ in $\B\in\Gale(\A)$ is called a \emph{dual simplicial cone} of~$C$ and is denoted by $\cone_\B(C^*)$.
Following \cite[Section 9.5]{de_loera_triangulations_2010}, to obtain a \defn{realization of a simplicial sphere} $\Delta$ \defn{as the boundary of a convex simplicial polytope} one possibility is to proceed in two steps:

\begin{enumerate}
\item[(T)] Obtain a vector configuration $\A$ supporting a complete simplicial fan $\F$ whose geodesic sphere $\Delta_\F$ has a face lattice isomorphic to the face lattice of the simplicial sphere $\Delta$, (Triangulation) and
\item[(R)] prove that the underlying triangulation of the space by~$\Delta_\F$ is regular. Equivalently, find one point on each ray, so that
taking the convex hull of these points yields a simplicial polytope whose boundary complex is isomorphic to $\Delta$, see
\cite[Corollary~9.5.3]{de_loera_triangulations_2010}. (Regularity)
\end{enumerate}

\noindent
The first step relies heavily on the combinatorial structure of the sphere, whereas the success of the second step relies heavily on the geometry of the obtained simplicial fan.

To complete step~(T), we need to find a vector configuration $\A$ with a triangulation $\Delta$, whose facial structure dictates orientations of sets of vectors.
In particular, every cone in $\F$ spanned by exactly $d$ vectors in $\A$ gets a sign in $\{-1,0,+1\}$.
This information is encoded via (abstract) chirotopes.

\begin{definition}[{Chirotope, see e.g. \cite[Theorem~3.6.2]{bjoerner_oriented_1999}}]
Let $J$ be a finite totally ordered label set and $r$ be a positive integer.
A \defn{chirotope} of rank $r$ is a map
$
\chi: J^r \rightarrow \{-1,0,1\}
$
such that
\begin{enumerate}
\item $\chi$ is an alternating function, i.e. $\chi\circ w = \sign(w)\chi$ for all permutations $w\in\Sym_J$.
\item the set $\{B\in J^r~:~\chi(B)\neq 0\}$ is non-empty and is the set of basis of a matroid of rank $r$, the \defn{underlying matroid} of $\chi$.
\item for any $\gamma\in J^{r-2}$ and $a,b,c,d\in J$, if
\[
\chi(a,d,\gamma)\cdot\chi(b,c,\gamma)\geq 0
\qquad\text{and}\qquad
\chi(c,d,\gamma)\cdot\chi(a,b,\gamma)\geq 0,
\]
then $\chi(a,c,\gamma)\cdot\chi(b,d,\gamma)\geq0$. (Pl\"ucker--Grassmann relations)
\end{enumerate}
\end{definition}

\begin{definition}[{Chirotope of a vector configuration \cite[Definition~8.1.1]{de_loera_triangulations_2010}}]
Let $\A$ be a vector configuration in~$\R^d$.
The \defn{chirotope $\chi_\A$ of rank} $r:=d$ \defn{associated to} $\A$ is the map
$$
\begin{array}{rrcl}
\chi_\A:& J^{d}             & \rightarrow & \{-1,0,1\}\\
& (j_1,\dots,j_{d}) & \mapsto     & \sign \det [\A]_{(j_1,\dots,j_{d})}.
\end{array}
$$
\end{definition}

In oriented matroid terms, chirotopes of vector configurations are also referred to as \emph{realizable chirotopes}.
In turn, realizable chirotopes form an equivalent encoding of \emph{realizable oriented matroids}.
The vector configuration $\A$ gives a realization of the underlying chirotope $\chi_\A$.
We recall the following lemma that gives conditions to form a complete simplicial fan and fulfill step (T).
The lemma has the peculiarity that the first two conditions are expressed using a Gale dual, which is practical when considering subword complexes.

\begin{lemma}[{see \cite[Lemma 3]{bergeron_fan_2015}}]
\label{lem:simpl_fan}
Let $\Delta$ be a simplicial complex on $J=[m]$ homeomorphic to a sphere of dimension $d-1$.
A totally cyclic vector configuration $\A\in\R^{d\times m}$ supports a complete simplicial fan realization of $\Delta$ if and only if the following conditions on $\A$ and a Gale transform $\B\in\Gale(\A)$ are satisfied.
\begin{itemize}
\item[(B)] Dual simplicial cones $\cone_\B(C^*)$ are independent in $\R^{m-r}$. (Basis)
\item[(F)] If $I$ and $J$ are two facets of $\Delta$ intersecting along a ridge, then the interior of the corresponding dual simplicial cones intersect. (Flip)
\item[(I)] There is a cone in $\A$ whose interior is not intersected by any other cones. (Injectivity)
\end{itemize}
\end{lemma}

Chirotopes are practical since properties (B) and (F) enjoy a simple interpretation in terms of signs: cones corresponding to facets of $\Delta$ have non-zero signs (Basis), and two facets intersecting along a ridge have compatible sign with respect to the underlying total order on $J$ (Flip).
This compatibility is made explicit in the next section.
If $\A$ is a vector configuration satisfying conditions (B) and (F), we say that $\Delta$ is \defn{realizable as a chirotope}, a \defn{chirotopal sphere}, or that $\A$ provides a \defn{realization of~$\Delta$ as a chirotope}.
The conditions (B) and (F) only give a partial description of a chirotope, hence it is possible to have several abstract chirotopes realizing the same simplicial sphere~$\Delta$.

Now, assume that we have completed step (T) obtaining a complete fan $\F$ with a vector configuration $\A$ and a Gale transform~$\B$.
The next step is to prove that the induced triangulation~$\Delta_\F$ of $\R^d$ is regular.
Having a fan $\F$ already guarantees that the dual simplicial cones of adjacent facets intersect in their interior.
To have a regular triangulation, the common intersection of \emph{all} dual simplicial cones should have non-empty interior:

\begin{proposition}[{see \cite[Theorem~5.4.5 and 5.4.7]{de_loera_triangulations_2010}}]
Let $\F$ be a complete simplicial fan in $\R^d$ supported by a configuration $\A$ of $m$ vectors and $\B\in\Gale(\A)$.
The triangulation of $\R^d$ induced by $\F$ is regular if and only if the intersection of all dual simplicial cones in $\B$ is a full-dimensional cone in~$\R^{m-d}$.
\end{proposition}

The relationships between geometric realizations of a spherical simplicial complex $\Delta$ as the boundary of a convex simplicial polytope, as a complete fan and as a chirotope are summarized as follows (compare \cite[Part~I, End of Chapter~3.5 and Chapter~5.7]{ewald_combinatorial_1996}):

\begin{center}
$\Delta$ is polytopal \quad$\Longrightarrow$\quad $\Delta$ is geodesic \quad$\Longrightarrow$\quad $\Delta$ is chirotopal.
\end{center}

\noindent
Going one step from right to left represents a significant leap in difficulty.
Evidently, realizable chirotopes are necessary to obtain geometric realizations as polytopes or complete fans.
Hence understanding the structure of the chirotopes realizing subword complexes is crucial.
Further, since chirotopes are inherently combinatorial and do not capture all of the necessary intricate convex structure, they may be described and studied with more ease.
In the upcoming sections, we thus proceed to describe a particular family of partial chirotopes---refered to as \emph{parameter matrices}---that uncovers the structure of chirotopes of subword complexes. 
Finally, we will show that the existence of parameter matrices is necessary for the realizability of subword complexes as chirotopes and hence as complete fans and polytopes.

\section{Sign functions of words}
\label{sec:s_sign}

In the symmetric group $\Sym_{n+1}\cong A_n$, the sign of a permutation is defined using the parity of its number of pairwise inversions; even permutations being ``$+$'' and odd permutations being ``$-$''.
This definition shows directly that the Hasse diagram of the weak order of the symmetric group (and more generally of Coxeter groups) is bipartite.
In this section, we present an extension of this notion on subsets of words of $S^*$ defined using sign functions.

\begin{definition}[Sign functions on words]
\label{def:sign_fct}
Let $M\subseteq S^*$ be a non-empty set of words in $S$.
A function from $M$ to the multiplicative group $\mathbb{Z}_2$ is a \defn{sign function} on $M$.
The set of sign functions on $M$ equipped with the binary operation
\[
\begin{split}
\mathbb{Z}_2\!{}^M \times \mathbb{Z}_2\!{}^M & \rightarrow\mathbb{Z}_2\!{}^M\\
(\phi,\psi) & \mapsto \phi\psi(m):=\begin{cases} 1, & \text{if } \phi(m)=\psi(m), \\ -1, & \text{else,}\end{cases}
\end{split}
\]
forms a group: the \defn{group of sign functions on} $M$.
\end{definition}

\subsection{Sign function on minors of $\G(w)$}

Given an element $w\in W$, Lemma~\ref{lem:megabi} gives a way to define a close cousin of signs of permutations where now the ground set is $\Red(w)$ and even and odd expressions are then defined in various ways, as Lemma~\ref{lem:megabi} permits.
Two cases are more relevant:

\begin{description}
\item[Sign function $F_{\odd}$] Changing the sign when an odd-length braid move is done, and
      leaving the sign unchanged when an even-length braid move is done. 
      Since the minor $\G^{\odd}(w)$ is bipartite, we can assign 
      one part to have positive sign and the other to have negative 
      sign. This way, the sign changes along every edge representing a braid move of 
      odd length, and the sign remains unchanged along the contracted edges.
\item[Sign function $F_{\even}$] Changing the sign when an even-length braid move is done, and
      leaving the sign unchanged when an odd-length braid move is done.
      Similarly, since $\G^{\even}(w)$ is bipartite we can assign 
      positive and negative signs to the reduced expressions. 
\end{description}

\noindent As illustrated in Figure \ref{fig:minors}, since $\G^{\odd}$ is a minor of $\G^{\comm}$ and $\G^{\even}$ is a minor of $\G^{\braid}$, the sign functions $F_{\even}$ and $F_{\odd}$ are class functions on braid classes and commutation classes, respectively.
These sign functions are unique up to a global multiplication by ``$-1$''.

\begin{example}[$F_{\even}$ and $F_{\odd}$ sign functions on braid and commutation classes in type~$A_3$]
We can give ``$+$'' and ``$-$'' signs to the vertices of $\G^{\even}(\wo)$ and $\G^{\odd}(\wo)$ and by Remark~\ref{rem:simply_laced}, we get
functions on braid and commutation classes that change along edges in $\G^{\even}(\wo)$ and $\G^{\odd}(\wo)$, see Figure~\ref{fig:A3_Tsign}.
\end{example}

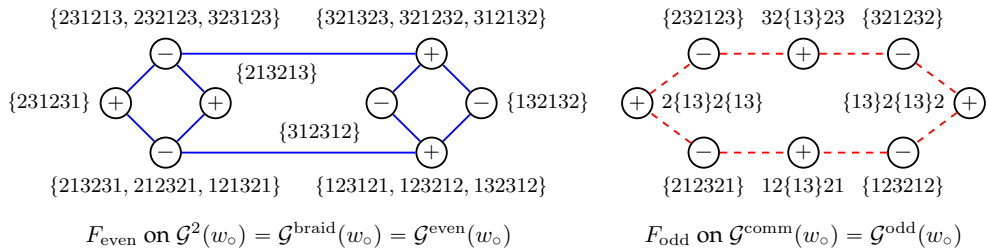
\begin{figure}[!ht]
\begin{center}
\resizebox{.9\hsize}{!}{
\begin{tabular}{cc}
\begin{tikzpicture}[vertex/.style={inner sep=1pt,circle,draw=black,thick},
	    deux/.style={thick,blue},
	    trois/.style={thick,dashed,red}]

\coordinate (center1) at (0,-2);
\node at (center1) {$F_{\even}$ on $\G^2(\wo)=\G^{\braid}(\wo)=\G^{\even}(\wo)$};

\node[vertex,label=below:{\small $\{123121,123212,132312\}$}] (132312) at (2,-0.75) {$+$};
\node[vertex,label=above:{\small $\{321323,321232,312132\}$}] (312132) at (2,0.75) {$+$};
\node[vertex,label=above:{\small $\{231213,232123,323123\}$}] (231213) at (-2,0.75) {$-$} edge[deux] (312132);
\node[vertex,label=below:{\small $\{213231,212321,121321\}$}] (213231) at (-2,-0.75) {$-$} edge[deux] (132312);

\node[vertex,label=below left:{\small $\{312312\}$}] (312312) at (1.25,0) {$-$} edge[deux] (132312) edge[deux] (312132);
\node[vertex,label=right:{\small $\{132132\}$}] (132132) at (2.75,0) {$-$} edge[deux] (132312) edge[deux] (312132);
\node[vertex,label=above right:{\small $\{213213\}$}] (213213) at (-1.25,0) {$+$} edge[deux] (213231) edge[deux] (231213);
\node[vertex,label=left:{\small $\{231231\}$}] (231231) at (-2.75,0) {$+$} edge[deux] (213231) edge[deux] (231213);

\end{tikzpicture}
&
\begin{tikzpicture}[vertex/.style={inner sep=1pt,circle,draw=black,thick},
	    deux/.style={thick,blue},
	    trois/.style={thick,dashed,red}]

\coordinate (center1) at (0,-2);
\node at (center1) {$F_{\odd}$ on $\G^{\comm}(\wo)=\G^{\odd}(\wo)$};

\node[vertex,label=below:{\small $12\{13\}21$}] (123121) at (0,-0.75) {$+$};
\node[vertex,label=below:{\small $\{123212\}$}] (123212) at (1.5,-0.75) {$-$} edge[trois] (123121);
\node[vertex,label=left:{\small $\{13\}2\{13\}2$}] (132312) at (2.5,0) {$+$} edge[trois] (123212);
\node[vertex,label=above:{\small $\{321232\}$}] (321232) at (1.5,0.75) {$-$} edge[trois] (132312);
\node[vertex,label=above:{\small $32\{13\}23$}] (321323) at (0,0.75) {$+$} edge[trois] (321232);
\node[vertex,label=above:{\small $\{232123\}$}] (232123) at (-1.5,0.75) {$-$} edge[trois] (321323);
\node[vertex,label=right:{\small $2\{13\}2\{13\}$}] (231213) at (-2.5,0) {$+$} edge[trois] (232123);
\node[vertex,label=below:{\small $\{212321\}$}] (212321) at (-1.5,-0.75) {$-$} edge[trois] (231213) edge[trois] (123121);

\end{tikzpicture}
\end{tabular}
}
\end{center}
\caption{$F_{\even}$ and $F_{\odd}$ sign functions on braid and commutation classes of $\wo$ in type $A_3$}
\label{fig:A3_Tsign}
\end{figure}

\begin{example}[$F_{\even}$ sign function on braid classes of type $B_3$]
We can give ``$+$'' and ``$-$'' signs to the vertices of $\G^{\even}(\wo)$ and get a class function on braid classes, illustrated in Figure~\ref{fig:B3_Tsign}.

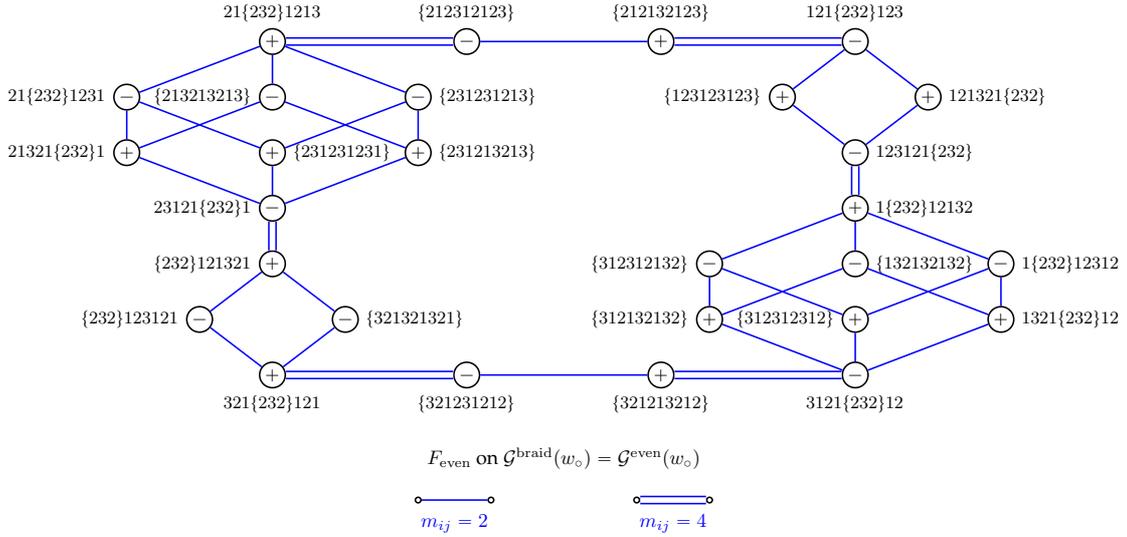
\begin{figure}[H]
\resizebox{1\hsize}{!}{
\begin{tikzpicture}[vertex/.style={inner sep=1pt,circle,draw=black,thick},
	    deux/.style={thick,blue},
	    quatre/.style={thick,blue,double,double distance=1mm}]

\def\lgx{2.6}
\def\lgy{1}

\coordinate (center1) at (0,-4.5);
\node at (center1) {$F_{\even}$ on $\G^{\braid}(\wo)=\G^{\even}(\wo)$};

\node[vertex] (a) at (-0.5*\lgx,-5.25*\lgy) {};
\node[vertex] (b) at (-1*\lgx,-5.25*\lgy) {} edge[deux] node[midway,label=below:${m_{ij}=2}$] {} (a);

\node[vertex] (c) at (0.5*\lgx,-5.25*\lgy) {};
\node[vertex] (d) at (1*\lgx,-5.25*\lgy) {} edge[quatre] node[midway,label=below:${m_{ij}=4}$] {} (c);

% Lower right cube:
\node[vertex,label=below:{\small $3121\{232\}12$}] (r_cu_0) at (2*\lgx,-3*\lgy) {$-$};
\node[vertex,label=left :{\small $\{312132132\}$}] (r_cu_11) at (\lgx,-2*\lgy) {$+$}
edge[deux] (r_cu_0);
\node[vertex,label=left :{\small $\{312312312\}$}] (r_cu_12) at (2*\lgx,-2*\lgy) {$+$}
edge[deux] (r_cu_0);
\node[vertex,label=right:{\small $1321\{232\}12$}] (r_cu_13) at (3*\lgx,-2*\lgy) {$+$}
edge[deux] (r_cu_0);
\node[vertex,label=left :{\small $\{312312132\}$}] (r_cu_21) at (\lgx,-1*\lgy) {$-$}
edge[deux] (r_cu_11)
edge[deux] (r_cu_12);
\node[vertex,label=right:{\small $\{132132132\}$}] (r_cu_22) at (2*\lgx,-1*\lgy) {$-$}
edge[deux] (r_cu_11)
edge[deux] (r_cu_13);
\node[vertex,label=right:{\small $1\{232\}12312$}] (r_cu_23) at (3*\lgx,-1*\lgy) {$-$}
edge[deux] (r_cu_12)
edge[deux] (r_cu_13);
\node[vertex,label=right:{\small $1\{232\}12132$}] (r_cu_3) at (2*\lgx,0) {$+$}
edge[deux] (r_cu_21)
edge[deux] (r_cu_22)
edge[deux] (r_cu_23);

% Upper right square:
\node[vertex,label=right:{\small $123121\{232\}$}] (r_sq_0)  at (2*\lgx,\lgy) {$-$}
edge[quatre] (r_cu_3);
\node[vertex,label=left :{\small $\{123123123\}$}] (r_sq_11)  at (1.5*\lgx,2*\lgy) {$+$}
edge[deux] (r_sq_0);
\node[vertex,label=right:{\small $121321\{232\}$}] (r_sq_12)  at (2.5*\lgx,2*\lgy) {$+$}
edge[deux] (r_sq_0);
\node[vertex,label=above:{\small $121\{232\}123$}] (r_sq_2)  at (2*\lgx,3*\lgy) {$-$}
edge[deux] (r_sq_11)
edge[deux] (r_sq_12);

% Upper left cube:
\node[vertex,label=above:{\small $21\{232\}1213$}] (l_cu_0)  at (-2*\lgx,3*\lgy) {$+$};
\node[vertex,label=right:{\small $\{231231213\}$}] (l_cu_11) at (-\lgx,2*\lgy) {$-$}
edge[deux] (l_cu_0);
\node[vertex,label=left :{\small $\{213213213\}$}] (l_cu_12) at (-2*\lgx,2*\lgy) {$-$}
edge[deux] (l_cu_0);
\node[vertex,label=left:{\small $21\{232\}1231$}] (l_cu_13) at (-3*\lgx,2*\lgy) {$-$}
edge[deux] (l_cu_0);
\node[vertex,label=right:{\small $\{231213213\}$}] (l_cu_21) at (-\lgx,1*\lgy) {$+$}
edge[deux] (l_cu_11)
edge[deux] (l_cu_12);
\node[vertex,label=right:{\small $\{231231231\}$}] (l_cu_22) at (-2*\lgx,1*\lgy) {$+$}
edge[deux] (l_cu_11)
edge[deux] (l_cu_13);
\node[vertex,label=left :{\small $21321\{232\}1$}] (l_cu_23) at (-3*\lgx,1*\lgy) {$+$}
edge[deux] (l_cu_12)
edge[deux] (l_cu_13);
\node[vertex,label=left :{\small $23121\{232\}1$}] (l_cu_3)  at (-2*\lgx,0) {$-$}
edge[deux] (l_cu_21)
edge[deux] (l_cu_22)
edge[deux] (l_cu_23);

% Lower left square:
\node[vertex,label=left :{\small $\{232\}121321$}] (l_sq_0)  at (-2*\lgx,-\lgy) {$+$}
edge[quatre] (l_cu_3);
\node[vertex,label=left :{\small $\{232\}123121$}] (l_sq_11)  at (-2.5*\lgx,-2*\lgy) {$-$}
edge[deux] (l_sq_0);
\node[vertex,label=right:{\small $\{321321321\}$}] (l_sq_12)  at (-1.5*\lgx,-2*\lgy) {$-$}
edge[deux] (l_sq_0);
\node[vertex,label=below:{\small $321\{232\}121$}] (l_sq_2)  at (-2*\lgx,-3*\lgy) {$+$}
edge[deux] (l_sq_11)
edge[deux] (l_sq_12);

% Lower vertices:
\node[vertex,label=below:{\small $\{321231212\}$}] (b_1)  at (-2*\lgx/3,-3*\lgy) {$-$}
edge[quatre] (l_sq_2);
\node[vertex,label=below:{\small $\{321213212\}$}] (b_2)  at (2*\lgx/3,-3*\lgy) {$+$}
edge[deux] (b_1)
edge[quatre] (r_cu_0);

% Upper vertices:
\node[vertex,label=above:{\small $\{212312123\}$}] (u_1)  at (-2*\lgx/3,3*\lgy) {$-$}
edge[quatre] (l_cu_0);
\node[vertex,label=above:{\small $\{212132123\}$}] (u_2)  at (2*\lgx/3,3*\lgy) {$+$}
edge[deux] (u_1)
edge[quatre] (r_sq_2);

\end{tikzpicture}
}
\caption{$F_{\even}$ sign function on braid classes for $\wo$ in type $B_3$}
\label{fig:B3_Tsign}
\end{figure}
\end{example}

The $T$-signature of reduced expressions is the sign function where even-length braid moves change the sign, i.e., the sign function $F_{\even}$ above on the graph $\G^{\even}(\wo)$.
The \defn{$T$-sign function} is defined as
\[
\tau:\Red(\wo) \rightarrow \{+1,-1\},
\]
such that 
if $w$ and $w'$ are two reduced expressions of $\wo$ related by a braid move of length $m_{i,j}$, 
then $\tau(w)=(-1)^{m_{i,j}-1}\tau(w')$, see~\cite[Definition~3.5]{bergeron_fan_2015}.

\begin{remark}
The letter $T$ is used to hint at a usual notation for the set of reflections of $W$, see \cite[Remark~3.7]{bergeron_fan_2015} for an equivalent formulation in type $A$ using $T$.
The $T$-sign function is well-defined by Lemma~\ref{lem:megabi} and unique up to a global multiplication by ``$-1$''.
A certain choice is well-suited for our purpose and is specified in Section~\ref{ssec:st_sign}.
% This signature has relation to the study of scattering amplitudes and plabic
% graphs, see \cite{arkani_grassmannian_2016}.\marginpar{add more recent
% references}
\end{remark}

\begin{example}
In Figures~\ref{fig:A3_Tsign} and \ref{fig:B3_Tsign}, we can determine the sign of all reduced expressions using the bipartiteness of $\G^{\even}(\wo)$, namely every element in the same braid class (of length $3$ in those cases) receive the same sign. 
\end{example}

The $T$-sign function is central to a combinatorial and geometric condition that led to the construction of complete simplicial fans for subword complexes in \cite{bergeron_fan_2015}.
It helped to deliver complete fans with cone lattices corresponding to subword complexes of type $A_3$ and some of type~$A_4$.
This condition is based on signature matrices:

\begin{definition}[{Coxeter signature matrices \cite[Definition~9]{bergeron_fan_2015}}]
\label{def:signature_matrix}
Let $(W,S)$ be a finite irreducible Coxeter system and $p\in S^m$.
A matrix $\B\in\R^{N\times m}$ is a \defn{signature matrix} of type~$W$ for $p$, if for every occurrence $Z$ of every reduced expression $v$ of $\wo$ in~$p$, the equality
\[
\sign(\det [\B]_Z) = \tau(v)
\]
holds.
\end{definition}

The vector configuration underlying a signature matrix yields a \emph{partially} defined chirotope: subsets corresponding to a reduced expression $v$ of $\wo$ must have the non-zero sign $\tau(v)$, while all others are left undetermined.
Let $p\in S^m$ and $\A\in~\R^{(m-N)\times m}$.
We denote by $\F_{p,\A}$ the collection of cones spanned by sets of columns of $\A$ that correspond to faces of the subword complex $\Delta_W(p)$.
The following proposition demonstrates the important role of signature matrices, and hence of the $T$-sign function, in order for $\F_{p,\A}$ to form a complete simplicial fan.
In particular, it bring to light how the $T$-sign prescribes the orientation of dual simplicial cones in order to get a complete simplicial fan.

\begin{proposition}[{\cite[Section~3.1, Theorem~3.7]{ceballos_associahedra_2012}, see also \cite[Theorem~3]{bergeron_fan_2015}}]
\label{prop:signature_matrix}
Let $p\in S^m$ and $\A\in~\R^{(m-N)\times m}$.
The collection of cones $\F_{p,\A}$ is a complete simplicial fan realizing the subword complex $\Delta_W(p)$ if and only if
\begin{itemize}
\item[(S)] a Gale transform $\B\in\Gale(\A)$ is a signature matrix for $p$, (Signature) and
\item[(I)] there is a facet of $\Delta_W(p)$ such that the interior of its associated cone in $\F_{p,\A}$ does not intersect any other cone of $\F_{p,\A}$. (Injectivity)
\end{itemize}
\end{proposition}

\subsection{The $S$-sign function}

The monoid morphism from $S^*$ to $\mathbb{Z}_2$ sending a word $w$ to $(-1)^{\ell(w)}$ is the \emph{parity sign function}.
This function does not take the lexicographic order on $S$ into account.
In contrast, the $S$-sign function defined below intrinsically makes use of the lexicographic order.
As it turns out, the $S$-sign function is an integral part of the sign of $\det[\B]_Z$ in Definition~\ref{def:signature_matrix}.

Before defining the $S$-sign function, we first give some basic concepts, see e.g.~\cite[Chapter~1]{diekert_combinatorics_1990} for more details.
Let $w\in S^m$ be a word with abelian vector $\abel_w=(c_i)_{s_i\in S}$.
The word $\overline{w}:=s_1^{c_1}\dots s_n^{c_n}$ is the \emph{lexicographic normal form} of $\abel_w$.
Permutations in $\Sym_m$ acts on the letters of $w$ as
\[
\pi\cdot w := w_{\pi(1)} \cdots w_{\pi(m)},
\]
where $\pi\in \Sym_m$.
The permutation of $\Sym_m$ with exactly the same inversions as $w$ (inversions are defined similarly as for permutations written as lists) is its \emph{standard permutation} and is denoted $\std(w)$.
The standard permutation is the minimal length permutation whose inverse sorts $w$: $\std(w)^{-1}\cdot w = \overline{w}$.
The \emph{inversion number} $\inv(w)$ of $w$ is the number of inversions of $\std(w)$.
Equivalently, the inversion number $\inv(w)$ is the smallest number of swaps of two consecutive letters of $w$ required to obtain the lexicographic normal form of
$\abel_w$.

\begin{remark}
The lexicographic normal form appears as the ``non-decreasing rearrangement'' of a word, see e.g. \cite[Section~2]{hohlweg_inverses_2001}.
Therein, the authors refer to an article of Schensted \cite{schensted_longest_1961} where the standardization of a word was introduced. 
\end{remark}

The following lemma can be verified through a direct calculation.

\begin{lemma}\label{lem:count_inv}
Let $v\in S^m$.
The number of swaps $\inv(v)$ equals the number of pairs $(i,j)\in[m]^2$ such that $i<j$ and the letter $v_j\in S$ is smaller than the letter $v_i\in S$ in the lexicographic order.
\end{lemma}

\begin{definition}[$S$-sign of a word]
Let $w\in S^*$.
The \defn{$S$-sign} $\sigma(w)$ of $w$ is
\[
\sigma(w):=(-1)^{\inv(w)},
\]
where $\inv(w)$ is the inversion number of $w$.
\end{definition}

The following proposition predicts the behavior of the $S$-sign depending on the abelian vector of words.

\begin{proposition}\label{prop:total_invs}
Let $w\in S^m$ and let $\abel_w=(c_i)_{s_i\in S}$ be its abelian vector.
The inversion number of~$w$ and its reverse $\rev(w)$ satisfy
\[
\inv(w) + \inv(\rev(w)) = \binom{m}{2} - \sum_{s\in S} \binom{c_i}{2}.
\]
Therefore, $\inv(w)$ and $\inv(\rev(w))$ have the same parity if and only if $\binom{m}{2} - \sum_{s\in S} \binom{c_i}{2}$ is even.
\end{proposition}

\begin{proof}
We give a bijective proof.
There are $\binom{m}{2}$ pairs of distinct positions in the word $w$.
These pairs split into three exclusive cases.
\begin{itemize}
\item The two letters are the same in $S$,
\item the two letters are in lexicographic order from left-to-right, or
\item the two letters are in lexicographic order from right-to-left.
\end{itemize}
There are $\sum_{s\in S}\binom{c_i}{2}$ pairs in the first case.
The two other cases are exactly the inversions of $w$ and $\rev(w)$ by Lemma~\ref{lem:count_inv}.
\end{proof}

\begin{example}[Sign of permutations]
The $S$-signature of words in $S^*$ is an extension of the usual sign function on permutation.
Let $w\in S^*$ with abelian vector $\abel_w=(1,\dots,1)$.
Reading the indices of the letters of the word $w$ from left to right gives a permutation of $[n]$.
The inversion number $\inv(w)$ counts the number of pairs of 
(necessarily distinct) numbers in $[n]$ which are unordered in~$w$ and hence the $S$-sign of $w$ is equal to the usual sign of the permutation that $w$ represents.
\end{example}

\begin{proposition}
Let $w\in S^*$ and $\phi:S\rightarrow S$ be the lexicographic order reversing map such that $\phi(s_i)=s_{n-i+1}$.
Denote by $\overleftarrow{\inv}(w)$ the number of swaps of $w$ in the reversed
ordering of the alphabet.
Then
\[
\overleftarrow{\inv}(w) = \inv(\rev(w)) = \inv(\phi(w)).
\]
\end{proposition}

\begin{proof}
When permuting the letters of $w$ to obtain the lexicographic normal form with respect to the reverse lexicographic order we obtain the reverse of the
lexicographic normal form obtained with the usual ordering of the alphabet.
Therefore, reversing $w$ and ordering alphabetically gives the same number of inversions by symmetry, proving $\overleftarrow{\inv}(w) = \inv(\rev(w))$.

By Lemma~\ref{lem:count_inv}, the number $\overleftarrow{\inv}(w)$ is equal to the number of pairs $(i,j)\in[m]^2$ such that $i<j$ and the letter $w_j\in S$ is \emph{larger} than the letter $w_i\in S$ in the lexicographic order.
By applying $\phi$ to~$w$ these pairs $(i,j)$ become exactly the inversions of $\phi(w)$, proving that $\overleftarrow{\inv}(w)=\inv(\phi(w))$.
\end{proof}

The $S$-sign of words behaves differently from the $T$-sign along braid moves
as the following theorem shows.

\begin{mainthm}
\label{thm:A_Sbraid}
Let $1\leq i<j \leq n$, and $u,v\in S^*$ be two words. 
Further, define
\[
b_{i,j} := s_is_js_i\dots \text{ of length } m_{i,j}, \quad
\kappa  := \sum_{i<k\leq j}|u|_{k}, \quad
\text{ and } \mu     := \sum_{i\leq k < j} |v|_k.
\]
In other words, the number $\kappa$ is the number of occurrences of letters $s_k$ in $u$ such that $i<k\leq j$ and $\mu$ is the number of occurrences of letters $s_k$ in $v$ such that $i\leq k < j$.
The $S$-sign function $\sigma$ satisfies
\[
\sigma(ub_{i,j}v)=
\begin{cases} (-1)^{\frac{m_{i,j}}{2}}\sigma(ub_{j,i}v), & \text{ if } m_{i,j} \text{ is even}, \\ 
(-1)^{\kappa+\mu}\sigma(ub_{j,i}v), & \text{ if } m_{i,j} \text{ is odd.}
\end{cases}
\]
\end{mainthm}

\begin{proof}
By Lemma~\ref{lem:count_inv}, we have to track the change in the number of inversions in the word after doing a braid move.

Suppose that $m_{i,j}$ is even.
Since the abelian vector of the words $ub_{i,j}v$ and $ub_{j,i}v$ are the same, it
suffices to examine the changes in the number of swaps involving two letters
that are contained in $b_{i,j}$. Indeed, the ordering of any
other pair of positions stay unchanged. The number of swaps in $b_{i,j}$ is
$m_{i,j}(m_{i,j}-2)/8$ and the number of swaps in $b_{j,i}$ is
$m_{i,j}(m_{i,j}+2)/8$ hence their difference is ${m_{i,j}}/{2}$.

Suppose that $m_{i,j}$ is odd.
It suffices to consider the change in the number of swaps involving at least
one position in $b_{i,j}$.  The number of swaps in $b_{i,j}$ and
$b_{j,i}$ are the same, since the first occurrence of~$s_i$ is not swapped with any occurrence of $s_j$ in
$b_{i,j}$, and putting it at the end and simultaneously replacing it by $s_j$ to obtain
$b_{j,i}$ does not create any new swap.
Therefore, we only need to count the
number of swaps involving the first occurrence of $s_i$ in $b_{i,j}$ with letters in $u$,
which are not swaps once the occurrence of $s_i$ is moved at the end of~$b_{i,j}$ and replaced by~$s_j$. 
This number is exactly $\kappa$. 
Further, after removing $s_i$ at the
beginning of $b_{i,j}$ and putting $s_j$ at its end, we create swaps with the
letters in $v$ which did not need to be swapped with $s_i$, this number of swaps is exactly~$\mu$.
\end{proof}

\begin{maincor}
\label{cor:B_commutations}
Let $w\in W$ and $u,v\in\Red(w)$ be two reduced words for $w$ that are related by $k$
commutations, i.e. braid moves of length 2. The $S$- and $T$-sign function 
satisfy
\[
\sigma(u)=(-1)^k\sigma(v) \text { and } \tau(u)=(-1)^k\tau(v).
\]
In other words, both the $S$-sign and the $T$-sign change along braid moves of length $2$.
\end{maincor}

\subsection{$S$-sign functions on reduced expressions for small rank Coxeter groups}

\begin{example}[Dihedral Group $I_2(m)$]
Let $W=I_2(m)$, with $m\geq 2$.
The $S$-sign function for the reduced expressions is determined by the residue of $m\ \mathrm{mod}\ 4$, see Figure~\ref{fig:I2_Ssign}.
\end{example}

\begin{figure}[H]
\begin{center}
\begin{tikzpicture}[vertex/.style={inner sep=1pt,circle,draw=black,thick},
	    deux/.style={thick,blue},
	    trois/.style={thick,dashed,red},
	    quatre/.style={thick,blue,double,double distance=1mm}]

\node at (-4.5,1) {$m\equiv 2\mod 4$:};

\node[vertex,label=left:$(s_1s_2)^{\frac{m}{2}}$] (s1) at (-5,0) {};
\node[vertex,label=right:${(s_2s_1)^{\frac{m}{2}}}$] (s2) at (-4,0) {};

\draw[deux] (s1) -- node[midway,label=above:$-$] {} (s2);

\node at (0,1) {$m\equiv 0\mod 4$:};

\node[vertex,label=left:$(s_1s_2)^{\frac{m}{2}}$] (t1) at (-0.5,0) {};
\node[vertex,label=right:${(s_2s_1)^{\frac{m}{2}}}$] (t2) at (0.5,0) {};

\draw[quatre] (t1) -- node[midway,label=above:$+$] {} (t2);

\node at (5,1) {$m\equiv 1,3\mod 4$:};

\node[vertex,label=left:$(s_1s_2)^{\lfloor\frac{m}{2}\rfloor} s_1$] (u1) at (4.5,0) {};
\node[vertex,label=right:$(s_2s_1)^{\lfloor\frac{m}{2}\rfloor} s_2$] (u2) at (5.5,0) {};

\draw[trois] (u1) -- node[midway,label=above:$+$] {} (u2);

\end{tikzpicture}
\end{center}
\caption{The $S$-sign for reduced expressions of $\wo$ for the dihedral group $I_2(m)$. Since the $S$-sign values vary within the same residue class, we label
the edge by the product of the $S$-signs of its vertices which is invariant.}
\label{fig:I2_Ssign}
\end{figure}

\begin{example}
Let $W=B_3$.
The $S$-sign function on reduced expressions of $\wo$ is illustrated in Figure~\ref{fig:B3_Ssign}.
The $S$-sign does not change on all braid moves of length $3$, hence in this case, the $S$-sign is a well-defined class function on braid
classes $\G^{\even}(\wo)=\G^{\braid}(\wo)$.
Nevertheless, it is not equal to the $T$-sign function on $\G^{\even}(\wo)$ since the $T$-sign function changes along braid moves of length~$4$.

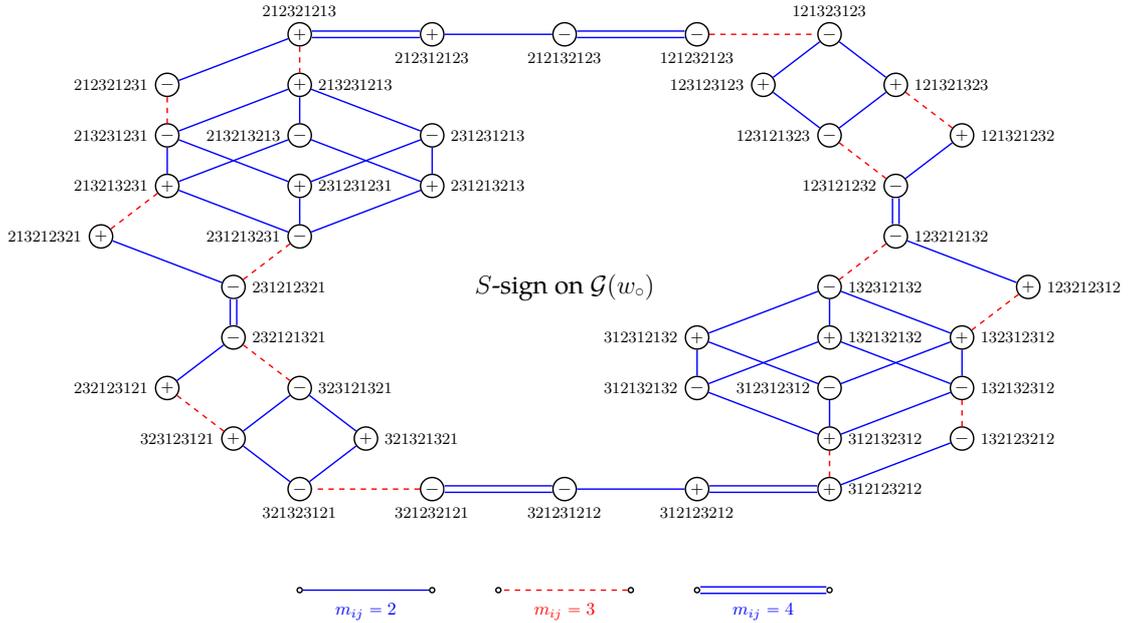
\begin{figure}[H]
\resizebox{\hsize}{!}{
\begin{tikzpicture}[vertex/.style={inner sep=1pt,circle,draw=black,thick},
	    deux/.style={thick,blue},
	    trois/.style={thick,dashed,red},
	    quatre/.style={thick,blue,double,double distance=1mm}]

\def\lgx{2.6}
\def\lgy{1}

\coordinate (center1) at (0,0);
\node at (center1) {{\LARGE $S$-sign on $\G(\wo)$}};

\node[vertex] (a) at (-2*\lgx,-6*\lgy) {};
\node[vertex] (b) at (-1*\lgx,-6*\lgy) {} edge[deux] node[midway,label=below:${m_{ij}=2}$] {} (a);

\node[vertex] (c) at (-0.5*\lgx,-6*\lgy) {};
\node[vertex] (d) at (0.5*\lgx,-6*\lgy) {} edge[trois] node[midway,label=below:${m_{ij}=3}$] {} (c);

\node[vertex] (e) at (1*\lgx,-6*\lgy) {};
\node[vertex] (f) at (2*\lgx,-6*\lgy) {} edge[quatre] node[midway,label=below:${m_{ij}=4}$] {} (e);

% Lower right cube:
\node[vertex,label=right:{\small $312132312$}] (r_cu_0) at (2*\lgx,-3*\lgy) {$+$};
\node[vertex,label=left :{\small $312132132$}] (r_cu_11) at (\lgx,-2*\lgy) {$-$}
edge[deux] (r_cu_0);
\node[vertex,label=left :{\small $312312312$}] (r_cu_12) at (2*\lgx,-2*\lgy) {$-$}
edge[deux] (r_cu_0);
\node[vertex,label=right:{\small $132132312$}] (r_cu_13) at (3*\lgx,-2*\lgy) {$-$}
edge[deux] (r_cu_0);
\node[vertex,label=left :{\small $312312132$}] (r_cu_21) at (\lgx,-1*\lgy) {$+$}
edge[deux] (r_cu_11)
edge[deux] (r_cu_12);
\node[vertex,label=right:{\small $132132132$}] (r_cu_22) at (2*\lgx,-1*\lgy) {$+$}
edge[deux] (r_cu_11)
edge[deux] (r_cu_13);
\node[vertex,label=right:{\small $132312312$}] (r_cu_23) at (3*\lgx,-1*\lgy) {$+$}
edge[deux] (r_cu_12)
edge[deux] (r_cu_13);
\node[vertex,label=right:{\small $132312132$}] (r_cu_3) at (2*\lgx,0) {$-$}
edge[deux] (r_cu_21)
edge[deux] (r_cu_22)
edge[deux] (r_cu_23);
\node[vertex,label=right:{\small $123212132$}] (r_cu_41) at (2.5*\lgx,1) {$-$}
edge[trois] (r_cu_3);
\node[vertex,label=right:{\small $123212312$}] (r_cu_32) at (3.5*\lgx,0) {$+$}
edge[trois] (r_cu_23)
edge[deux]  (r_cu_41);
\node[vertex,label=right:{\small $132123212$}] (r_cu_02) at (3*\lgx,-3) {$-$}
edge[trois] (r_cu_13);
\node[vertex,label=right:{\small $312123212$}] (r_cu_m1) at (2*\lgx,-4) {$+$}
edge[deux] (r_cu_02)
edge[trois]  (r_cu_0);

% Upper right square:
\node[vertex,label=left:{\small $123121232$}] (r_sq_m1)  at (2.5*\lgx,2*\lgy) {$-$}
edge[quatre] (r_cu_41);
\node[vertex,label=left:{\small $123121323$}] (r_sq_0)  at (2*\lgx,3*\lgy) {$-$}
edge[trois] (r_sq_m1);
\node[vertex,label=right:{\small $121321232$}] (r_sq_02)  at (3*\lgx,3*\lgy) {$+$}
edge[deux] (r_sq_m1);
\node[vertex,label=left :{\small $123123123$}] (r_sq_11)  at (1.5*\lgx,4*\lgy) {$+$}
edge[deux] (r_sq_0);
\node[vertex,label=right:{\small $121321323$}] (r_sq_12)  at (2.5*\lgx,4*\lgy) {$+$}
edge[deux] (r_sq_0)
edge[trois] (r_sq_02);
\node[vertex,label=above:{\small $121323123$}] (r_sq_2)  at (2*\lgx,5*\lgy) {$-$}
edge[deux] (r_sq_11)
edge[deux] (r_sq_12);

% Upper left cube:
\node[vertex,label=right:{\small $213231213$}] (l_cu_0)  at (-2*\lgx,4*\lgy) {$+$};
\node[vertex,label=right:{\small $231231213$}] (l_cu_11) at (-\lgx,3*\lgy) {$-$}
edge[deux] (l_cu_0);
\node[vertex,label=left :{\small $213213213$}] (l_cu_12) at (-2*\lgx,3*\lgy) {$-$}
edge[deux] (l_cu_0);
\node[vertex,label=left:{\small $213231231$}] (l_cu_13) at (-3*\lgx,3*\lgy) {$-$}
edge[deux] (l_cu_0);
\node[vertex,label=right:{\small $231213213$}] (l_cu_21) at (-\lgx,2*\lgy) {$+$}
edge[deux] (l_cu_11)
edge[deux] (l_cu_12);
\node[vertex,label=right:{\small $231231231$}] (l_cu_22) at (-2*\lgx,2*\lgy) {$+$}
edge[deux] (l_cu_11)
edge[deux] (l_cu_13);
\node[vertex,label=left :{\small $213213231$}] (l_cu_23) at (-3*\lgx,2*\lgy) {$+$}
edge[deux] (l_cu_12)
edge[deux] (l_cu_13);
\node[vertex,label=left :{\small $231213231$}] (l_cu_3)  at (-2*\lgx,1) {$-$}
edge[deux] (l_cu_21)
edge[deux] (l_cu_22)
edge[deux] (l_cu_23);
\node[vertex,label=right:{\small $231212321$}] (l_cu_41) at (-2.5*\lgx,0) {$-$}
edge[trois] (l_cu_3);
\node[vertex,label=left :{\small $213212321$}] (l_cu_32) at (-3.5*\lgx,1) {$+$}
edge[trois] (l_cu_23)
edge[deux]  (l_cu_41);
\node[vertex,label=left :{\small $212321231$}] (l_cu_02) at (-3*\lgx,4) {$-$}
edge[trois] (l_cu_13);
\node[vertex,label=above:{\small $212321213$}] (l_cu_m1) at (-2*\lgx,5) {$+$}
edge[deux] (l_cu_02)
edge[trois]  (l_cu_0);

% Lower left square:
\node[vertex,label=right:{\small $323121321$}] (l_sq_0)  at (-2*\lgx,-2*\lgy) {$-$};
\node[vertex,label=left :{\small $323123121$}] (l_sq_11)  at (-2.5*\lgx,-3*\lgy) {$+$}
edge[deux] (l_sq_0);
\node[vertex,label=right:{\small $321321321$}] (l_sq_12)  at (-1.5*\lgx,-3*\lgy) {$+$}
edge[deux] (l_sq_0);
\node[vertex,label=below:{\small $321323121$}] (l_sq_22)  at (-2*\lgx,-4*\lgy) {$-$}
edge[deux] (l_sq_11)
edge[deux] (l_sq_12);
\node[vertex,label=left:{\small $232123121$}] (l_sq_02)  at (-3*\lgx,-2*\lgy) {$+$}
edge[trois] (l_sq_11);
\node[vertex,label=right:{\small $232121321$}] (l_sq_m1)  at (-2.5*\lgx,-1*\lgy) {$-$}
edge[quatre] (l_cu_41)
edge[trois] (l_sq_0)
edge[deux] (l_sq_02);

% Lower vertices:
\node[vertex,label=below:{\small $321232121$}] (b_1)  at (-1*\lgx,-4*\lgy) {$-$}
edge[trois] (l_sq_22);
\node[vertex,label=below:{\small $321231212$}] (b_2)  at (-0*\lgx,-4*\lgy) {$-$}
edge[quatre] (b_1);
\node[vertex,label=below:{\small $312123212$}] (b_3)  at (1*\lgx,-4*\lgy) {$+$}
edge[deux] (b_2)
edge[quatre] (r_cu_m1);

% Upper vertices:
\node[vertex,label=below:{\small $212312123$}] (u_1)  at (-1*\lgx,5*\lgy) {$+$}
edge[quatre] (l_cu_m1);
\node[vertex,label=below:{\small $212132123$}] (u_2)  at (-0*\lgx,5*\lgy) {$-$}
edge[deux] (u_1);
\node[vertex,label=below:{\small $121232123$}] (u_3)  at (1*\lgx,5*\lgy) {$-$}
edge[quatre] (u_2)
edge[trois] (r_sq_2);

\end{tikzpicture}
}
\caption{The $S$-sign function on reduced words of $\wo$ in type $B_3$}
\label{fig:B3_Ssign}
\end{figure}
\end{example}

\begin{example}[Symmetric group $\Sym_4=A_3$]
\label{ex:s_sign_a3}
Let $W=A_3$.
The $S$-sign function for the reduced expressions of $\wo$ \emph{is not} equal to the $T$-sign, see Figure~\ref{fig:A3_Ssign}.
Consider the underlined braid move of length $3$ between $\underline{\smash{s_1s_2s_1}}s_3s_2s_1$ and $\underline{\smash{s_2s_1s_2}}s_3s_2s_1$.
On the one hand, the $S$-sign changes because in this case, we compute $\kappa=0$ and $\mu=1$ in Theorem~\ref{thm:A_Sbraid}.
On the other hand, the $T$-sign does not change since~$3$ is odd.

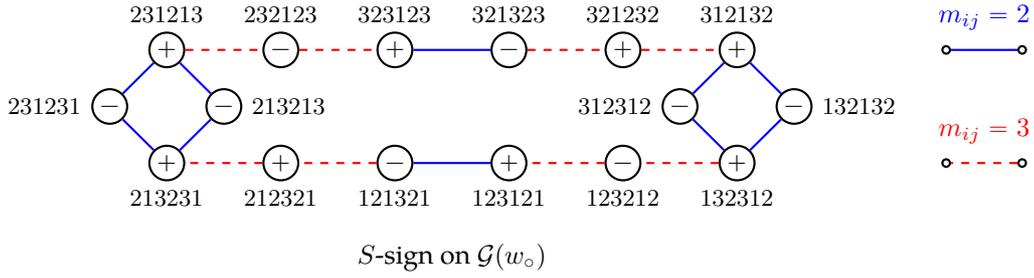
\begin{figure}[H]
\begin{center}
\begin{tikzpicture}[vertex/.style={inner sep=1pt,circle,draw=black,thick},
	    deux/.style={thick,blue},
	    trois/.style={thick,dashed,red}]

\coordinate (center1) at (0,-2);
\node at (center1) {$S$-sign on $\G(\wo)$};

\node[vertex,label=below:{\small $123121$}] (123121) at (0.75,-0.75) {$+$};
\node[vertex,label=below:{\small $123212$}] (123212) at (2.25,-0.75) {$-$} edge[trois] (123121);
\node[vertex,label=below:{\small $132312$}] (132312) at (3.75,-0.75) {$+$} edge[trois] (123212);
\node[vertex,label=above:{\small $312132$}] (312132) at (3.75,0.75) {$+$};
\node[vertex,label=above:{\small $321232$}] (321232) at (2.25,0.75) {$+$} edge[trois] (312132);
\node[vertex,label=above:{\small $321323$}] (321323) at (0.75,0.75) {$-$} edge[trois] (321232);
\node[vertex,label=above:{\small $323123$}] (323123) at (-0.75,0.75) {$+$} edge[deux] (321323);
\node[vertex,label=above:{\small $232123$}] (232123) at (-2.25,0.75) {$-$} edge[trois] (323123);
\node[vertex,label=above:{\small $231213$}] (231213) at (-3.75,0.75) {$+$} edge[trois] (232123);
\node[vertex,label=below:{\small $213231$}] (213231) at (-3.75,-0.75) {$+$};
\node[vertex,label=below:{\small $212321$}] (212321) at (-2.25,-0.75) {$+$} edge[trois] (213231);
\node[vertex,label=below:{\small $121321$}] (121321) at (-0.75,-0.75) {$-$} edge[trois] (212321) edge[deux] (123121);

\node[vertex,label=left:{\small $312312$}] (312312) at (3,0) {$-$} edge[deux] (132312) edge[deux] (312132);
\node[vertex,label=right:{\small $132132$}] (132132) at (4.5,0) {$-$} edge[deux] (132312) edge[deux] (312132);
\node[vertex,label=right:{\small $213213$}] (213213) at (-3,0) {$-$} edge[deux] (213231) edge[deux] (231213);
\node[vertex,label=left:{\small $231231$}] (231231) at (-4.5,0) {$-$} edge[deux] (213231) edge[deux] (231213);

\node[vertex] (a) at (6.5,0.75) {};
\node[vertex] (b) at (7.5,0.75) {} edge[deux] node[midway,label=above:${m_{ij}=2}$] {} (a);

\node[vertex] (a) at (6.5,-0.75) {};
\node[vertex] (b) at (7.5,-0.75) {} edge[trois] node[midway,label=above:${m_{ij}=3}$] {} (a);
\end{tikzpicture}
\end{center}
\caption{The $S$-sign for reduced expressions of $\wo$ in type $A_3$}
\label{fig:A3_Ssign}
\end{figure}
\end{example}

\subsection{The punctual sign function}
\label{ssec:st_sign}

The existence of geodesic realizations of subword complexes requires the existence of the $T$-sign function, as first observed in
\cite[Proposition~3.4]{ceballos_associahedra_2012}.
As we have seen in Proposition~\ref{prop:signature_matrix}, the $T$-sign prescribes the orientation of dual simplicial cones in order to get a complete simplicial fan.
The $S$-sign further takes care of the intrinsic ordering related to a word and contributes to determine the orientation of dual simplicial cones. 
We investigate these relations further in Sections~\ref{sec:model_mat} and~\ref{sec:parameter_sign_mat}.
Here, we introduce the \emph{punctual sign function} which considers both signs and give some examples.

\begin{definition}[Punctual sign function $\punc$\ \footnote{By multiplying the sign functions $S$ and $T$, we see the abbreviation ``s.t.'' (lat. {\it sine tempore}), which, in german academic culture, describes events starting punctually. 
We suggest to pronounce the hourglass symbol $\punc(w)$ as ``clock of~$w$''.}]
The \defn{punctual sign function}~$\punc$ is defined
as 
\[
\begin{split}
\punc:\Red(\wo) & \rightarrow \{+1,-1\}\\
w & \mapsto \sigma(w)\cdot\tau(w),
\end{split}
\]
where $\sigma$ is the $S$-sign function on words in $S^*$ and $\tau$ is the $T$-sign function on reduced words~$\Red(\wo)$.
\end{definition}

Since $T$ is defined up to a global multiplication by ``$-1$'', the punctual sign function is also well-defined up to a global multiplication by ``$-1$''.
For this reason, we henceforth fix the $T$-sign of the lexicographically first reduced subword of $\wo$ occuring in $(s_1\cdots s_n)^\infty$ to have positive sign. 
The definition of product of sign functions allows to interpret the values of the $\punc$-sign function: it is positive when the $S$ and $T$ functions are equal, and negative otherwise.
Further, its behavior along braid moves is determined as follows.
Set $w=ub_{i,j}v$ and $w'=ub_{j,i}v$ with $\ell(b_{i,j})=m_{i,j}$ as in Theorem~\ref{thm:A_Sbraid}, then
\[
\punc(w) = 
\begin{cases}
\punc(w') & \text{ if } m_{i,j}\equiv 2\mod 4, \\
-\punc(w') & \text{ if } m_{i,j}\equiv 0\mod 4, \\
(-1)^{\kappa+\mu}\punc(w') & \text{ if } m_{i,j}\equiv 1 \text{ or } 3\mod 4. \\
\end{cases}
\]
Although this definition using modular arithmetic gives some insight on $\punc$, is it not clear whether there is a combinatorial interpretation of $\punc$ not stemming from $\sigma$ and $\tau$.

\begin{example}[Dihedral Group $I_2(m)$]
Let $W=I_2(m)$, with $m\geq 2$.
The punctual sign function for the reduced expressions of $\wo$ is determined by the residue of $m\ \mathrm{mod}\ 4$, see Figure~\ref{fig:I2_STsign}.

\begin{figure}[H]
\begin{center}
\begin{tikzpicture}[vertex/.style={inner sep=1pt,circle,draw=black,thick},
	    deux/.style={thick,blue},
	    trois/.style={thick,dashed,red},
	    quatre/.style={thick,blue,double,double distance=1mm}]

\node at (-4.5,1) {$m\equiv 2\mod 4$};

\node[vertex,label=left:$(s_1s_2)^{\frac{m}{2}}$] (s1) at (-5,0) {};
\node[vertex,label=right:${(s_2s_1)^{\frac{m}{2}}}$] (s2) at (-4,0) {};

\draw[deux] (s1) -- node[midway,label=above:$+$] {} (s2);

\node at (0,1) {$m\equiv 0\mod 4$};

\node[vertex,label=left:$(s_1s_2)^{\frac{m}{2}}$] (t1) at (-0.5,0) {};
\node[vertex,label=right:${(s_2s_1)^{\frac{m}{2}}}$] (t2) at (0.5,0) {};

\draw[quatre] (t1) -- node[midway,label=above:$-$] {} (t2);

\node at (5,1) {$m\equiv 1,3\mod 4$};

\node[vertex,label=left:$(s_1s_2)^{\lfloor\frac{m}{2}\rfloor} s_1$] (u1) at (4.5,0) {};
\node[vertex,label=right:$(s_2s_1)^{\lfloor\frac{m}{2}\rfloor} s_2$] (u2) at (5.5,0) {};

\draw[trois] (u1) -- node[midway,label=above:$+$] {} (u2);

\end{tikzpicture}
\end{center}
\caption{The punctual signs for reduced expressions of the dihedral group $I_2(m)$. Since the punctual sign values varies within the same residue class, we label
the edge by the product of the punctual signs of its vertices which is invariant.}
\label{fig:I2_STsign}
\end{figure}
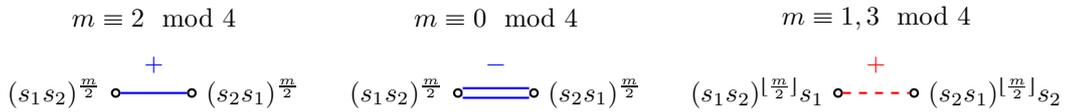
\end{example}

\begin{example}[Symmetric group $\Sym_4=A_3$]
\label{ex:punc_a3}
Let $W=A_3$.
The punctual sign function is illustrated in Figure~\ref{fig:A3_STsign}.
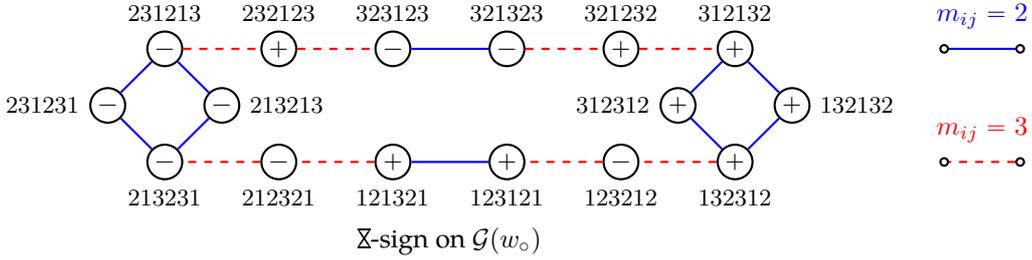
\begin{figure}[H]
\begin{center}
\begin{tikzpicture}[vertex/.style={inner sep=1pt,circle,draw=black,thick},
	    deux/.style={thick,blue},
	    trois/.style={thick,dashed,red}]

\coordinate (center1) at (0,-1.8);
\node at (center1) {$\punc$-sign on $\G(\wo)$};

\node[vertex,label=below:{\small $123121$}] (123121) at (0.75,-0.75) {$+$};
\node[vertex,label=below:{\small $123212$}] (123212) at (2.25,-0.75) {$-$} edge[trois] (123121);
\node[vertex,label=below:{\small $132312$}] (132312) at (3.75,-0.75) {$+$} edge[trois] (123212);
\node[vertex,label=above:{\small $312132$}] (312132) at (3.75,0.75) {$+$};
\node[vertex,label=above:{\small $321232$}] (321232) at (2.25,0.75) {$+$} edge[trois] (312132);
\node[vertex,label=above:{\small $321323$}] (321323) at (0.75,0.75) {$-$} edge[trois] (321232);
\node[vertex,label=above:{\small $323123$}] (323123) at (-0.75,0.75) {$-$} edge[deux] (321323);
\node[vertex,label=above:{\small $232123$}] (232123) at (-2.25,0.75) {$+$} edge[trois] (323123);
\node[vertex,label=above:{\small $231213$}] (231213) at (-3.75,0.75) {$-$} edge[trois] (232123);
\node[vertex,label=below:{\small $213231$}] (213231) at (-3.75,-0.75) {$-$};
\node[vertex,label=below:{\small $212321$}] (212321) at (-2.25,-0.75) {$-$} edge[trois] (213231);
\node[vertex,label=below:{\small $121321$}] (121321) at (-0.75,-0.75) {$+$} edge[trois] (212321) edge[deux] (123121);

\node[vertex,label=left:{\small $312312$}] (312312) at (3,0) {$+$} edge[deux] (132312) edge[deux] (312132);
\node[vertex,label=right:{\small $132132$}] (132132) at (4.5,0) {$+$} edge[deux] (132312) edge[deux] (312132);
\node[vertex,label=right:{\small $213213$}] (213213) at (-3,0) {$-$} edge[deux] (213231) edge[deux] (231213);
\node[vertex,label=left:{\small $231231$}] (231231) at (-4.5,0) {$-$} edge[deux] (213231) edge[deux] (231213);

\node[vertex] (a) at (6.5,0.75) {};
\node[vertex] (b) at (7.5,0.75) {} edge[deux] node[midway,label=above:${m_{ij}=2}$] {} (a);

\node[vertex] (a) at (6.5,-0.75) {};
\node[vertex] (b) at (7.5,-0.75) {} edge[trois] node[midway,label=above:${m_{ij}=3}$] {} (a);
\end{tikzpicture}
\end{center}
\caption{The punctual sign function for reduced expressions of the group $A_3$}
\label{fig:A3_STsign}
\end{figure}
\end{example}

\begin{example}[Hyperoctahedral group $B_3$]
Let $W=B_3$.
The punctual sign function is illustrated in Figure~\ref{fig:B3_STsign}.
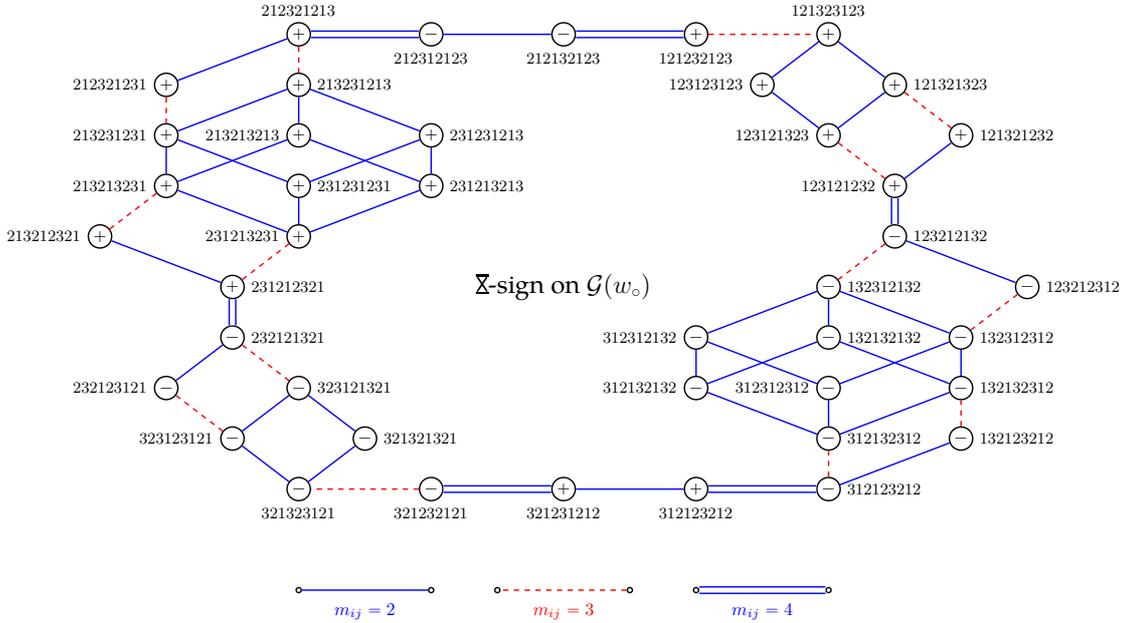
\begin{figure}[!ht]
\resizebox{1\hsize}{!}{
\begin{tikzpicture}[vertex/.style={inner sep=1pt,circle,draw=black,thick},
	    deux/.style={thick,blue},
	    trois/.style={thick,dashed,red},
	    quatre/.style={thick,blue,double,double distance=1mm}]

\def\lgx{2.6}
\def\lgy{1}

\coordinate (center1) at (0,0);
\node at (center1) {{\LARGE $\punc$-sign on $\G(\wo)$}};

\node[vertex] (a) at (-2*\lgx,-6*\lgy) {};
\node[vertex] (b) at (-1*\lgx,-6*\lgy) {} edge[deux] node[midway,label=below:${m_{ij}=2}$] {} (a);

\node[vertex] (c) at (-0.5*\lgx,-6*\lgy) {};
\node[vertex] (d) at (0.5*\lgx,-6*\lgy) {} edge[trois] node[midway,label=below:${m_{ij}=3}$] {} (c);

\node[vertex] (e) at (1*\lgx,-6*\lgy) {};
\node[vertex] (f) at (2*\lgx,-6*\lgy) {} edge[quatre] node[midway,label=below:${m_{ij}=4}$] {} (e);

% Lower right cube:
\node[vertex,label=right:{\small $312132312$}] (r_cu_0) at (2*\lgx,-3*\lgy) {$-$};
\node[vertex,label=left :{\small $312132132$}] (r_cu_11) at (\lgx,-2*\lgy) {$-$}
edge[deux] (r_cu_0);
\node[vertex,label=left :{\small $312312312$}] (r_cu_12) at (2*\lgx,-2*\lgy) {$-$}
edge[deux] (r_cu_0);
\node[vertex,label=right:{\small $132132312$}] (r_cu_13) at (3*\lgx,-2*\lgy) {$-$}
edge[deux] (r_cu_0);
\node[vertex,label=left :{\small $312312132$}] (r_cu_21) at (\lgx,-1*\lgy) {$-$}
edge[deux] (r_cu_11)
edge[deux] (r_cu_12);
\node[vertex,label=right:{\small $132132132$}] (r_cu_22) at (2*\lgx,-1*\lgy) {$-$}
edge[deux] (r_cu_11)
edge[deux] (r_cu_13);
\node[vertex,label=right:{\small $132312312$}] (r_cu_23) at (3*\lgx,-1*\lgy) {$-$}
edge[deux] (r_cu_12)
edge[deux] (r_cu_13);
\node[vertex,label=right:{\small $132312132$}] (r_cu_3) at (2*\lgx,0) {$-$}
edge[deux] (r_cu_21)
edge[deux] (r_cu_22)
edge[deux] (r_cu_23);
\node[vertex,label=right:{\small $123212132$}] (r_cu_41) at (2.5*\lgx,1) {$-$}
edge[trois] (r_cu_3);
\node[vertex,label=right:{\small $123212312$}] (r_cu_32) at (3.5*\lgx,0) {$-$}
edge[trois] (r_cu_23)
edge[deux]  (r_cu_41);
\node[vertex,label=right:{\small $132123212$}] (r_cu_02) at (3*\lgx,-3) {$-$}
edge[trois] (r_cu_13);
\node[vertex,label=right:{\small $312123212$}] (r_cu_m1) at (2*\lgx,-4) {$-$}
edge[deux] (r_cu_02)
edge[trois]  (r_cu_0);

% Upper right square:
\node[vertex,label=left:{\small $123121232$}] (r_sq_m1)  at (2.5*\lgx,2*\lgy) {$+$}
edge[quatre] (r_cu_41);
\node[vertex,label=left:{\small $123121323$}] (r_sq_0)  at (2*\lgx,3*\lgy) {$+$}
edge[trois] (r_sq_m1);
\node[vertex,label=right:{\small $121321232$}] (r_sq_02)  at (3*\lgx,3*\lgy) {$+$}
edge[deux] (r_sq_m1);
\node[vertex,label=left :{\small $123123123$}] (r_sq_11)  at (1.5*\lgx,4*\lgy) {$+$}
edge[deux] (r_sq_0);
\node[vertex,label=right:{\small $121321323$}] (r_sq_12)  at (2.5*\lgx,4*\lgy) {$+$}
edge[deux] (r_sq_0)
edge[trois] (r_sq_02);
\node[vertex,label=above:{\small $121323123$}] (r_sq_2)  at (2*\lgx,5*\lgy) {$+$}
edge[deux] (r_sq_11)
edge[deux] (r_sq_12);

% Upper left cube:
\node[vertex,label=right:{\small $213231213$}] (l_cu_0)  at (-2*\lgx,4*\lgy) {$+$};
\node[vertex,label=right:{\small $231231213$}] (l_cu_11) at (-\lgx,3*\lgy) {$+$}
edge[deux] (l_cu_0);
\node[vertex,label=left :{\small $213213213$}] (l_cu_12) at (-2*\lgx,3*\lgy) {$+$}
edge[deux] (l_cu_0);
\node[vertex,label=left:{\small $213231231$}] (l_cu_13) at (-3*\lgx,3*\lgy) {$+$}
edge[deux] (l_cu_0);
\node[vertex,label=right:{\small $231213213$}] (l_cu_21) at (-\lgx,2*\lgy) {$+$}
edge[deux] (l_cu_11)
edge[deux] (l_cu_12);
\node[vertex,label=right:{\small $231231231$}] (l_cu_22) at (-2*\lgx,2*\lgy) {$+$}
edge[deux] (l_cu_11)
edge[deux] (l_cu_13);
\node[vertex,label=left :{\small $213213231$}] (l_cu_23) at (-3*\lgx,2*\lgy) {$+$}
edge[deux] (l_cu_12)
edge[deux] (l_cu_13);
\node[vertex,label=left :{\small $231213231$}] (l_cu_3)  at (-2*\lgx,1) {$+$}
edge[deux] (l_cu_21)
edge[deux] (l_cu_22)
edge[deux] (l_cu_23);
\node[vertex,label=right:{\small $231212321$}] (l_cu_41) at (-2.5*\lgx,0) {$+$}
edge[trois] (l_cu_3);
\node[vertex,label=left :{\small $213212321$}] (l_cu_32) at (-3.5*\lgx,1) {$+$}
edge[trois] (l_cu_23)
edge[deux]  (l_cu_41);
\node[vertex,label=left :{\small $212321231$}] (l_cu_02) at (-3*\lgx,4) {$+$}
edge[trois] (l_cu_13);
\node[vertex,label=above:{\small $212321213$}] (l_cu_m1) at (-2*\lgx,5) {$+$}
edge[deux] (l_cu_02)
edge[trois]  (l_cu_0);

% Lower left square:
\node[vertex,label=right:{\small $323121321$}] (l_sq_0)  at (-2*\lgx,-2*\lgy) {$-$};
\node[vertex,label=left :{\small $323123121$}] (l_sq_11)  at (-2.5*\lgx,-3*\lgy) {$-$}
edge[deux] (l_sq_0);
\node[vertex,label=right:{\small $321321321$}] (l_sq_12)  at (-1.5*\lgx,-3*\lgy) {$-$}
edge[deux] (l_sq_0);
\node[vertex,label=below:{\small $321323121$}] (l_sq_22)  at (-2*\lgx,-4*\lgy) {$-$}
edge[deux] (l_sq_11)
edge[deux] (l_sq_12);
\node[vertex,label=left:{\small $232123121$}] (l_sq_02)  at (-3*\lgx,-2*\lgy) {$-$}
edge[trois] (l_sq_11);
\node[vertex,label=right:{\small $232121321$}] (l_sq_m1)  at (-2.5*\lgx,-1*\lgy) {$-$}
edge[quatre] (l_cu_41)
edge[trois] (l_sq_0)
edge[deux] (l_sq_02);

% Lower vertices:
\node[vertex,label=below:{\small $321232121$}] (b_1)  at (-1*\lgx,-4*\lgy) {$-$}
edge[trois] (l_sq_22);
\node[vertex,label=below:{\small $321231212$}] (b_2)  at (-0*\lgx,-4*\lgy) {$+$}
edge[quatre] (b_1);
\node[vertex,label=below:{\small $312123212$}] (b_3)  at (1*\lgx,-4*\lgy) {$+$}
edge[deux] (b_2)
edge[quatre] (r_cu_m1);

% Upper vertices:
\node[vertex,label=below:{\small $212312123$}] (u_1)  at (-1*\lgx,5*\lgy) {$-$}
edge[quatre] (l_cu_m1);
\node[vertex,label=below:{\small $212132123$}] (u_2)  at (-0*\lgx,5*\lgy) {$-$}
edge[deux] (u_1);
\node[vertex,label=below:{\small $121232123$}] (u_3)  at (1*\lgx,5*\lgy) {$+$}
edge[quatre] (u_2)
edge[trois] (r_sq_2);

\end{tikzpicture}
}
\caption{The punctual sign function on reduced words of $\wo$ in type $B_3$.}
\label{fig:B3_STsign}
\end{figure}
\end{example}

\section{Model matrices}
\label{sec:model_mat}

In this section, we give a factorization formula for the determinant of matrices
\[
\Big(f_{i,j}(x_j)\Big)_{i,j\in[k]},
\]
where $f_{i,j}(x_j)$ is a polynomial in $\R[x_j]$ based on the Binet--Cauchy formula.
Matrices of this form include the so-called \defn{alternant matrices}.
They were considered already in the XIX$^{\text{th}}$ century, if not earlier, see \cite[Chapter 7, Notes]{stanley_enumerative_1999}.
The case when $f_{i,j}$ does not depend on the index $j$ (equivalently, if interchanging variables is equivalent to permuting the columns) is classical \cite[Chapter~6]{aitken_determinants_1939} and \cite[Chapter~XI]{muir_treatise_1960}.
The Vandermonde matrix is the case when $f_{i,j}(x_j)=x_j^{i-1}$, for ${i,j\in[k]}$.
Since all the columns of the Vandermonde matrix are equal up to change of variables, its determinant is an alternating polynomial in the variables $x_1,\dots, x_k$ with respect to the group action of~$\Sym_k$ by permuting the variable indices.
As Definition~\ref{def:schur} shows, the same holds for the columns of the matrix used to define Schur functions, leading to the fact that its determinant is divisible by the Vandermonde determinant and thus the quotient is a symmetric polynomial in the variables $x_1,\dots,x_k$. 
At the opposite end, if no two columns are equal up to a change of variables, no non-trivial permutation action acts canonically on the determinant.

Having the construction of signature matrices in mind, we are particularly interested in the case when subsets of columns are equal up to a change of variables depending on a reduced word.
When the columns are partitioned into subsets of columns that are equal up to permuting the variable indices, we get a ``partially'' symmetric polynomial expressible as a product of symmetric polynomials: Theorem~\ref{thm:B_det} in Section~\ref{ssec:formula_param} provides an explicit formula for the determinant in Definition~\ref{def:signature_matrix} while Corollary~\ref{cor:D_sign_of_model} reveals the structure of the sign of the determinant using the $S$-sign function.
Additionnally, Example~\ref{ex:dual_cauchy} shows how to recover the dual Cauchy identity from Theorem~\ref{thm:B_det}

\subsection{Definitions}

In order to study matrices with polynomial entries, we define certain tensors.
They allow to dissect the data into smaller pieces, that are then easier to control and analyze as done in Sections~\ref{ssec:vandermonde} and \ref{ssec:partial_schur}.
The variables tensor is used to provide polynomials of degree at most $d-1$ in a $(N\times N)$-matrix: 

\begin{definition}[Variables tensor]
Let $d$ and $N$ be positive integers.
The \defn{variables tensor} $\T^{k,j}{}_{l}(d,N)$ is the tensor in $V_d\otimes V_N \otimes {V_N}^*$ over $\R[x_1,\dots,x_N]$ defined as
\[
\T^{k,j}{}_{l}(d,N):=\sum_{j=1}^N \sum_{l=j}^{j}\left(\sum_{k=1}^{d}x_l^{k-1}\right) \mathbf{e}^k\otimes \mathbf{e}^j \otimes \mathbf{f}_{l}.
\]
\end{definition}

\noindent
The parameter tensor encodes the coefficients of the polynomials that appear in the matrix whose determinant we aim to determine:

\begin{definition}[Parameter tensor]
Let $d$ and $N$ be positive integers and $S$ be an alphabet of cardinality~$n$.
A \defn{parameter tensor} $\PP^{i}{}_{s,k}(N,n,d)$ is a tensor in $V_N\otimes {V_n}^* \otimes {V_d}^*$ over $\R$.
\end{definition}

For practical reasons, we index the rank-$1$ tensors of $V_N\otimes {V_n}^* \otimes {V_d}^*$ with the set $[N] \times S \times \{0,1,\dots,{d-1}\}$.
In particular, the columns of a parameter tensor are indexed by couples in $S\times \{0,1,\dots,{d-1}\}$.

\begin{definition}[Coefficients tensor of a word]
Let $d\geq 1$, $v=v_1v_2\dots v_N$ be a word in~$S^*$, and $\PP=(p^i{}_{s,j})_{(i,s,j)\in [N]\times S\times \{0,\dots,d-1\}}$ be a parameter tensor.
The \defn{coefficients tensor} of~$v$ with respect to~$\PP$ is the tensor in $V_N\otimes {V_N}^* \otimes {V_{d}}^*$ defined as

\[
\C^i{}_{j,k}(v,\PP):=\sum_{i=1}^{N}\sum_{j=1}^N\sum_{k=1}^{d}p^i{}_{v_j,k-1}\ \mathbf{e}^i\otimes \mathbf{f}_j\otimes \mathbf{f}_k.
\]

\end{definition}

Similarly, the columns of a coefficients tensor are indexed by couples in $[N]\times \{0,1,\dots,{d-1}\}$.
Multiplying the coefficients tensor with the variables tensor and flattening the product, we get a matrix that models square matrices where certain groups of columns are equal
up to a relabeling of variables, according to occurrences of letters in the chosen word $v$.

\begin{definition}[Model matrix of a word]
Let $d\geq 1$, $v=v_1v_2\dots v_N$ be a word in $S^*$, and $\PP=(p^i{}_{s,j})_{(i,s,j)\in [N]\times S\times \{0,\dots,d-1\}}$ be a parameter tensor.
Denote by $\R[\PP]$ the real polynomial ring whose variables are the non-zero coefficients of $\PP$ (when considering the entries of $\PP$ are real variables).
The \defn{model matrix} of~$v$ with respect to $\PP$ is the $(N\times N)$-matrix
\[
M^{i}{}_{l}(v,\PP):=\C^i{_{j,k}}(v,\PP)\cdot \T^{k,j}{}_{l}(d,N),
\]
whose entries in column $l$ are polynomials of degree $d-1$ in the variable
$x_l$ with coefficients taken from the parameter tensor $\PP$ with second index $v_l$.
\end{definition}

The entries of $M(v,\PP)$ in column $l$ are polynomials in $(\R[\PP|_{v_l}])[x_l]$, where $\PP|_{v_l}$ restricts $\PP$ to the subtensor indexed by $v_l$.
Further, whenever $v_i=v_j$ and $i\neq j$, the columns $i$ and $j$ are equal up to relabeling their variables.
We have already seen such examples: the Vandermonde matrix in Section~\ref{ssec:vandermonde}, and in Example~\ref{ex:partition}.
Here is another example that we examine further later on.

\begin{example}
\label{ex:model4}
Consider the matrix
\[
M = 
\left(\begin{array}{rrrr}
1 & 0 & 1 & 0 \\
0 & 1 & 0 & 1 \\
-x_{1} & x_{2} & -x_{3} & x_{4} \\
x_{1}^{2} & -x_{2}^{2} & x_{3}^{2} & -x_{4}^{2}
\end{array}\right).
\]
In this case, $d-1=2$, $N=4$, and $v=s_1s_2s_1s_2$.
The corresponding parameter tensor is
\[
\PP:=   (\mathbf{e}^1 \otimes \mathbf{f}_{s_1} \otimes \mathbf{f}_0)
- (\mathbf{e}^3 \otimes \mathbf{f}_{s_1} \otimes \mathbf{f}_1)
+ (\mathbf{e}^4 \otimes \mathbf{f}_{s_1} \otimes \mathbf{f}_2)
+ (\mathbf{e}^2 \otimes \mathbf{f}_{s_2} \otimes \mathbf{f}_0)
+ (\mathbf{e}^3 \otimes \mathbf{f}_{s_2} \otimes \mathbf{f}_1)
- (\mathbf{e}^4 \otimes \mathbf{f}_{s_2} \otimes \mathbf{f}_2).
\]
The matrix can be written using the corresponding coefficients and variables matrices as:
\[
\resizebox{\hsize}{!}{$
M = \bigoplus_{k=1}^2
\overbrace{
\begin{blockarray}{cccccc}
\begin{block}{cccccc}
(s_1,0) & (s_1,1) & (s_1,2) & (s_2,0) & (s_2,1) & (s_2,2) \\ 
\end{block}
\begin{block}{(ccc|ccc)}
1 & 0 & 0 & 0 & 0 & 0 \\
0 & 0 & 0 & 1 & 0 & 0 \\
0 & -1 & 0 & 0 & 1 & 0 \\
0 & 0 & 1 & 0 & 0 & -1 \\
\end{block}
\end{blockarray}}^{\PP}
\hspace{0.2cm}\times
\left(\begin{array}{rrrrrrrrrrrr}
1 & x_{1} & x_{1}^{2} & 0 & 0 & 0 & 0 & 0 & 0 & 0 & 0 & 0 \\
0 & 0 & 0 & 1 & x_{2} & x_{2}^{2} & 0 & 0 & 0 & 0 & 0 & 0 \\
0 & 0 & 0 & 0 & 0 & 0 & 1 & x_{3} & x_{3}^{2} & 0 & 0 & 0 \\
0 & 0 & 0 & 0 & 0 & 0 & 0 & 0 & 0 & 1 & x_{4} & x_{4}^{2} \\
\end{array}\right)^{\top},$}
\]
where the $\oplus$ indicates columnwise concatenation of matrices.
\end{example}

\subsection{Model matrices for reduced words}

For the remainder of Section~\ref{sec:model_mat}, we present the results with the combinatorics of Coxeter groups in mind.
The general result about the factorization of determinants of matrices of polynomials can be deduced directly by removing the restrictions coming from the Coxeter group in play.
We remind the reader of the definitions of the following objects used throughout this section. 

\begin{center}
\begin{tabular}{rcl}
$(W,S)$ & := & a finite irreducible Coxeter system, \\
$n$ & := & $\# S$, the cardinality of $S$, \\
$N$ & := & $\ell(\wo)$, the length of the longest element,\\
$\nu$ & := & the h\"ochstfrequenz of $W$, defined in Section~\ref{ssec:cox_groups} on page~\pageref{def:hf}.
\end{tabular}
\end{center}

\begin{definition}[Variables tensor of a Coxeter system]
The \defn{variables tensor} of $(W,S)$ is the variables tensor $\T_W:=\T^{k,j}{}_{l}(\nu,N)$.
\end{definition}

\begin{definition}[Model matrix of a reduced word]
Let $\PP=(p^i{}_{s,j})_{(i,s,j)\in [N]\times S\times \{0,\dots,\nu-1\}}$ be a parameter tensor and $v=v_1v_2\dots v_N\in\Red(\wo)$.
The \defn{model matrix} of $v$ with respect to $\PP$ is the $(N\times N)$-matrix
\[
M^{i}{}_{l}(v,\PP):=\C^i{_{j,k}}(v,\PP)\cdot \T_W.
\]
\end{definition}

\begin{example}[Symmetric group $\Sym_3=A_2$]
\label{ex:model_mat_a2}
We have $n=2$, $\nu=2$, $\Red(\wo)=\{s_1s_2s_1,s_2s_1s_2\}$, and $N=3$.
The model matrix for $s_1s_2s_1$ is
\[
M(s_1s_2s_1,\PP)=
\left(\begin{array}{rrr}
p^{1}{}_{s_1,0} + p^{1}{}_{s_1,1}x_1 & p^{1}{}_{s_2,0} + p^{1}{}_{s_2,1}x_2 & p^{1}{}_{s_1,0} + p^{1}{}_{s_1,1}x_3 \\
p^{2}{}_{s_1,0} + p^{2}{}_{s_1,1}x_1 & p^{2}{}_{s_2,0} + p^{2}{}_{s_2,1}x_2 & p^{2}{}_{s_1,0} + p^{2}{}_{s_1,1}x_3 \\
p^{3}{}_{s_1,0} + p^{3}{}_{s_1,1}x_1 & p^{3}{}_{s_2,0} + p^{3}{}_{s_2,1}x_2 & p^{3}{}_{s_1,0} + p^{3}{}_{s_1,1}x_3
\end{array}\right).
\]
The model matrix for $s_2s_1s_2$ is
\[
M(s_2s_1s_2,\PP)=
\left(\begin{array}{rrr}
p^{1}{}_{s_2,0} + p^{1}{}_{s_2,1}x_1 & p^{1}{}_{s_1,0} + p^{1}{}_{s_1,1}x_2 & p^{1}{}_{s_2,0} + p^{1}{}_{s_2,1}x_3 \\
p^{2}{}_{s_2,0} + p^{2}{}_{s_2,1}x_1 & p^{2}{}_{s_1,0} + p^{2}{}_{s_1,1}x_2 & p^{2}{}_{s_2,0} + p^{2}{}_{s_2,1}x_3 \\
p^{3}{}_{s_2,0} + p^{3}{}_{s_2,1}x_1 & p^{3}{}_{s_1,0} + p^{3}{}_{s_1,1}x_2 & p^{3}{}_{s_2,0} + p^{3}{}_{s_2,1}x_3
\end{array}\right).
\]
\end{example}

\subsection{Binet--Cauchy on model matrices}
\label{ssec:bin_cau_model}

We use Binet--Cauchy's formula \eqref{eq:cauchy_binet} from Section~\ref{ssec:vandermonde} to give a
description of the determinants of model matrices.
Before giving a first description, we set some bookkeeping notations and give two lemmas.
Let $Z\subseteq\{0,\dots,\nu N-1\}$ and $\# Z=N$, we write
\[
Z=\{z_1,z_2,\dots,
z_N\},
\]
such that $z_1<z_2\dots<z_N$, and
\[
z_i=q_i\nu+r_i, \text{ with } 0\leq r_i<\nu, \text{ for all } i\in[N].
\]
We use the set $Z$ to index columns of the coefficients tensor and the rows of
the variables tensor. The correspondence between indices in $Z$ and columns 
of the coefficients tensor is described as follows.
Given a couple $(j,k)\in[N]\times \{0,1,\dots,\nu-1\}$ indexing a column of the coefficients tensor, we define $q:=j-1$, $r:=k$, and $z:=q\nu+r$.
This way, the couple $(j,k)$ corresponds to a unique element $z$ in $\{0,\dots,\nu N-1\}$, and vice-versa.
The correspondence with rows of the variables tensor works similarly.
The following two lemmas can be checked using the definition of parameter and variables tensors and properties of the determinant.

\begin{lemma}\label{lem:det_T}
Let $Z\subseteq\{0,\dots,\nu N-1\}$ with $\# Z=N$.
The determinant of the variables tensor~$\T_W$ restricted to the rows in $Z$ is
\[
\det [\T_W]_Z=\begin{cases} 0 & \text{ if } q_i= q_j \text{ for some } i\neq j, \\
x_1^{r_1}\cdots x_N^{r_N} & \text{ else}.
\end{cases}
\]
\end{lemma}

\begin{lemma}
\label{lem:det_coeff}
Let $\PP$ be a parameter tensor for $(W,S)$, $v=v_1v_2\cdots v_N\in\Red(\wo)$, and $Z\subseteq\{0,\dots,\nu N-1\}$ with $\# Z=N$.
If $r_i=r_j$ and $v_i= v_j$ with $1\leq i < j\leq N$, then the determinant of
the coefficients tensor $\C(v,\PP)$ restricted to the columns in $Z$ is 0.
\end{lemma}

Given a reduced expression $v=v_1v_2\cdots v_N\in\Red(\wo)$, the previous lemmas lead to the definition of the following collection of $N$-subsets of $\{0,\dots,\nu N-1\}$:
\[
\ZZ_v:= \{Z\subset \{0,\dots,\nu N-1\} : \# Z=N,\ q_i\neq q_j\ \text{for all } i\neq j, \text{ and if } v_i= v_j, \text{ then } r_i\neq r_j\}.
\]
The subsets in $\ZZ_v$ are precisely those whose summand are not implied to be equal to zero in the Binet--Cauchy formula for the determinant of the model matrix.
The following proposition is a consequence of Lemmas~\ref{lem:binetcauchy},~\ref{lem:det_T}, and~\ref{lem:det_coeff}, and is improved in Theorem~\ref{thm:B_det}.

\begin{proposition}\label{prop:determinant_v1}
Let $\PP$ be a parameter tensor for $(W,S)$, and $v\in\Red(\wo)$.
The determinant of the model matrix $\M(v,\PP)$ is
\[
\det \M(v,\PP)=\sum_{Z\in \ZZ_v} \det [\C(v,\PP)]_Z \cdot x_1^{r_1}\cdots x_N^{r_N}.
\]
\end{proposition}

\subsection{Symmetric formula for the determinant of model matrices via parameter matrices}
\label{ssec:formula_param}

We proceed to express $\det \M(v,\PP)$ for some reduced word $v\in\Red(\wo)$ in terms of minors of $\PP$, i.e. square submatrices of maximal size.
We also interchangeably use the term minor to refer to the determinant of such submatrices.
We begin by a bijection to encode the columns appropriately using a permutation and a tuple of partitions, which allows the usage of Schur functions.
Given a set of indices $Z\in\ZZ_v$, if $v_i= v_j$ and $i\neq j$, then $r_i\neq r_j$.
Consequently, the matrix $[\C(v,\PP)]_Z$ formed by concatenating the columns in $Z$ increasingly is equal to a column permutation~$\pi_Z$ of the matrix $[\PP]_{\ZZbij}$ of $\PP$ formed by concatenating the columns in the set
\[
\ZZbij:=\{(s_i,r_j) :  j\in[N],\ v_j= s_i \text{ and } z_j=q_j\nu+r_j\in Z\}
\]
increasingly with respect to the lexicographic order.
Given the abelian vector $\abel_v=(c_i)_{s_i\in S}$ of $v$, this motivates the definition of the following collection of subsets of columns of $\PP$:
\[
\mathfrak{Z}_{\abel_v} := 
      \{\ZZbij\subseteq S\times \{0,\dots, \nu-1\} : \ZZbij \text{ contains exactly } c_i \text{ elements } (s_i,\cdot), \forall s_i\in S\}.
\]
The collection $\mathfrak{Z}_{\abel_v}$ describes precisely the minors of $\PP$ that have exactly $c_i$ columns of type $(s_i,\cdot)$ and, as we will see, only these minors matter for the factorization of $\det M(v,\PP)$.
Given some set of indices ${Z=\{z_i=q_i\nu+r_i\}_{i=1}^N}$, for each letter $s_i\in S$, we write the values $r_j$ where $v_j=s_i$ in an ordered list $R_i$ of length $c_i$ created by scanning the values of $r_j$ from $1$ to $N$ and keeping only those where $v_j=s_i$:
\[
R_i := [r_j: \text{if } v_j=s_i]_{j=1}^N.
\]
Since all entries in $R_i$ are distinct, the list $R_i$ corresponds canonically to a permutation $\pi_i$ in $\Sym_{\{v\}_i}$ via their relative order.
The permutations $\{\pi_i\}_{i=1}^n$ act on disjoint sets and can be seen as permutations in $\Sym_{N}$, so we define $\pi_Z:=\pi_1\cdots \pi_n\in\Sym_N$.
Observe that the permutation $\pi_Z\in\Sym_N$ is such that $\pi\cdot v = v$.
Further, the map
\[
\begin{split}
\ZZ_v & \rightarrow \mathfrak{Z}_{\abel_v}\times \prod_{i=1}^n\Sym_{\{v\}_i}\\
Z & \mapsto (\ZZbij,\pi_Z)
\end{split}
\]
is a bijection.

\begin{lemma}
\label{lem:reorder_Z}
Let $v\in\Red(\wo)$.
If $Z\in\ZZ_v$, then
\[
\det[\C(v,\PP)]_Z=\sigma(v)\sigma(\pi_Z)\det [\PP]_{\ZZbij}=\sigma(v)\sign(\pi_Z)\det [\PP]_{\ZZbij}.
\]
\end{lemma}

\begin{proof}
From the definition of $\C(v,\PP)$ and $\ZZ_v$, the matrix $[\C(v,\PP)]_Z$ is a permutation of the columns of a column-submatrix of $\PP$.
The column of $\C(v,\PP)$ indexed by $z_j=q_j\nu + r_j$ correspond to the column of $\PP$ indexed by the ordered pair $(s_i,r_j)$ where $v_j=s_i$.
We use the latter labeling to obtain the permutation of the columns of $[\C(v,\PP)]_Z$ in two steps as follows.
First permute the columns using $\pi_Z^{-1}$; the index of the first column correspond to the relabeled index $\pi_Z^{-1}(1)$, and the index of the last column correspond to the relabeled index $\pi_Z^{-1}(N)$.
This permutation orders increasingly the labels $r_j$ while keeping the labels~$s_i$ unchanged.
Then, permute the columns using $\std(v)^{-1}$.
Since the standard permutation of $v$ has shortest length, it does not change the ordering whenever two columns have the same first-label coordinate.
\end{proof}

Before giving the factorization formula, we give two last definitions.
The first one is related to a common bijection between subsets $\binom{[n]}{k}$ and partitions with exactly $k$ parts (that may be empty) of size at most $n-k$.

\begin{definition}[Standard partitions $\Lambda_\ZZbij$]
Let $v\in\Red(\wo)$ with abelian vector $\abel_v=(c_i)_{s_i\in S}$ and $\ZZbij\in \mathfrak{Z}_{\abel_v}$.
For $i\in[n]$, order decreasingly the elements of~$R_i$ obtained from $\ZZbij$ and substract $c_i-j$ to the element at position~$j$ (starting at $j=1$) to obtain the \defn{standard partition} $\lambda^{\ZZbij,i}$.
The sequence of partitions $\Lambda_\ZZbij$ is defined as $(\lambda^{\ZZbij,i})_{i=1}^n$.
\end{definition}

\begin{example}
Let $W=A_2$.
If $v=s_1s_2s_1$, then $\abel_v=(2,1)$ and
\begin{align*}
\mathfrak{Z}_{\abel_v} & = \left\{\ZZbij_1=\{(s_1,0),(s_1,1),(s_2,0)\},\ZZbij_2=\{(s_1,0),(s_1,1),(s_2,1)\}\right\},\\
\ZZ_v & = \left\{\{0,2,5\},\{0,3,5\},\{1,2,4\},\{1,3,4\}\right\}.
\end{align*}
For $Z=\{1,3,4\}=\{0\cdot2+1,1\cdot2+1,2\cdot2+0\}$, we get $R_1=[1,0]$ and $R_2=[1]$ so that $\pi_1=[3,1]\in\Sym_{\{1,3\}}$, $\pi_2=[2]\in\Sym_{\{2\}}$ and $\pi_Z=[3,2,1]$.
So $Z$ is sent to $(\ZZbij_2,\pi_Z)$ and the other three pairs are obtained similarly for the other three elements of $\ZZ_v$.
For $\ZZbij_1$, ordering $R_1$ and $R_2$ decreasingly gives $[1,0]$ and $[0]$, leading to the standard partitions $((0,0),(0))$.
For $\ZZbij_2$, ordering $R_1$ and $R_2$ decreasingly gives $[1,0]$ and $[1]$, leading to the standard partitions $((0,0),(1))$.
\end{example}

\begin{definition}[Vandermonde divisor]
Given $v\in\Red(\wo)$ with abelian vector $\abel_v=(c_i)_{s_i\in S}$, let
\[ \V(v):=\prod_{\substack{s_i\in S \\
c_i\geq 2}}\prod_{\substack{v_j= v_k=s_i \\ j<k}} (x_k-x_j)
\]
be the \defn{Vandermonde divisor} of $v$.
The degree of $\V(v)$ is $\sum_{s_i\in S} \binom{c_i}{2}$.
\end{definition}

\begin{mainthm}
\label{thm:B_det}
Let $\PP$ be a parameter tensor for a Coxeter system $(W,S)$, and $v\in\Red(\wo)$.
The determinant of the model matrix $\M(v,\PP)$ of $v$ with respect to $\PP$ is the multivariate polynomial 
\[
\det M(v,\PP) = \sigma(v) 
	 \V(v)
	 \sum_{\ZZbij\in\mathfrak{Z}_{\abel_v}}
	 \det [\PP]_{\ZZbij}
	 \Sh_{\Lambda_{\ZZbij},\Omega_v},
\]
where $\Lambda_{\ZZbij}=(\lambda^{\ZZbij,1},\dots,\lambda^{\ZZbij,n})$, $\Omega_v=(\{v\}_1,\dots,\{v\}_n)$, and $\Sh_{\Lambda_{\ZZbij},\Omega_v}$ is the partial Schur function with respect to $\Lambda_{\ZZbij}$ and $\Omega_v$, as defined in Section~\ref{ssec:partial_schur}.
\end{mainthm}

\begin{proof}
Let $v=v_1v_2\cdots v_N$, with $v_i\in S$.
If $v_j=v_k=s_i$ for some $s_i\in S$ and $j\neq k$, then setting $x_k=x_j$ in the matrix $M(v,\PP)$ makes its determinant vanish.
Consequently, $(x_k-x_j)$ divides the determinant by Hilbert's Nullstellensatz, see e.g. \cite[Lemma~3.3]{humphreys_reflection_1990}.
Hence, we know that~$\V(v)$ divides $\det M(v,\PP)$, and it remains to determine the quotient of the division.

Yet by Proposition~\ref{prop:determinant_v1},
\[
\det M(v,\PP)=\sum_{Z\in \ZZ_v} \det[\C(v,\PP)]_Z \cdot x_1^{r_1}\cdots x_N^{r_N}.
\]
By Lemma~\ref{lem:reorder_Z}
\[
\det[\C(v,\PP)]_Z=\sigma(v)\sigma(\pi_Z)\det [\PP]_{\ZZbij},
\]
and $M(v,\PP)$ now becomes
\[
\begin{split}
\det M(v,\PP) & = \sum_{Z\in
\ZZ_v}\sigma(v)\sigma(\pi_Z)\det [\PP]_{\ZZbij}\cdot x_1^{r_1}\cdots
x_N^{r_N}, \\
& = \sigma(v)\sum_{Z\in \ZZ_v}\det [\PP]_{\ZZbij}\sigma(\pi_Z)\cdot x_1^{r_1}\cdots
x_N^{r_N}.
\end{split}
\]
Since $Z$ is uniquely determined by $(\ZZbij,\pi_Z)$ and vice-versa, we rewrite the sum as
\[
\resizebox{\hsize}{!}{$
\begin{split}
\det \M(v,\PP) & = \sigma(v) \sum_{\ZZbij\in\mathfrak{Z}_{\abel_v}}
	 \det [P]_{\ZZbij}
	 \sum_{\pi_1\in\Sym_{\{v\}_1}}
	 \sum_{\pi_2\in\Sym_{\{v\}_2}}
	 \cdots
	 \sum_{\pi_n\in\Sym_{\{v\}_n}}    \sigma(\pi_Z) \cdot
	 x_1^{\pi(1)}\cdots x_N^{\pi(N)}, \\
     & = \sigma(v) \sum_{\ZZbij\in\mathfrak{Z}_{\abel_v}}
	 \det [P]_{\ZZbij}
	 \sum_{\pi_1\in\Sym_{\{v\}_1}} \sigma(\pi_1)
	 \sum_{\pi_2\in\Sym_{\{v\}_2}} \sigma(\pi_2)
	 \ \cdots
	 \sum_{\pi_n\in\Sym_{\{v\}_n}}    \sigma(\pi_n) \cdot
	 x_1^{\pi(1)}\cdots x_N^{\pi(N)}. \\
\end{split}$}
\]
The powers of the variables $x_j$ such that $j\in [N]$ and $v_j\neq s_n$ stay constant in the last sum, so we factor their product to get
\[
\resizebox{\hsize}{!}{$
\begin{split}
\det \M(v,\PP) & = \sigma(v) \sum_{\ZZbij\in\mathfrak{Z}_{\abel_v}}
	 \det [P]_{\ZZbij}
	 \sum_{\pi_1\in\Sym_{\{v\}_1}} \sigma(\pi_1)
	 \sum_{\pi_2\in\Sym_{\{v\}_2}} \sigma(\pi_2)
	 \ \cdots
	 \prod_{i\in [N]\setminus \{v\}_{n}}x_i^{\pi(i)}
	 \sum_{\pi_n\in\Sym_{\{v\}_n}} \sigma(\pi_n) \cdot
	 \prod_{j\in \{v\}_{n}}x_j^{\pi(j)}. \\
\end{split}
$}
\]
By Definition~\ref{def:schur} of Schur functions, we get
\[
\sum_{\pi_n\in\Sym_{\{v\}_n}} \sigma(\pi_n) \cdot
\prod_{j\in \{v\}_{n}}x_j^{\pi(j)}= \det \Vand_{\{v\}_{n}}(c_n)\sh_{\lambda^{\ZZbij,n},\{v\}_{n}}.
\]
The latter equality leads to the equation 
\[
\resizebox{\hsize}{!}{$
\begin{split}
\det \M(v,\PP) = &\  \sigma(v) 
	   \det \Vand_{\{v\}_{n}}(c_n) \\
	 & \times \sum_{\ZZbij\in\mathfrak{Z}_{\abel_v}}
	   \det [P]_{\ZZbij}
	   \sh_{\lambda^{\ZZbij,n},\{v\}_{n}}
	   \sum_{\pi_1\in\Sym_{\{v\}_1}} \sigma(\pi_1)
	   \sum_{\pi_2\in\Sym_{\{v\}_2}} \sigma(\pi_2)
	   \ \cdots
	   \sum_{\pi_{n-1}\in\Sym_{\{v\}_{n-1}}} \sigma(\pi_{n-1})
	   \prod_{i\in[N]\setminus \{v\}_{n}}x_i^{\pi(i)}.\\
\end{split}$}
\]
Repeating the last step $n-1$ times, we get
\[
\begin{split}
\det M(v,\PP) & = \sigma(v) 
	 \V(v)
	 \sum_{\ZZbij\in\mathfrak{Z}_{\abel_v}}
	 \det [\PP]_{\ZZbij}
	 \sh_{\lambda^{\ZZbij,1},\{v\}_{1}}
	 \sh_{\lambda^{\ZZbij,2},\{v\}_{2}}
	 \cdots
	 \sh_{\lambda^{\ZZbij,n},\{v\}_{n}}.\qedhere
\end{split}
\]
\end{proof}

\begin{corollary}
If the abelian vector of $v$ is $\abel_v=(c_i)_{s_i\in S}$, then the polynomial 
\[
\sum_{\ZZbij\in\mathfrak{Z}_{\abel_v}}
\det [P]_{\ZZbij}
\Sh_{\Lambda_{\ZZbij},\Omega_v},
\]
is symmetric with respect to the group $\prod_{i=1}^{n}\Sym_{c_i}$ acting
on $\{x_1,\dots,x_N\}$ by permutation of indices such that $v_{\pi(j)}=v_j$, for all $j\in[N]$.
\end{corollary}

\begin{maincor}
\label{cor:D_sign_of_model}
If $x_i>0$ for all $i\in[N]$ and $x_j>x_i$ whenever $i<j$ and $v_i= v_j$, then
\[
\sign(\det M(v,\PP))=\sigma(v)\sign\left(
\sum_{\ZZbij\in\mathfrak{Z}_{\abel_v}}
\det [\PP]_{\ZZbij}
\Sh_{\Lambda_{\ZZbij},\Omega_v}
\right),
\]
where the $\sign$ on the right-hand side is the usual sign function on real numbers.
\end{maincor}

Although at first sight the formula of the latter corollary seems devoid of structure, it nonetheless reveals precisely:\\
$\sqbullet$ how the sign of the determinant behaves depending on the word $v$ through $\sigma(v)$,\\
$\sqbullet$ which minors of~$\PP$ influence the value of the determinant, and\\
$\sqbullet$ how Schur functions (whose monomial all have positive coefficients) interact with the minors of $\PP$ to produce the determinant.\\
From this formula, we deduce that given a parameter tensor $\PP$, two words with the same abelian vector will yield the same determinant up to sign and appropriate index relabelings of the $x_i$'s.
Further, the set indexing the sum has a simple description: the only relevant minors of~$\PP$ are those obtained taking the columns according to the values in the abelian vector $\abel_v$ of $v$, that is independently of the combinatorial type!

\begin{example}[{Example~\ref{ex:model4} continued}]
\label{ex:model4_2}
Let $(W,S)=(B_2,\{s_1,s_2\})$.
The $S$-sign of~$v=s_1s_2s_1s_2$ is $-1$ and the Vandermonde divisor is $(x_3-x_1)(x_4-x_2)$.
To obtain the determinant of $M$, we should compute the minors of $\PP$ with 2 columns in the first block corresponding to the letter~$s_1$ and 2 columns in the second block corresponding to the letter $s_2$.
There are 9 minors in total, of which only two are non-zero: when $\ZZbij_1=\{(s_1,0),(s_1,1),(s_2,0),(s_2,2)\}$, and $\ZZbij_2=\{(s_1,0),(s_1,2),(s_2,0),(s_2,1)\}$.
We get $\det [\PP]_{\ZZbij_1}=-1$ and $\det [\PP]_{\ZZbij_2}=1$.
By Theorem~\ref{thm:B_det}, the determinant of $M$ is
\[
\det M = -(x_3-x_1)(x_4-x_2)\left((-1)\left(\sh_{(0,0),\{1,3\}}\sh_{(1,0),\{2,4\}}\right)+(1)\left(\sh_{(1,0),\{1,3\}}\sh_{(0,0),\{2,4\}}\right)\right).
\]
We get the values of the Schur polynomials from Example~\ref{ex:partial_schur}.
Thus, the determinant is indeed
\[
\det M = -{\left(x_{3} - x_{1}\right)} {\left(x_{4} - x_{2}\right)}{\left(x_{1} - x_{2} + x_{3} - x_{4}\right)}.
\]
The sum in the formula of Theorem~\ref{thm:B_det} for the determinant is in fact a tensor in $V\otimes V$, where~$V$ is the vector space whose basis elements are labeled by the partitions of length $2$ with parts of size at most $1$:
\[
\resizebox{\hsize}{!}{$
\begin{blockarray}{rrrr}
\begin{block}{rrrr}
& \sh_{(0,0),\{2,4\}} &  \sh_{(1,0),\{2,4\}} & \sh_{(1,1),\{2,4\}} \\
\end{block}
\begin{block}{r(rcc)}
\sh_{(0,0),\{1,3\}} & 0  & -1 & 0 \\ 
\sh_{(1,0),\{1,3\}} & 1 & 0 & 0 \\ 
\sh_{(1,1),\{1,3\}} & 0  & 0 & 0 \\ 
\end{block}
\end{blockarray}
= -\sh_{(0,0),\{1,3\}}\otimes\sh_{(1,0),\{2,4\}}+\sh_{(1,0),\{1,3\}}\otimes\sh_{(0,0),\{2,4\}}.
$}
\]

\end{example}

\begin{example}
This result also applies on any matrix with univariate polynomials in its columns.
Consider the matrix
\[
M=
\left(\begin{array}{lll}
x_{1} + 1 & x_{2} + 1 & x_{3} + 1 \\
x_{1}^{2} + x_{1} & x_{2}^{2} + x_{2} & x_{3}^{2} + x_{3} \\
x_{1}^{2} + 1 & x_{2}^{2} + 1 & x_{3}^{2} + 1
\end{array}\right).
\]
It can be written using the coefficients and variables matrices as:
\[
M = \bigoplus_{i=1}^3
\left(\begin{array}{rrr}
1 & 1 & 0 \\
0 & 1 & 1 \\
1 & 0 & 1 
\end{array}\right)
\hspace{0.1cm}\times
\left(\begin{array}{rrrrrrrrr}
1 & x_{1} & x_{1}^{2} & 0 & 0 & 0 & 0 & 0 & 0 \\
0 & 0 & 0 & 1 & x_{2} & x_{2}^{2} & 0 & 0 & 0 \\
0 & 0 & 0 & 0 & 0 & 0 & 1 & x_{3} & x_{3}^{2}
\end{array}\right)^{\top}.
\]
Since interchanging variables is equivalent to permuting columns, the parameter tensor is a usual matrix
\[
P=
\left(\begin{array}{rrr}
1 & 1 & 0 \\
0 & 1 & 1 \\
1 & 0 & 1 
\end{array}\right)
\]
and there is only one subset $\ZZbij\in\mathfrak{Z}_{(3)}$, i.e. all the columns of $P$.
The $S$-sign of the word $s_1s_1s_1$ is $+1$.
The Vandermonde part is $(x_3-x_1)(x_3-x_2)(x_2-x_1)$.
The determinant of $P$ is $2$ and $\sh_{(0,0,0)}=1$.
Using Theorem~\ref{thm:B_det} we get
\[
\det M = 2(x_3-x_1)(x_3-x_2)(x_2-x_1).
\]
\end{example}

\begin{example}[{Dual Cauchy Identity, \cite[Theorem~7.14.3]{stanley_enumerative_1999}}]
\label{ex:dual_cauchy}
The theorem can also be applied to reobtain a proof of the dual Cauchy identity, that we illustrate in an example.
Let $a=2$ and $b=4$. Consider the matrix
\[
M_{a,b}=M_{2,4}=
\left(\begin{array}{rrrrrr}
1         &         1 & -y_{1}^{5} & -y_{2}^{5} & -y_{3}^{5} & -y_{4}^{5} \\
x_{1}     & x_{2}     &  y_{1}^{4} &  y_{2}^{4} &  y_{3}^{4} &  y_{4}^{4} \\
x_{1}^{2} & x_{2}^{2} & -y_{1}^{3} & -y_{2}^{3} & -y_{3}^{3} & -y_{4}^{3} \\
x_{1}^{3} & x_{2}^{3} &  y_{1}^{2} &  y_{2}^{2} &  y_{3}^{2} &  y_{4}^{2} \\
x_{1}^{4} & x_{2}^{4} & -y_{1}     & -y_{2}     & -y_{3}     & -y_{4}     \\
x_{1}^{5} & x_{2}^{5} &      1     &          1 &          1 & 1
\end{array}\right).
\]
If $x_1=x_2$, then the determinant of $M_{2,4}$ vanishes, and similarly for two distinct $y_i$ and $y_j$.
Thus, the Vandermonde divisor is ${\left(x_{2} - x_{1}\right)} {\left(y_{2} - y_{1}\right)} {\left(y_{3} - y_{1}\right)} {\left(y_{4} - y_{1}\right)} {\left(y_{3} - y_{2}\right)} {\left(y_{4} - y_{2}\right)} {\left(y_{4} - y_{3}\right)}$.
Further, one notices that if $x_i=-y_j^{-1}$ then, the corresponding columns are colinear by a factor $-y_j^{5}$.
Hence the determinant of $M_{2,4}$ is
\[
\det M_{2,4}=\V(s_1s_1s_2s_2s_2s_2)\cdot \prod_{\substack{1\leq i \leq 2\\ 1\leq j\leq 4}} (1+x_iy_j).
\]
On the other hand, the corresponding parameter matrix is:
\[
P_{a,b} = P_{2,4} = 
\begin{blockarray}{ccccccc}
\begin{block}{ccccccc}
(s_1,0\dots 5) & (s_2,0) & (s_2,1) & (s_2,2) & (s_2,3) & (s_2,4) & (s_2,5)\\ 
\end{block}
\begin{block}{(c|cccccc)}
\multirow{6}{*}{$\Id_6$ } & 0 & 0 & 0 & 0 & 0 & -1 \\
& 0 & 0 & 0 & 0 & 1 & 0 \\
& 0 & 0 & 0 & -1 & 0 & 0 \\
& 0 & 0 & 1 & 0 & 0 & 0 \\
& 0 & -1 & 0 & 0 & 0 & 0 \\
& 1 & 0 & 0 & 0 & 0 & 0 \\
\end{block}
\end{blockarray}
\]
There are exactly $\binom{a+b}{a}=\binom{6}{2}=15$ non-zero minors of $P$ with 2 columns in the first part and 4 columns in the second which are all equal to 1.
Following Example~\ref{ex:model4_2}, Theorem~\ref{thm:B_det} gives
\[
\det M_{2,4}=\V(s_1s_1s_2s_2s_2s_2)\sum_{\lambda}\sh_{\lambda,\{x_1,x_2\}}\sh_{\lambda',\{y_1,y_2,y_3,y_4\}},
\]
where $\lambda'$ is the conjugate partition of $\lambda$ and the sum is over all partitions $\lambda$ of length $a=2$ with parts of size at most $b=4$.
Cancelling the Vandermonde divisor, we recover the dual Cauchy identity.
The general case works similarly.
\end{example}

\subsection{The parameter matrices behind Bergeron--Ceballos--Labb\'e's counting matrices}
\label{ssec:BCL_parameter}

In the article \cite{bergeron_fan_2015}, the construction of fans is based on a matrix
called ``counting matrix'', whose entries enumerate occurrences of certain subwords
contained in a fixed word. 
As it turns out, these counting matrices are signature matrices.
The factorizations of the determinants of counting matrices were critical in order to prove the correctness of the construction.
In view of the intricate description in the previous
section, the fact that counting matrices of type $A_3$---obtained via a simple
combinatorial rule---are signature matrices should be regarded
as a highly exceptional and not fully explained behavior.

Yet, Theorem~\ref{thm:B_det} gives a complete description of the factorizations of determinants of minors of counting matrices.
Furthermore, for any finite irreducible Coxeter group, it precisely dictates how one could obtain signature matrices through parameter matrices.
We revisit here these counting matrices using parameter matrices. 

\subsubsection{Type $A_1$}

We have $n=N=\nu=1$.
Parameter matrices are $(1\times 1)$-matrices $P=(p)$ containing the real number $p$ and the variables matrix $\T_W$ is $(1)$.
This way, given the only reduced word $v=s_1$, the model matrix $M(v,P)$ is the $(1\times 1)$-matrix $(p)$.
In order to be a signature matrix, the real number $p$ should be non-zero.
In \cite[Appendix]{bergeron_fan_2015}, the counting matrix is obtained by setting $p=1$.

\subsubsection{Type $A_2$}

We have $n=2$, $N=3$, and $\nu=2$.
Parameter tensors are $(3\times 2\times 2)$-dimensional, compare with Example~\ref{ex:model_mat_a2}.
For some given positive integer $m$, the counting matrix called $D_{s_1s_2,m}$ in \cite{bergeron_fan_2015} gives rise to the parameter matrix
\[
P_{s_1s_2,m}=
\begin{blockarray}{cccc}
\begin{block}{cccc}
(s_1,0) & (s_1,1) & (s_2,0) & (s_2,1)\\ 
\end{block}
\begin{block}{(cc|cc)}
1 & 0  & 0 & 0 \\
m & -1 & 0 & 1 \\
0 & 0  & 1 & 0 \\
\end{block}
\end{blockarray}\ .
\]
This parameter matrix has two non-zero minors $\{(s_1,0),(s_1,1),(s_2,0)\}$ and $\{(s_1,0)$, $(s_2,0)$, $(s_2,1)\}$, which are both equal to $-1$.
Further, the corresponding partial Schur functions $\Sh_{((0,0),(0)),(\{1,3\},\{2\})}$ and $\Sh_{((0),(0,0)),(\{2\},\{1,3\})}$ are both equal to~$1$.
For $v=s_1s_2s_1$, the model matrix $M(s_1s_2s_1,P_{s_1s_2,m})$ is
\[
M(s_1s_2s_1,P_{s_1s_2,m})=
\left(\begin{array}{rrr}
1 & 0 & 1 \\
-x_{1} + m & x_{2} & -x_{3} + m \\
0 & 1 & 0
\end{array}\right).
\]
For $v=s_2s_1s_2$, the model matrix $M(s_2s_1s_2,P_{s_1s_2,m})$ is
\[
M(s_2s_1s_2,P_{s_1s_2,m})=
\left(\begin{array}{rrr}
0 & 1 & 0 \\
x_{1} & -x_{2} + m & x_{3} \\
1 & 0 & 1
\end{array}\right).
\]
To get back the counting matrix, one has to set the parameter $x_i$ to be the position of the factor~$s_1s_2$ in which $v_i$ appears in $(s_1s_2)^m$, and remove $1$ if $v_i=s_1$.
This number fits exactly with how the counting matrix is defined in this case.
From this, we get that
\[
\begin{split}
\det M(s_1s_2s_1,P_{s_1s_2,m}) & = \sigma(s_1s_2s_1)(x_3-x_1)\cdot(-1\cdot 1) \\
			 &  =(-1)(x_3-x_1)(-1)=(x_3-x_1).
\end{split}
\]
The determinant of the model matrix $M(s_2s_1s_2,P_{s_1s_2,m})$ is similar.

\subsubsection{Type $A_3$}
\label{sssec:a3}

We have $n=3$, $N=6$, and $\nu=3$.
Parameter tensors are $(6\times 3\times 3)$-dimensional.
For some given positive integer $m$, the counting matrix $D_{s_1s_2s_3,m}$ from \cite{bergeron_fan_2015} gives rise to the parameter matrix
\[
P_{s_1s_2s_3,m}=
\begin{blockarray}{ccccccccc}
\begin{block}{ccccccccc}
0 & 1 & 2 & 3 & 4 & 5 & 6 & 7 & 8=\nu n -1 \\ 
(s_1,0) & (s_1,1) & (s_1,2) & (s_2,0) & (s_2,1) & (s_2,2) & (s_3,0) & (s_3,1) & (s_3,2) \\ 
\end{block}
\begin{block}{(ccc|ccc|ccc)}
1 & 0 & 0 & 0 & 0 & 0 & 0 & 0 & 0 \\
0 & 1 & 0 & m + 1 & -1 & 0 & 0 & 0 & 0 \\
0 & \frac{1}{2} & \frac{1}{2} & 0 & m + 1 & -1 & \binom{m+2}{2} & -m - \frac{3}{2} & \frac{1}{2} \\
0 & 0 & 0 & 1 & 0 & 0 & 0 & 0 & 0 \\
0 & 0 & 0 & 0 & 1 & 0 & m + 1 & -1 & 0 \\
0 & 0 & 0 & 0 & 0 & 0 & 1 & 0 & 0 \\
\end{block}
\end{blockarray}\ .
\]
Whereas, the counting matrix $D_{s_2s_1s_3,m}$ gives rise to the parameter matrix
\[
P_{s_2s_1s_3,m}=
\begin{blockarray}{ccccccccc}
\begin{block}{ccccccccc}
0 & 1 & 2 & 3 & 4 & 5 & 6 & 7 & 8=\nu n -1 \\ 
(s_1,0) & (s_1,1) & (s_1,2) & (s_2,0) & (s_2,1) & (s_2,2) & (s_3,0) & (s_3,1) & (s_3,2) \\ 
\end{block}
\begin{block}{(ccc|ccc|ccc)}
0 & 0 & 0 & 1 & 0 & 0 & 0 & 0 & 0 \\
m + 1 & -1 & 0 & 0 & 1 & 0 & 0 & 0 & 0 \\
0 & 0 & 0 & 0 & 1 & 0 & m + 1 & -1 & 0 \\
\binom{m+1}{2} & \frac{1}{2} & -\frac{1}{2} & 0 & 0 & 1 & \binom{m+1}{2} & \frac{1}{2} & -\frac{1}{2} \\
0 & 0 & 0 & 0 & 0 & 0 & 1 & 0 & 0 \\
1 & 0 & 0 & 0 & 0 & 0 & 0 & 0 & 0 \\
\end{block}
\end{blockarray}\ .
\]
Although the two parameter matrices look different, their non-zero minors are almost all equal.
In order to specify the columns of the parameter matrices compactly, we label the nine columns from $0$ to $\nu n -1 =8$ from left-to-right:

\[
\begin{array}{rrr}
\det [P_*]_{\{0,1,2|3,4|6\}} = \frac{1}{2},  & \det [P_*]_{\{0,1,2|3|6,7\}} = \frac{1}{2}, & \det [P_*]_{\{0,1|3,4,5|6\}} = -1 \\[0.5em]
\det [P_*]_{\{0,1|3,4|6,8\}} = -\frac{1}{2},  & \det [P_*]_{\{0,1|3,5|6,7\}} = 1, & \det [P_*]_{\{0,1|3|6,7,8\}} = -\frac{1}{2} \\[0.5em]
\det [P_*]_{\{0,2|3,4|6,7\}} = -\frac{1}{2},  & \det [P_*]_{\{0|3,4,5|6,7\}} = 1, & \det [P_*]_{\{0|3,4|6,7,8\}} = -\frac{1}{2} \\[0.5em]
\end{array}
\]
The only minor which is different is $\det [P_*]_{\{0,1|3,4|6,7\}}$ which is $0$ for $P_{s_1s_2s_3,m}$ and $1$ for~$P_{s_2s_1s_3,m}$.
This explains the very small differences in the formulas for the determinants in \cite[Table~2 and~3]{bergeron_fan_2015} although the way they are obtained are different.
The difference appears in a factor for words with abelian vector $(2,2,2)$.
For example, consider the reduced word $v=s_2s_1s_3s_2s_3s_1$.
Then $\Omega_v=(\{2,6\},\{1,4\},\{3,5\})$ and there are three minors of $P_{s_1s_2s_3,m}$ in $\mathfrak{Z}_{\abel_v}$ that are non-zero: $\ZZbij_1:=\{0,1,3,4,6,8\},\ZZbij_2:=\{0,1,3,5,6,7\},$ and $\ZZbij_3:=\{0,2,3,4,6,7\}$.
By Theorem~\ref{thm:B_det} and Example~\ref{ex:s_sign_a3} the determinant of the model matrix $M(v,P_{s_1s_2s_3,m})$ is
\[
\resizebox{\hsize}{!}{$
\begin{split}
\det M(v,P_{s_1s_2s_3,m}) & = (1) (x_4 - x_1) (x_6 - x_2) (x_5 - x_3) \left( \det [P]_{\ZZbij_1}\Sh_{\Lambda_{\ZZbij_1,\Omega_v}} + \det [P]_{\ZZbij_2}\Sh_{\Lambda_{\ZZbij_2,\Omega_v}} + \det [P]_{\ZZbij_3}\Sh_{\Lambda_{\ZZbij_3,\Omega_v}}\right), \\
& = (x_4 - x_1) (x_6 - x_2) (x_5 - x_3)\left( -\frac{1}{2}\cdot (x_3+x_5) + 1\cdot (x_1+x_4) -\frac{1}{2}\cdot(x_2+x_6)\right), \\
& = -\frac{1}{2}(x_4 - x_1) (x_6 - x_2) (x_5 - x_3)\left( x_3+x_5 - 2\cdot (x_1+x_4) +x_2+x_6\right), \\
& = -\frac{1}{2}(x_1 - x_4) (x_2 - x_6) (x_3 - x_5)\left( 2\cdot (x_1+x_4) -x_2-x_6-x_3-x_5 \right).
\end{split}
$}
\]
The last expression is in accordance with what is written in \cite[Table~3]{bergeron_fan_2015}.

\subsubsection{Type $A_4$}

As noticed in \cite[Section~9]{bergeron_fan_2015}, the construction using counting matrices in type~$A_4$ delivered a parameter matrix which was not generic enough.
Indeed, the reduced word $w=s_2s_1s_2s_3s_4s_2s_3s_2s_1s_2$ has abelian vector $(2,5,2,1)$ and the counting matrix gives three curves of degree $3$ in $\R^{10}$ where points should be taken.
Hence taking $5$ points on the curve corresponding to the letter~$s_2$ does not span a $5$-dimensional subspace, which is necessary for a dual simplicial cone
corresponding to the reduced word $w$.
Let $c=s_2s_4s_1s_3$ and $\wo(c):=s_2s_4s_1s_3s_2s_4s_1s_3s_2s_4$.
The reduced word $w$ appears first as a subword of $c^k\wo(c)$ when $k\geq 2$.
This explains the non-zero numbers in the fourth column of Table~7 of \cite{bergeron_fan_2015}, which represented (in particular) the word $w$.
This case epitomizes the \emph{fundamental} difference between realizing the \emph{cluster complex} as a simplicial fan and realizing the \emph{multi}-cluster complex as a simplicial fan.
When $k$ increases in the above word, certain reduced words which \emph{never} appear as subword in the cluster complex, suddenly appear and require a higher genericity.
In fact, the h\"ochstfrequenz is at least $ne^{\Omega(\sqrt{\log n/2})}$ and the parameter tensor should have at least this degree of genericity in order to produce a signature matrix.

\section{Universality of parameter matrices}
\label{sec:parameter_sign_mat}

In this section, we show that parameter tensors are universal in the following sense:

\begin{quote}
\textbf{Universality of parameter tensors.} Given a chirotopal realization of a subword complex $\Delta_W(p)$ supported by a vector configuration $\A$ and a Gale transform $\B\in\Gale(\A)$, there exists a parameter tensor $\PP_{\A}$ that parametrizes $\B$. Equivalently, $\B$ is the product of a variables tensor and a coefficient tensor $\C(p,\PP)$ given by some parameter tensor~$\PP_\A$.
\end{quote}

Concretely, consider some matrix $\A\in\R^{(m-N)\times m}$ realizing $\Delta_W(p)$ as a chirotope, and $\B\in\Gale(\A)$.
Recall that the $i$-th column of $A$ corresponds to the $i$-th letter~$p_i$ of~$p$.
For each letter $s_j\in S$, we proceed as follows.
Consider the columns $i\in [m]$ of $\B$ such that $p_i=s_j$, i.e. the occurrences of $s_j$ in $p$.
It is possible to find a polynomial of degree at most $|p|_j-1$, for each coordinate $k\in [N]$, that interpolates the values of the $k$-th coordinate of the columns corresponding to the occurrences of $s_j$.
There are many possibilities to do so; to get a specific choice, we consider the two-dimensional points $(i,\B(k,i))$ for the $k$-th coordinate, where $\B(k,i)$ denotes the $k$-th entry of the $i$-th column of $\B$.
This way, as $i$ increases, so does the first entry $x_i:=i$.
Doing this for each letter $s_j\in S$, we get a parameter tensor $\PP_\B$ with $\nu$ replaced by $\max\{|p|_j-1~:~j\in[n]\}$.
In order to know if $\B$ is a signature matrix for $p$, we use Corollary~\ref{cor:D_sign_of_model}.

\begin{theorem}
\label{thm:construct}
Let $p\in S^m$ and $\A\in\R^{(m-N)\times m}$.
Further let $\PP_\B$ denote the parameter tensor associated to a Gale transform~$\B\in\Gale(\A)$ obtained as above.
In particular, assume that $x_i>0$ for all $i\in[N]$ and $x_i<x_j$ whenever $i<j$ and $v_i= v_j$.
The matrix $\B$ is a signature matrix for $p$ if and only if
\[
\sign\left(
\sum_{\ZZbij\in\mathfrak{Z}_{\abel_v}}
\det [\PP_\B]_{\ZZbij}
\Sh_{\Lambda_{\ZZbij},\Omega_v}
\right) = \punc(v) = \sigma(v)\tau(v)
\]
for every reduced word $v$ of $\wo$ which is a subword of $p$.
\end{theorem}

\begin{proof}
This is a consequence of Definition~\ref{def:signature_matrix} and Corollary~\ref{cor:D_sign_of_model}.
\end{proof}

After doing a braid move of length $2$ on $v$, the variables $x_i$'s on the left-hand side of the equation get relabeled, the minors $[\PP_\B]_\ZZbij$'s remain unchanged, and the right-hand side remains invariant, thanks to Corollary~\ref{cor:B_commutations}.
This yields exactly one equation up to labeling of the $x_i$'s to be fulfilled for each commutation class.
Furthermore, two reduced words with the same abelian vector have equal left-hand sides up to relabeling of the $x_i$'s.
Hence for each abelian vector $\abel_v$, there is either~1 or 2 equalities up to relabeling of the $x_i$'s that have to be fulfilled depending on whether the punctual sign is constant for the abelian vector~$\abel_v$, leading to a constant number of equations to be satisfied for each Coxeter group $W$.

The definition of signature matrix involves checking an equation for \emph{every} occurrence of \emph{every} reduced word, that is for every facet, directly on the Gale transform $\B$.
Under the condition that the numbers $\{x_i\}_{i\in[r]}$ increase on occurrences of letters in $S$, these conditions can be expressed with at most two explicit conditions per abelian vector of combinatorial type of facet and the values of the $x_i$'s play a less significant role.
Thus, the previous theorem allows to reduce significantly the study of signature matrices to that of minors of parameter tensors given by abelian vectors of combinatorial types of facets of subword complexes.

\begin{example}[Symmetric group $\Sym_{4}=A_3$]
There are $5$ abelian vectors for $\wo$ (see Table~\ref{tab:abelian_vectors_A} in the appendix): $(3,2,1),(2,3,1),(2,2,2),(1,3,2),(1,2,3)$.
Thus, there are 5 types of determinants of model matrices by Theorem~\ref{thm:construct}, one of them is described in Section~\ref{sssec:a3}.
Now, using Example~\ref{ex:punc_a3}, we get that determinants should be positive for reduced expressions with abelian vectors $(3,2,1)$ and $(1,3,2)$, and negative for reduced expressions with abelian vectors $(1,2,3)$ and $(2,3,1)$.
The determinants for reduced expressions with abelian vectors $(2,2,2)$ have to be either positive or negative depending on the expression.
Hence we get a total of six types of inequalities given by prescribing signs of minors of $\PP_\A$.
For example, the minors $\det[\PP_\B]_\ZZbij$ for $\ZZbij\in(\mathfrak{Z}_{(3,2,1)}\cup\mathfrak{Z}_{(1,3,2)})$ should be non-negative, with at least one positive.
Similarly, for $(1,2,3)$ and $(2,3,1)$ they should be non-positive, with at least one negative.
Finally, for $(2,2,2)$ it should be non-negative for the words in $\{13\}2\{13\}2$ and non-positive for the words in $2\{13\}2\{13\}$.
Indeed, as Section~\ref{sssec:a3} reveals, it is possible to find parameter tensors with at most 10 non-zero minors with the appropriate sign patterns.
What is more, the involved Schur functions have degree 0 or 1, i.e. are constant or linear leading to a polyhedral realization space.
Within a set $\mathfrak{Z}_{\abel_v}$, the signs of the minor are all equal or 0 except for the abelian vector $(2,2,2)$ which has both signs and the relative values of the $x_i$'s imposed by the combinatorial construction leads to a valid chirotope.
\end{example}

The following universality result follows from the above discussion and Theorem~\ref{thm:construct}.

\begin{mainthm}[Universality of parameter tensors]
\label{thm:C_univ}
Let $p\in S^m$ and $\A\in\R^{(m-N)\times m}$.
If $\A$ is a chirotopal realization of the subword complex $\Delta_W(p)$, then there exist a parameter tensor $\PP_\A$, and $m$ real numbers $x_i>0$, with $i\in [m]$, such that
\begin{itemize}
	\item $i<j$ and $p_i=p_j$ implies $x_i<x_j$, and
	\item for every occurrence of each reduced word $v$ of $\wo$ which is a subword of $p$, the following equality holds \[
\sign\left(
\sum_{\ZZbij\in\mathfrak{Z}_{\abel_v}}
\det [\PP_\A]_{\ZZbij}
\Sh_{\Lambda_{\ZZbij},\Omega_v}
\right) = \punc(v) = \sigma(v)\tau(v).
\]
\end{itemize}
\end{mainthm}

Thus, in order to obtain geometric realizations (either as chirotopes, fans or polytopes) of subword complexes, the family $\mathfrak{X}$ of parameter tensors in the previous theorem constitute a natural and universal object to study.
This theorem opens the door to the usage of the $S$-sign and $T$-sign functions and partial Schur functions to study covering of realization spaces of cyclic polytopes, (generalized) associahedra, and more generally of subword complexes.
Finally, we draw the following two conclusions.

$\sqbullet$ \underline{\smash{Degree $\max\{|p|_j-1~:~j\in[n]\}$}}: This construction does not come with an upper bound on the degree of the interpolating polynomials.
The degree depends on the abelian vector of the word~$p$.
Is it possible to give an upper bound on the degree necessary that is independent of $p$?
Given a finite Coxeter group, as a first approximation, one may consider a shortest word $p_{\text{univ}}$ that contains an occurrence of every reduced expressions in $\Red(\wo)$, see~\cite[Question~6.1]{knutson_subword_2004}.
Obtaining $p_{\text{univ}}$ is a particular case of the ``shortest common supersequence problem'' which is known to be NP-complete \cite{raiha_shortest_1981}.
Are the degrees obtained with such a word an upper bound?
A priori, there is no reason supporting the fact that an upper bound on the degrees exists for each Coxeter group.
The non-existence of an upper bound would place yet another difficulty to overcome to provide geometric realizations of subword complexes.
However, small cases seem to indicate that such upper bounds do exist, as described in Section~\ref{ssec:BCL_parameter}.

\begin{example}[Example~\ref{ex:model4_2} continued, $2k$-dimensional cyclic polytopes on $2k+4$ vertices]
Let $W=B_2$ and $S=\{s_1,s_2\}$.
Let ${k\geq 1}$, $c=s_1s_2$, $\wo(c)=s_1s_2s_1s_2$, and $p=c^k\wo(c)$.
We now consider the curves $f_1(x)=(1,0,-x,x^2)$ and $f_2(x)=(0,1,x,-x^2)$ and we assign a number $x_i>0$ to each letter $p_i$ of $p$, such that $x_j>x_i$ whenever $p_i=p_j$ and $j<i$.
If $p_i=s_1$, we evaluate~$f_1$ at~$x_i$, otherwise $p_i=s_2$ and we evaluate $f_2$ at $x_i$ to assign a vector in $\R^4$ to each letter of $p$. 
There are two reduced words $s_1s_2s_1s_2$ and $s_2s_1s_2s_1$ for $\wo$.
Using the computation in Example~\ref{ex:model4_2}, we get the following conditions
\[
\begin{split}
-1 = \punc(s_1s_2s_1s_2) = \sign (x_{i_1} - x_{i_2} + x_{i_3} - x_{i_4}) & \qquad\text{ if } p_{i_1}p_{i_2}p_{i_3}p_{i_4}=s_1s_2s_1s_2, \\
1 = \punc(s_2s_1s_2s_1) = \sign (-x_{i_1} + x_{i_2} - x_{i_3} + x_{i_4}) & \qquad\text{ if } p_{i_1}p_{i_2}p_{i_3}p_{i_4}=s_2s_1s_2s_1,
\end{split}
\]
to get a signature matrix.
These conditions are equivalent to $x_1<x_2<\dots < x_{2k+3}< x_{2k+4}$.
This comes as no surprise, since this is an instance of the $2k$-dimensional cyclic polytope on $2k+4$ vertices \cite[Section~6.4]{ceballos_subword_2014}.
Taking $x_i=i$, we can verify that all conditions are satisfied and we get a signature matrix for the boundary complex of the cyclic polytope on $2k+4$ vertices.
\end{example}

$\sqbullet$ \underline{\smash{Relation to the halving lines problem}}: The geometric realization of subword complexes and the halving line problem are related in the following sense. 
On the one hand, if every subword complex of type $A_n$ admits a realization using parameter matrices as above, in particular one has a realization for the word $p_{\text{univ}}$. 
Since every set of $n+1$ points in general position on the plane leads to a reduced expression contained in $p_{\text{univ}}$, one gets an upper bound on the halving line problem by taking the maximal degree of the polynomials involved in the realization for $p_{\text{univ}}$.
Thus, realizations of subword complexes of type~$A$ with words containing $p_{\text{univ}}$ as a subword provide upper bounds for the number of halving lines.
On the other hand, solving the halving line problem for $n+1$ points delivers a first approximation of the size of a word $p_{\text{univ}}'$ that contains an occurrence of every reduced word in $\Red(\wo)$ up to commutation.
This way, one can get a lower bound on the size of the word $p_{\text{univ}}$ and a lower bound on the degree of the polynomials involved to realize all subword complexes. 

\newpage
\newcommand{\etalchar}[1]{$^{#1}$}
\providecommand{\bysame}{\leavevmode\hbox to3em{\hrulefill}\thinspace}
\providecommand{\MR}{\relax\ifhmode\unskip\space\fi MR }
\providecommand{\MRhref}[2]{%
  \href{http://www.ams.org/mathscinet-getitem?mr=#1}{#2}
}
\providecommand{\href}[2]{#2}

\appendix

\newpage

\section{Some Abelian vectors of reduced words of the longest elements}
\label{app:abelian vectors}

In the tables below, we give the possible abelian vectors of the longest element $\wo$ for the finite irreducible Coxeter groups of small rank.
The computations used {\tt Sagemath}'s implementation of Coxeter groups to generate all reduced word \cite{sagemath}.
The generation of the reduced word proceeds without much difficulty; the current bottleneck being that in types $A,B,D,H$ of higher ranks and other types, the computations all require more than 256GB of RAM memory.

\begin{table}[!ht]
\begin{center}
\resizebox{0.75\textwidth}{!}{
\begin{tabular}{c|l}
Type $A_n$ & Abelian vectors of $\wo$ \\\hline
\begin{tikzpicture}[vertex/.style={inner sep=1pt,circle,draw=black,fill=blue,thick},
                    deux/.style={thick,blue},
		    trois/.style={thick,dashed,red}]

\node[vertex,label=left:$s_1$] (s1) at (0,0) {};

\end{tikzpicture} & \begin{tikzpicture}\node at (0,0) {$\{(1)\}$};\end{tikzpicture} \\\hline
\begin{tikzpicture}[vertex/.style={inner sep=1pt,circle,draw=black,fill=blue,thick},
                    deux/.style={thick,blue},
		    trois/.style={thick,dashed,red}]

\node[vertex,label=left:$s_1$] (s1) at (0,0) {};
\node[vertex,label=right:$s_2$] (s2) at (1,0) {};

\draw[deux] (s1) -- (s2);

\end{tikzpicture} & \begin{tikzpicture}\node at (0,0) {$\{(2,1),(1,2)\}$};\end{tikzpicture} \\\hline
\begin{tikzpicture}[vertex/.style={inner sep=1pt,circle,draw=black,fill=blue,thick},
                    deux/.style={thick,blue},
		    trois/.style={thick,dashed,red}]

\node[vertex,label=left:$s_1$] (s1) at (0,0) {};
\node[vertex,label=above:$s_2$] (s2) at (1,0) {};
\node[vertex,label=right:$s_3$] (s3) at (2,0) {};

\draw[deux] (s1) -- (s2) -- (s3);

\end{tikzpicture} & \begin{tikzpicture}\node at (0,0) {$\{(3,2,1),(2,3,1),(2,2,2),(1,3,2),(1,2,3)\}$};\end{tikzpicture} \\\hline
\begin{tikzpicture}[vertex/.style={inner sep=1pt,circle,draw=black,fill=blue,thick},
                    deux/.style={thick,blue},
		    trois/.style={thick,dashed,red}]

\node[vertex,label=left:$s_1$] (s1) at (0,0) {};
\node[vertex,label=above:$s_2$] (s2) at (1,0) {};
\node[vertex,label=above:$s_3$] (s3) at (2,0) {};
\node[vertex,label=right:$s_4$] (s4) at (3,0) {};

\draw[deux] (s1) -- (s2) -- (s3) -- (s4);

\end{tikzpicture} & \begin{tabular}{l}
             $\{(4, 3, 2, 1), (3, 4, 2, 1), (3, 3, 3, 1),$ \\ 
               $(3, 3, 2, 2), (3, 2, 4, 1), (3, 2, 3, 2),$ \\
               $(2, 5, 2, 1), (2, 4, 3, 1), (2, 4, 2, 2),$ \\
               $(2, 3, 4, 1), (2, 3, 3, 2), (2, 3, 2, 3),$ \\
               $(2, 2, 4, 2), (2, 2, 3, 3), (1, 4, 3, 2),$ \\
               $(1, 4, 2, 3), (1, 3, 4, 2), (1, 3, 3, 3),$ \\
               $(1, 2, 5, 2), (1, 2, 4, 3), (1, 2, 3, 4)\}$ \\
	     \end{tabular}\\\hline
\begin{tikzpicture}[vertex/.style={inner sep=1pt,circle,draw=black,fill=blue,thick},
                    deux/.style={thick,blue},
		    trois/.style={thick,dashed,red}]

\node[vertex,label=left:$s_1$] (s1) at (0,0) {};
\node[vertex,label=above:$s_2$] (s2) at (1,0) {};
\node[vertex,label=above:$s_3$] (s3) at (2,0) {};
\node[vertex,label=above:$s_4$] (s4) at (3,0) {};
\node[vertex,label=right:$s_5$] (s5) at (4,0) {};

\draw[deux] (s1) -- (s2) -- (s3) -- (s4) -- (s5);

\end{tikzpicture} & \begin{tabular}{l}
             $97$ abelian vectors with coordinatewise \\
	     minimum $(1,2,3,2,1)$ and maximum $(5,6,6,6,5)$.
	     \end{tabular}\\
\end{tabular}}
\end{center}
\caption{Abelian vectors of the reduced words for the longest element in type $A$}
\label{tab:abelian_vectors_A}
\end{table}

\vspace{-0.25cm}
\begin{table}[!ht]
\begin{center}
\resizebox{0.75\textwidth}{!}{
\begin{tabular}{c|l}
Type $B_n$ & Abelian vectors of $\wo$ \\\hline
\begin{tikzpicture}[vertex/.style={inner sep=1pt,circle,draw=black,fill=blue,thick},
                    deux/.style={thick,blue},
		    trois/.style={thick,dashed,red}]

\node[vertex,label=left:$s_1$] (s1) at (0,0) {};
\node[vertex,label=right:$s_2$] (s2) at (1,0) {};

\draw[deux] (s1) -- node[inner sep=0pt,midway,label=above:{\color{black} $4$}] {} (s2);

\end{tikzpicture} & \begin{tikzpicture}\node at (0,0) {$\{(2,2)\}$};\end{tikzpicture} \\\hline
\begin{tikzpicture}[vertex/.style={inner sep=1pt,circle,draw=black,fill=blue,thick},
                    deux/.style={thick,blue},
		    trois/.style={thick,dashed,red}]

\node[vertex,label=left:$s_1$] (s1) at (0,0) {};
\node[vertex,label=above:$s_2$] (s2) at (1,0) {};
\node[vertex,label=right:$s_3$] (s3) at (2,0) {};

\draw[deux] (s1) -- node[inner sep=0pt,midway,label=above:{\color{black} $4$}] {} (s2) -- (s3);

\end{tikzpicture} & \begin{tikzpicture}\node at (0,0) {$\{(3,4,2),(3,3,3)\}$};\end{tikzpicture} \\\hline
\begin{tikzpicture}[vertex/.style={inner sep=1pt,circle,draw=black,fill=blue,thick},
                    deux/.style={thick,blue},
		    trois/.style={thick,dashed,red}]

\node[vertex,label=left:$s_1$] (s1) at (0,0) {};
\node[vertex,label=above:$s_2$] (s2) at (1,0) {};
\node[vertex,label=above:$s_3$] (s3) at (2,0) {};
\node[vertex,label=right:$s_4$] (s4) at (3,0) {};

\draw[deux] (s1) -- node[inner sep=0pt,midway,label=above:{\color{black} $4$}] {} (s2) -- (s3) -- (s4);

\end{tikzpicture} & \begin{tikzpicture}\node at (0,0) {$\{(4, 6, 4, 2), (4, 6, 3, 3), (4, 5, 5, 2), (4, 5, 4, 3), (4, 4, 6, 2), (4, 4, 5, 3), (4, 4, 4, 4)\}$};\end{tikzpicture} \\\hline
\end{tabular}}
\end{center}
\caption{Abelian vectors of the reduced words for the longest element in type $B$}
\label{tab:abelian_vectors_B}
\end{table}

\vspace{-0.25cm}
\begin{table}[!ht]
\begin{center}
\resizebox{0.75\textwidth}{!}{
\begin{tabular}{c|l}
Type $D_n$ & Abelian vectors of $\wo$ \\\hline
\begin{tikzpicture}[vertex/.style={inner sep=1pt,circle,draw=black,fill=blue,thick},
                    deux/.style={thick,blue},
		    trois/.style={thick,dashed,red}]

\node[vertex,label=left:$s_1$] (s1) at (0,-0.5) {};
\node[vertex,label=left:$s_2$] (s2) at (0,0.5) {};
\node[vertex,label=above:$s_3$] (s3) at (1,0) {};
\node[vertex,label=right:$s_4$] (s4) at (2,0) {};

\draw[deux] (s1) -- (s3);
\draw[deux] (s2) -- (s3) -- (s4);

\end{tikzpicture} & \begin{tikzpicture}\node[align=left] at (0,0) {$\{(4, 2, 4, 2), (3, 3, 4, 2), (3, 3, 3, 3),$ \\
             $(3, 2, 5, 2), (3, 2, 4, 3), (2, 4, 4, 2),$ \\
             $(2, 3, 5, 2), (2, 3, 4, 3), (2, 2, 6, 2),$ \\
             $(2, 2, 5, 3), (2, 2, 4, 4) \}$ \\};\end{tikzpicture} \\\hline
\begin{tikzpicture}[vertex/.style={inner sep=1pt,circle,draw=black,fill=blue,thick},
                    deux/.style={thick,blue},
		    trois/.style={thick,dashed,red}]

\node[vertex,label=left:$s_1$] (s1) at (0,-0.5) {};
\node[vertex,label=left:$s_2$] (s2) at (0,0.5) {};
\node[vertex,label=above:$s_3$] (s3) at (1,0) {};
\node[vertex,label=above:$s_4$] (s4) at (2,0) {};
\node[vertex,label=right:$s_5$] (s5) at (3,0) {};

\draw[deux] (s1) -- (s3);
\draw[deux] (s2) -- (s3) -- (s4) -- (s5);

\end{tikzpicture} & \begin{tikzpicture}\node[align=left] at (0,0) {$111$ abelian vectors with coordinatewise \\[0.5em]
	     minimum $(2,2,4,3,2)$ and maximum $(6,6,9,7,5)$};\end{tikzpicture}
\end{tabular}}
\end{center}
\caption{Abelian vectors of the reduced words for the longest element in type~$D_4$ and $D_5$}
\label{tab:abelian_vectors_D}
\end{table}

\begin{table}[!ht]
\begin{center}
\begin{tabular}{c|l}
Type $H_n$ & Abelian vectors of $\wo$ \\\hline
\begin{tikzpicture}[vertex/.style={inner sep=1pt,circle,draw=black,fill=blue,thick},
                    deux/.style={thick,blue},
		    trois/.style={thick,dashed,red}]

\node[vertex,label=left:$s_1$] (s1) at (0,0) {};
\node[vertex,label=above:$s_2$] (s2) at (1,0) {};
\node[vertex,label=right:$s_3$] (s3) at (2,0) {};

\draw[deux] (s1) -- node[inner sep=0pt,midway,label=above:{\color{black} $5$}] {} (s2) -- (s3);

\end{tikzpicture} & \begin{tikzpicture}\node[align=left] at (0,0) {$\{(6, 6, 3), (5, 7, 3), (5, 6, 4), (5, 5, 5)\}$};\end{tikzpicture}
\end{tabular}
\end{center}
\caption{Abelian vectors of the reduced words for the longest element in type $H_3$}
\label{tab:abelian_vectors_H}
\end{table}

\end{document}